\documentclass{imsart}

\RequirePackage{amsthm,amsmath,amsfonts,amssymb}
\RequirePackage[numbers]{natbib}
\RequirePackage[colorlinks,citecolor=blue,urlcolor=blue]{hyperref}
\RequirePackage{graphicx}

\usepackage{mathrsfs}
\usepackage{mathptmx,amssymb}
\usepackage{mathtools}
\usepackage{amsmath}
\usepackage{color}
\definecolor{BLUE}{RGB}{41,86,143}
\definecolor{RED}{RGB}{178,31,53}
\usepackage{graphicx}

\usepackage{amssymb}             
\usepackage{accents}
\DeclareSymbolFontAlphabet{\amsmathbb}{AMSb}

\allowdisplaybreaks

\newcommand{\refappendix}[1]{\hyperref[#1]{Appendix}}

\startlocaldefs

\newcommand{\bb}[1]{\mathbb{#1}}
\newcommand{\E}{\ensuremath{ \bb{E} } }
\newcommand{\Q}{\ensuremath{ \bb{Q} } }
\newcommand{\bo}[1]{\ensuremath{{\bf #1 } }}

\newcommand{\re}{\ensuremath{\mathbb{R}}}
\newcommand{\paren}[1]{\ensuremath{\left( #1\right) }}
\newcommand{\F}{\ensuremath{\mathscr{F}}}

\newcommand{\p}{\mathbb{P}}

\newcommand{\z}{\ensuremath{\mathbb{Z}}}
\newcommand{\na}{\ensuremath{\mathbb{N}}}

\newcommand{\mc}[1]{\ensuremath{\mathscr{#1}}}
\newcommand{\G}{\ensuremath{\mc{G}}}
\newcommand{\floor}[1]{\ensuremath{\lfloor #1\rfloor}}

\theoremstyle{plain}
\newtheorem{teo}{Theorem}
\newtheorem{lemma}{Lemma}
\newtheorem{coro}{Corollary}
\newtheorem{propo}{Proposition}

\theoremstyle{definition}

\newtheorem{remark}{Remark}

\newcounter{subsubsubsection}[subsubsection]
\renewcommand{\thesubsubsubsection}{\thesubsubsection.\arabic{subsubsubsection}}
\newcommand{\subsubsubsection}[1]{%
  \refstepcounter{subsubsubsection}%
  \paragraph*{\thesubsubsubsection\quad #1}%
}

\usepackage{bm}

\endlocaldefs

\begin{document}

\begin{frontmatter}
\title{Uniform sampling of multitype continuous-time Bienaym\'e-Galton-Watson trees}
\runtitle{Uniform sampling of MBGW trees}

\begin{aug}
\author[A]{\fnms{Osvaldo}~\snm{Angtuncio Hern\'andez}\ead[label=e1]{osvaldo.angtuncio@cimat.mx}\orcid{0009-0008-9599-623X}},
\author[B]{\fnms{Simon C.}~\snm{Harris}\ead[label=e2]{simon.harris@auckland.ac.nz}}
\and
\author[C]{\fnms{Juan Carlos}~\snm{Pardo}\ead[label=e3]{jcpardo@cimat.mx}}

\address[A]{Departamento de Probabilidad y Estad\'istica, Centro de Investigaci\'on en Matem\'aticas\printead[presep={,\ }]{e1,e3}}

\address[B]{University of Auckland, New Zealand\printead[presep={,\ }]{e2}}

\address[C]{Departamento de Probabilidad y Estad\'istica, Centro de Investigaci\'on en Matem\'aticas}

\end{aug}

\begin{abstract}
We study the genealogy of a sample of $k$ individuals taken uniformly without replacement from a continuous-time multitype Bienaym\'e--Galton--Watson process at  fixed times. Our results are quite general, requiring only that the process be non-simple and conservative, and that every type has a positive probability to  ``eventually lead to'' all other types within the population.  The corresponding single-type case has recently been studied by Johnston (2019),
Harris, Johnston, and Roberts (2020),
and  Harris, Johnston, and Pardo (2024).
Our approach is based on a $k$-spine decomposition and a suitable change of measure under which the distinguished spines form a uniform sample at time $T$, while the population size is subject to $k$-size biasing and exponential discounting. This construction preserves a branching Markov property and yields an explicit description of the genealogical tree at fixed times.

In particular, we characterise spine splitting times, offspring distributions, and type-dependent ancestral structures, revealing rich interactions between types that are absent in the single-type setting. The present results form the basis of a forthcoming series of papers in which limiting genealogical behaviour is analysed under various asymptotic regimes and more general sampling schemes by the authors, see Angtuncio et al. (2026b), (2026c) and (2026d).
\end{abstract}



\end{frontmatter}


\section{Introduction}

Continuous-time multitype Bienaym\'e-Galton-Watson (MBGW)  processes arise as a natural generalisation of Bienaym\'e-Galton-Watson  (BGW) processes, in which individuals are differentiated by types that determine their offspring distribution. They were first studied by Kolmogorov and his coauthors  in the discrete time setting during the 1940s. We refer to Athreya and Ney \cite{MR2047480}  and  Sewastjanow \cite{MR0408019} for good introductions to these processes. It turns out that their analysis is considerably eased under an irreducibility assumption, namely, that every type has a positive probability to ``eventually lead to'' all other types within the population (i.e. each other type can be produced after some number of generations with positive probability). Under this hypothesis, together with  the finiteness of the mean offspring matrix, one can use variants of the Perron-Frobenius theorem to describe the asymptotic behaviour of the mean matrix, thereby obtaining both qualitative and quantitative results for the original process. Informally, the large-scale behaviour of irreducible MBGW processes resembles that of single-type BGW processes whose mean offspring distribution is given by the Perron eigenvalue of the mean matrix. On the other hand, Janson \cite{MR2226887} has shown that when the irreducibility assumption fails, a wide variety of behaviours may arise in the discrete setting.

In this manuscript, we will be particularly interested in the genealogy of a sample of $k> 1$ particles taken uniformly without replacement, from a population alive at fixed times for general continuous-time MBGW process.

Define $\mathbb{N}:=\{1,2,3,\dots\}$, $\mathbb{Z}_+:=\{0\}\cup\mathbb{N}$ and let $d\in \mathbb{N}$ and $\alpha_1,\dots,\alpha_d>0$.   We consider a continuous-time MBGW process with $d$ types, as follows.
Start initially with one individual  (or particle) of type $m\in\{1,2,\dots,d\}$. 
All individuals currently alive evolve independently of one another, with each individual's behaviour depending on its type.
Any individual currently alive which is of type $m\in\{1,\ldots, d\}$ will die at rate $\alpha_m$, that is, each individual of type $m$ has an independent exponentially distributed lifetime of rate $\alpha_m$.
At the end of its lifetime, an individual of type $m$ is replaced by a random number of offspring across the $d$ types, according to an independent realisation of the random vector $\mathbf{L}_m=(L^{(1)}_m, \ldots, L^{(d)}_m)\in\mathbb{Z}_+^d$, 
where  $p_m({\bm \ell}):=\mathbb{P}(\mathbf{L}_m={\bm\ell})$ for $\bm{\ell}\in \mathbb{Z}^d_+$,
and
\[
\sum_{{\bm \ell}\in \mathbb{Z}^d_+}p_m({\bm \ell})=1, \qquad\textrm{for all }\,\, m\in\{1,\ldots, d\}.
\]
In other words, when an individual of type $m$ dies, $p_m({\bm \ell})$ is the probability that it will be replaced by $\ell_1$ offspring of type 1, $\ell_2$ offspring of type 2, and so on, for each ${\bm \ell}=(\ell_1,\dots,\ell_d)\in\mathbb{Z}_+^d$.
Once born, these offspring all evolve independently  following a similar law to their parent but according to their own type, as described above, and so on.

For each time $t\geq0$, we let $\mathbf{Z}_t:=(Z_t^{(1)}, \ldots, Z_t^{(d)})$ where $Z_t^{(m)}$ represents the  total number of type $m$ individuals alive at time $t$. 
We  also define $N_t:=\sum_{m=1}^d Z^{(m)}_t$ which represents  the total number of individuals alive at time $t>0$.
The process $\mathbf{Z}$ is said to satisfy the \emph{branching property}, meaning that the law of the process $\mathbf{Z}$  starting from $\mathbf{x}+\mathbf{y}\in\mathbb{Z}_+^d$ coincides with that of the sum of two independent copies of  $\mathbf{Z}$ started from $\mathbf{x}$ and $\mathbf{y},$ respectively.
The process $\mathbf{Z}=(\mathbf{Z}_t)_{t\ge 0}$ 
is commonly referred to as the \emph{MBGW process}. However, as we wish to consider the genealogies of individuals in the population, we will require an enriched process which includes information about how individuals are related to one another. 
We will do this by assigning a unique label to each individual according to the very convenient \emph{Ulam-Harris} convention.  For this construction, we  assume that the process $\mathbf{Z}$ is \emph{conservative}, that is, it  does not explode  almost surely. A sufficient condition for conservativeness of $\mathbf{Z}$ is  given by Savits \cite{MR282426}, namely
\begin{equation}\label{savits}
\int^1\frac{{\rm d}s}{s-\overline{F}(s)}<\infty, \qquad \textrm{with}\qquad \overline{F}(s)=\max_{m\in \{1,\ldots, d\}}\sum_{n\ge 2} \mathbb{P}\left( \sum_{i=1}^d L^{(i)}_m=n\right) s^{n}.
\end{equation}
See for instance Proposition 2.12 and Remark 2.13 in the aforementioned reference.

The Ulam-Harris labelling encodes the genealogical structure of a single-type family tree as follows. 
The initial ancestor at the root of a tree is labelled by $\emptyset$ 
(this convention is particularly convenient when considering subtrees within a larger tree).
Let $\mathbb{U}:=\cup_{n=0}^{\infty}\mathbb{N}^n$ denote the set of all finite sequences of positive integers, with the convention $\mathbb{N}^0:=\{\emptyset\}$.
Whenever an individual with label $u\in\mathbb{N}^n$ dies and is replaced by $\ell$ offspring, these offspring are labeled consecutively as $u1, u2, \dots, u\ell\in\mathbb{N}^{n+1}$ according to their order of appearance, 
where each  label is formed by concatenating the parent's label with the unique index of the offspring amongst its siblings. This recursive scheme extends naturally to the entire genealogical tree. 
We use $|u|$ to represent the generation of the individual labelled $u$, that is,  $|u|=n$ whenever $u\in\mathbb{N}^n$.
For example, an individual with label $u=(2,3,5)$ would refer to the 5th child of the 3rd child of the 2nd child of the initial ancestor, and $|u|=3$ means $u$ is in the 3rd generation since the initial ancestor.
We also write $u\prec v$ if $u$ is a strict ancestor of $v$, and $u\preceq v$ if $u$ is an ancestor of $v$ or $u=v$.

For each $t\ge 0$, we let $\mathcal{N}_{t}$ be the set of Ulam-Harris labels of all individuals alive at time $t$.
Then, with a single initial ancestor, we have $\mathcal{N}_{0}=\{\emptyset\}$. 
For any individual $u\in\mathcal{N}_{t}$, let its type be represented by $Z^u_t$. 
Then $N_t={\rm card}\{\mathcal{N}_t\}$, 
and $Z_t^{(m)}=\sum_{u\in\mathcal{N}_{t}} \mathbf{1}_{\{Z^u_t=m\}}$. Thus, the collection $((u,Z^u_t) : u\in \mathcal{N}_t)$ contains information about each individual's type at time $t$ as well as its genealogical history up until time $t$. 
As this information will now be sufficient for our needs, we will let $(\mathcal{F}_t)_{t\ge 0}$ be the 
right-continuous filtration generated by the enriched process $((u,Z^u_t) : u\in \mathcal{N}_t)_{t\geq 0}$,
that is, $\mathcal{F}_t:=\sigma( (u,Z^u_s)_{u\in \mathcal{N}_s} : s\leq t)$.
We assume that the filtrations $(\mathcal{F}_t)_{t\ge 0}$ satisfy the usual hypotheses, and let $\mathcal{F}:=\sigma(\cup_{t\ge 0}\mathcal{F}_t )$.

For each $m\in\{1,\dots,d\}$, we let $\mathbb{P}_{m}$ be the law of our MBGW process, including its genealogical information, starting from a single initial ancestor of type $m$ with label $\emptyset$. 
These probability measures are defined on the filtered probability space $(\Omega, \mathcal{F}, (\mathcal{F}_t)_{t\ge 0})$.

Later, we will want to start with more than one initial ancestor, specifying a given initial number $z_m$ of each type $m$.
Unless explicitly stated otherwise, we typically won't be concerned with the precise labelling of the initial individuals when starting with more than one individual. In such situations, by default we will simply assign the labels $1,\dots,n_0$ at random when $n_0:=\sum_{m=1}^d z_m>1$ (otherwise, we use label $\emptyset$ when $n_0=1$).
Then, for $\mathbf{z}=(z_1,\dots,z_d)\in \mathbb{Z}_+^d$, we will denote by $\mathbb{P}_{\mathbf{z}}$  the law of the MBGW process, including its genealogical information, starting with $z_m$ individuals of type $m$ for $m=1,\dots,d$. 
Note, letting $(\mathbf{e}_m)_{m\in[d]}$ denote the canonical basis in $\mathbb{R}^d$, that is, $\mathbf{e}_m$ is a vector in $\mathbb{R}^d$ with value 1 in its $m$-th coordinate and $0$ in all others, 
we also have $\mathbb{P}_{\mathbf{e}_m}=\mathbb{P}_{m}$, in agreement with before.

We call  $\mathbf{L}:=(\mathbf{L}_1,\dots,\mathbf{L}_d)$ the \emph{offspring random variable}, 
and ${\bo p}:=({\bo p}({\bm \ell}))_{{\bm \ell}\in \mathbb{Z}^d_+}$
 the \emph{offspring distribution of the MBGW}, where
 \[
{\bo p}({\bm \ell}):=(p_1({\bm \ell}), \ldots, p_d({\bm \ell}))\in [0,1]^d.
\]
Note, for any two vectors ${\bm r},{\bm \ell}\in \mathbb{Z}_+^d$, we often make use of the product notation $\bo{r^{\bm \ell}}:=r_1^{\ell_1}\cdots r_d^{\ell_d}$.
For $\bo{r}\in [0,1]^d$, we denote by $\bo{f} (\bo r):=\paren{f_1(\bo{r}),\ldots, f_d(\bo{r})}\in [0,1]^d$ the probability generating function associated with the offspring  distribution ${\bo p}$, where
for every component $i\in \{1,\ldots,d\}$
\[
f_i(\bo{r}):=\E_i\left[\prod_{m=1}^dr_m^{L^{(m)}}\right]=\sum_{{\bm \ell}\in \mathbb{Z}^d_+}p_i(\bm \ell )\bo{r^{\bm \ell}}=\sum_{{\bm \ell}\in \mathbb{Z}^d_+}p_i\paren{\ell_1,\ldots, \ell_d}r_1^{\ell_1}\cdots r_d^{\ell_d}.
\]

For simplicity of exposition, we write $[d]:=\{1,\ldots,d\}$. Recall that a MBGW branching process is called simple if its generating function $\bo{f}$ is such that for all $m\in [d]$, $f_m$ is linear
in each coordinate with no constant term, i.e.
\[
f_m(\bo{r})=p_m(\bo e_1)r_1+\cdots+  p_m(\bo e_d)r_d, \quad \textrm{for}\quad \bo{r}\in [0,1]^d.
\]
In other words, each individual has exactly one offspring possibly of different type and thus the process has a constant number of individuals.

For our purposes, we further require  that each type has a positive probability of producing offspring of every other type.  We express this assumption in terms of the so-called {\it mean matrix} of $\bo Z$. More precisely, we define the mean matrix  $\bo{M}:=(m_{ij})_{i,j\in [d]}$, where $m_{ij}$ denotes the expected number of type $j$ offspring produced by an individual of type $i$, namely
\begin{equation}\label{meanmatrix}
m_{ij}:=\E_{i}\left[L^{(j)}\right]=\sum_{\bm \ell\in \mathbb{Z}^d_+} \ell_jp_i(\bm{\ell}).
\end{equation}
We say that the process $\bo Z$ is {\it irreducible} if its mean matrix $\bo M$ satisfies $m_{ij}^{(n)}>0$, for some $n$ and for all $i,j\in [d]$, where $m^{(n)}_{ij}$ is the $(i,j)$-th entry of the matrix ${\bo M}^n$.

 From now on, we assume that 
 \begin{equation}\label{hyp1}
 \tag{\bf{H}} \textrm{$\bo Z$ is non-simple,  conservative, and irreducible.}
  \end{equation}
  The mean matrix  plays a central role in the analysis of the  long-term behaviour of the process when all its entries are finite. In this case
\[
m_{ij}=\frac{\partial f_i}{\partial r_j}\paren{\vec{1}}<\infty,
\]
where $\vec{1}:=(1, \ldots, 1)\in \mathbb{Z}_+^d$. The above  condition  is sufficient to ensure that the process $\bo Z$ is conservative. If in addition, $\bo Z$ is irreducible,  the matrix $\bo{C}:=\textrm{diag}(\bm \alpha)(\bo M-\bo I)$ is well defined and  irreducible, where $\bo I$ denotes the identity matrix. An extension of the Perron-Frobenius theorem (see for instance, Theorem 2.5 in Seneta \cite{MR389944} ) implies that $\bo{C}$ admits a real eigenvalue $\rho$ strictly larger than the real part of any other eigenvalue.  The eigenvalue $\rho$ plays the role of the Malthusian parameter in the single-type case, leading  to  the classical classification of branching processes, i.e. the MBGW  process $\bo{Z}$ with law $\mathbb{P}_{\bo z}$ is {\it subcritical, critical} or {\it supercritical} according  as $\rho<0$, $\rho=0$ or $\rho>0$. \\

We emphasise that all subsequent results rely solely on Assumption \eqref{hyp1}, without requiring finiteness  of the mean matrix.\\

In the single-type case, continuous-time BGW trees are endowed with a natural notion of genealogy. Indeed,  each particle living at some time $t$ had a unique ancestor particle living at each earlier time $s < t$. It is then natural to ask questions about the shared genealogy of different particles  in the population alive at a certain time. 
Specifically, conditioning on the event  that there are at least $k$ particles alive at a time $T > 0$, consider picking $k$ 
particles uniformly at random without replacement
from the population alive at time $T$. 
Label these $k$ sampled particles with the integers $1$ through $k$.  

Recalling some standard terminology,  a {\it block} is a subset $B\subseteq \mathbb{N}$. Hereafter, the block formed by the $k$ first integers $[k]=\{1,\dots,k\}$ will play a special role. 
 A {\it partition} of the block $B\subseteq \mathbb{N}$ is a countable collection 
 $A=\{A_i, i\in \mathbb{N}\}$ 
 of pairwise disjoint blocks  such that $\cup_{i\in \mathbb{N}}A_i=B$.

We may associate with the sample of $k$ labelled particles a stochastic process $\pi^{(k,T)} := (\pi^{(k,T)}_t)_{t \in [0,T]}$ taking values in the collection of set partitions of $[k]$, also known as ancestral process,  by declaring:
\[
\text{$i$ and $j$ in the same block of $\pi^{(k,T)}_t$} \iff \text{$i$ and $j$ are descended from the same time $t$ ancestor},
\]
or in other words, that the time $T$ particle labelled with $i \in [k]$ and the time $T$ particle labelled with $j \in [k]$ share the same unique ancestor in the time $t$ population. 
This ancestral process construction is also seen in, for example, Bertoin and Le Gall \cite{MR1771663} and Johnston \cite{MR4003147}.

Since the entire process begins with a single particle at time $0$, it follows that each of the $k$ particles share the same 
initial ancestor, i.e. $\pi^{(k,T)}_0 = \{[k]\}$, Conversely, since we choose uniformly without replacement, each of the particles are distinct at time $T$, hence $\pi^{(k,T)}_T = \{\{1\},\ldots,\{k\} \}$ is the partition of $[k]$ into singletons. More generally, as $t$ increases across $[0,T]$, the stochastic process $\pi^{(k,T)}$ takes a range of values in the partitions of $[k]$, with the property that the constituent blocks of the process break apart as time passes. With this picture in mind, we define the \emph{split times} 
\begin{align*}
\tau_1 < \cdots < \tau_m
\end{align*}
to be the times of discontinuity of $\pi^{(k,T)}$. That is, at each time $\tau_i$, a block of $\pi^{(k,T)}_{\tau_i -}$ breaks into several smaller blocks in $\pi^{(k,T)}_{\tau_i}$.  We note that $\pi^{(k,T)} $ is almost-surely right continuous.

Recently, substantial progress has been made in the  study of  the process $\pi^{(k,T)}$ for various classes of continuous-time BGW trees, see for instance \cite{MR4003147,MR4133376,MR4674065,MR4718398}. Related developments for discrete-time  BGW trees in varying environments have been obtained by  Boenkost et al. \cite{boenkost2024genealogynearlycriticalbranching} and Harris et al. \cite{MR4840489}. In particular, Harris, Johnston and Roberts \cite{MR4133376} analysed  the  asymptotic behaviour of  $\pi^{(k,T)}$, as $T\to \infty$, in the so-called 
near critical regime, where the mean of the offspring distribution converges to one while the variance remains finite. This framework includes the classical  critical case, in which the offspring distribution $(p_i)_{i\ge 0}$ satisfies  $\sum_{i \geq 0 } ip_i = 1$ and  $\sum_{i \geq 0} i(i-1)p_i<\infty$.

Importantly, the same limiting genealogy in the critical regime (with finite variance)
 also arises for BGW trees in varying environments, as shown in
\cite{boenkost2024genealogynearlycriticalbranching} and \cite{MR4840489}. Moreover, the limiting genealogies of subcritical and supercritical  continuous time BGW trees with finite variance offspring distribution are investigated in \cite{MR4003147}.  Finally, the finite-variance assumption is not necessary for the existence of a limiting genealogy in the critical setting. Indeed, for continuous-time BGW trees with heavy-tailed offspring distributions, where multiple mergers appear in the limiting genealogy, we refer the reader to  \cite{MR4718398}.

\subsection{Main result} Our aim is to  study the multitype analogue of the ancestral process $\pi^{(k,T)}$ defined above  for fixed $T$, under the sole standing assumption that the associated multitype BGW process  satisfies \eqref{hyp1}. The results established here constitute a foundational step forward for  a forthcoming analysis of the limiting genealogy of the ancestral process in regimes where the mean matrix $\bo{M}$ is finite  and the process is  subcritical, critical, or supercritical.

Our broader objective is to develop a comprehensive description of the genealogical limits across all these regimes. In the critical case with finite-variance (see \cite{AHP-p2})  we  go beyond the uniform sampling scheme and also consider two additional \emph{sampling mechanisms} that depend explicitly on the type configuration. These extensions highlight the flexibility of the multitype framework and illustrate how type-dependent effects influence the resulting genealogical structures.

The main results of the present paper will serve as key inputs for future work on the subcritical and supercritical regimes (see \cite{AHP-p3}). A separate study will address the critical regime with heavy-tailed offspring distributions, where novel genealogical phenomena, such as multiple merger structures, are expected to emerge (see \cite{AHP-p4}).\\

To state our main results, we introduce some notation and the notion of integer partitions that incorporate type information, which we refer to as  {\it coloured partitions}. 
 A {coloured partition} is  a partition whose blocks are endowed with colours (or types). 
 Formally, a coloured partition of $[k]$ is a collection $\bo{P}=\{P_1,\ldots, P_d\}$,  where for each type $m\in [d]$, $P_m$ consists of a family of disjoint blocks  $A_{m,1}, A_{m,2},\ldots,A_{m,g_m}$, each color 
 $m$ and with respective   sizes $a_{m,1},\ldots, a_{m,g_m}$. Here $ g_m\ge 0$ denotes the number of blocks (or groups) of colour $m$. These blocks form a partition of $[k]$, so that
\[
\sum_{m\in [d]} \sum_{q\in[g_m]} a_{m,q}=k.
\]
For a block $A_{m,q}$ in  $P_m$, any two elements $h_1,h_2\in A_{m,q}$ satisfy the equivalence relation
\[
\textrm{ $h_1\sim h_2$ if the marks $h_1$ and $h_2$ follow the same individual of type $m$. }
 \]
Thus,  the total number of elements of $[k]$ belonging to blocks of type $m$ is 
\begin{equation}\label{prom}
\overline{a}_m:=a_{m,1}+\cdots +a_{m,g_m}=\sum_{q\in[g_m]} a_{m,q}.
\end{equation}
 Figure \ref{figNotationFirstSplittingTimePartition} illustrates an  explicit example of a coloured partition embedded in a MBGW tree with marks.

\begin{figure}
\includegraphics[width=.6\textwidth]{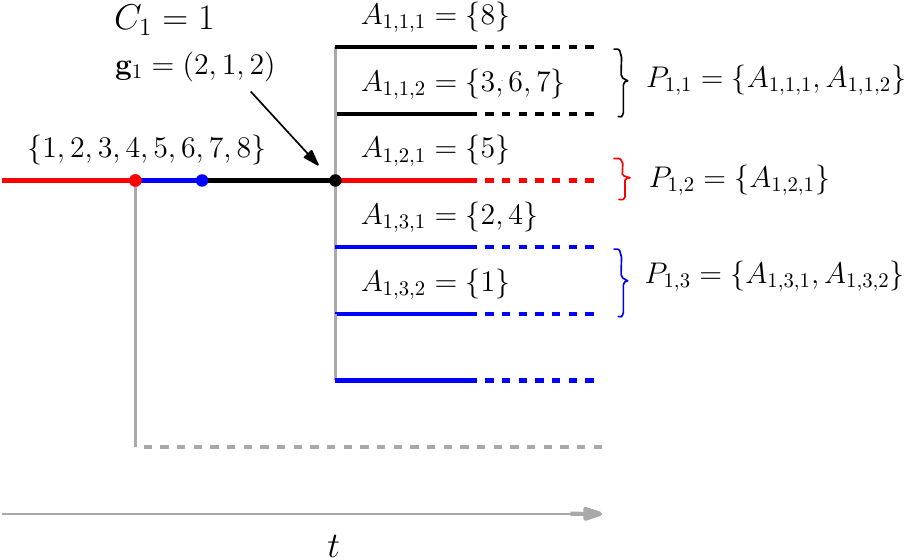}
\caption{Example of a coloured partition after a splitting event with $k=8$ marks, and three types $1,2,3$ represented by colours Black, Red, Blue, respectively. 
The vertex  carrying all marks  at time $t-$, has type 1 (Black) and offspring types $\bm \ell=(2,1,3)$ (that is, 2 Black, 1 Red, and 3 Blue). 
After the splitting event, the coloured partition formed is $\bo{P}=(P_1,P_2,P_3)$. 
In this case there are $g_1=2$ blocks of type 1, namely  $A_{1,1}=\{8\}$ and $A_{1,2}=\{3,6,7\}$, 
with sizes $a_{1,1}=1$ and $a_{1,2}=3$, respectively. Moreover, there are $g_2=1$ blocks of type 2, namely  $A_{2,1}=\{5\}$, of size $a_{2,1}=1$.
There are $g_3=2$ blocks of type 3, which are $A_{3,1}=\{2,4\}$ and $A_{3,2}=\{1\}$ of respective sizes $a_{3,1}=2$ and $a_{3,2}=1$. Finally, the number of marks following type 1, 2 or 3 individuals are $\overline a_1=4,\ \overline a_2=1$ and $\overline a_3=3$. 
}\label{figNotationFirstSplittingTimePartition}
\end{figure}

As in the single-type BGW case, we wish to understand the genealogical structure induced by a uniform sample at a fixed time $T$. We sample $k$ particles from the population alive at time $T$ according to the following procedure. Conditionally on the event  $\{N_T\ge k\}$, where $N_T=\sum_{m=1}^d Z^{(m)}_T$, we select a given sample of $k$ distinct individuals with probability 
\[
\frac{1}{N_T(N_T-1)\cdots(N_T-k+1)}.
\]

 Label the sampled particles with the integers $1,\ldots,k$.  To this sample we associate  a stochastic process taking values in the space of coloured partitions of $[k]$,  
 \[
 \pi^{(d,k,  T, u)} := (\pi^{(d,k,T, u)}_t)_{t \in [0,T]}.
 \]
  This ancestral process is defined as follows:
  \begin{center}
$i$ and $j$ are in the same block of $\pi^{(d,k, T, u)}_t$ with colour $m$ if and only if they descend from the same ancestor of type $m$ at time $t$.
\end{center}
Equivalently, the particles labelled $i$ and $j$ at time $T$ share a unique common ancestor of type  $m$ at time $t$.

Similarly to the single-type case,  all $k$ sampled particles share the same initial ancestor. If the root of the MBGW tree  has type $r\in [d]$, the initial state of the process is 
\[
\pi^{(d,k,T, u)}_0 =\overline{[k]}_{r}:=\{\emptyset,\ldots,\emptyset, [k],\emptyset, \ldots, \emptyset\}
\]
where the block $[k]$ is in the position corresponding to  type $r$,  and all other entries are empty sets, $\emptyset:=\{\}$. This represents a single block with a unique colour corresponding to the common ancestor at time zero.

 Since sampling is performed without replacement, the terminal state of the process is the discrete coloured partition into singletons, 
 $\pi^{(d,k,T, u)}_T =\{S_m, m\in [d]\}$, where $S_{m}$ denotes the collection of singleton blocks corresponding to particles of type $m$.
  As time evolves over  $[0,T]$, the process $\pi^{(d,k,T, u)}$ moves through   the space of coloured partitions of $[k]$. Blocks may either split into two or more (possibly differently coloured) sub-blocks, or change colour without splitting.

Let $M$ denote the number of splitting events required to decompose  the initial block $[k]$ into singletons, and let
$$
0=\tau_0<\tau_1 < \cdots < \tau_M,
$$
be the corresponding splitting times. These are the times at which the process  $\pi^{(d,k,T, u)}$ experiences  a discontinuity due to an actual block splitting.

In contrast to the single-type setting, the multitype framework introduces an additional source of discontinuity: colour changes within blocks. That is, even when the underlying partition structure remains unchanged, a change in the type of one or more elements results in a discontinuity of the coloured partition process. The process  $\pi^{(d,k, T, u)}$ is therefore almost surely right-continuous, with jumps arising from either block splittings or colour changes.

To state the main result of this paper, we introduce additional notation.   The {\it topology} $\mathcal{T}$  of $\pi^{(d,k, T, u)}$ is defined as the sequence 
\[
\mathcal{T}:=(\mathcal{T}_0,\cdots, \mathcal{T}_{M}) \qquad \textrm{with}\qquad \mathcal{T}_{h}=\pi^{(d,k, T, u)}_{\tau_h}.
\]
Intuitively, one may view this construction as encoding a multitype tree with edge lengths and $k$ marked leaves. Each branch undergoes  colour changes, recorded through the colours of the blocks in  
 $\pi^{(d,k, T, u)}$.  For each $h = 0, \ldots, M$, the partition $\mathcal{T}_h$ captures both the block structure and the type of each individual alive at time $\tau_h$. For each $h\in [M]$ and $m\in [d]$, let  $G_{h,m}$ denote the number of new blocks of  type $m$ created at time $\tau_h$, and write  $\bo G_{h}=(G_{h,1},\ldots, G_{h,d})$.  Thus, the  sequence of coloured partitions $\{\mathcal{T}_h\}_{h = 0}^{M}$ satisfies the following properties:
\begin{itemize}
\item $\mathcal{T}_0$ is the trivial coloured partition consisting of a single block,
\item $\mathcal{T}_{M}$ is the discrete coloured partition consisting of singletons, and
\item for each $h = 0, \ldots, M - 1$, the coloured partition $\mathcal{T}_{h+1}$ 
consists of 
\begin{itemize}
\item the coloured subpartitions of $\mathcal{T}_h$ that did not split (updated with their current colours at time $\tau_{h+1}$), and
\item the newly created coloured subpartitions, determined by the vector $\bo G_{h+1}$, at time $\tau_{h+1}$.
\end{itemize}
\end{itemize}

 From $\{\mathcal{T}_h\}_{h = 0}^{M}$,  we extract the {\it ancestral coloured subsequence}
 \[
 \mathcal{P}:=({\mathcal{P}}_0, \ldots, \mathcal{P}_M),
 \]
 which contains only those blocks whose size changes at each splitting event. We set $\mathcal{P}_0=\mathcal{T}_0$ and for $h\ge 1$, $\mathcal{P}_h$ consists precisely of the newly created coloured subpartition at time $\tau_h$.  See Figure \ref{figNotationColouredPartitionSplittingPartitionProcessV2} for an illustration.
	\begin{figure}
\includegraphics[width=.6\textwidth]{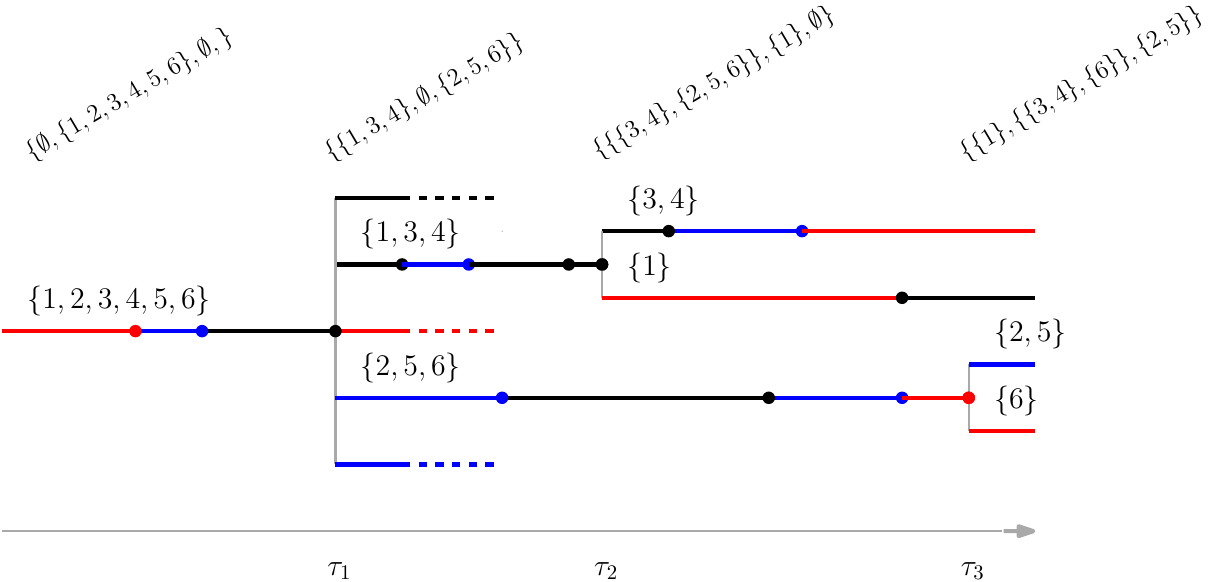}\caption{3-type MBGW tree with 6 spines starting with one individual of  type 2 and its coloured partition process. Individuals of type 1 are depicted with color Black, type 2 with color Red, and type 3 with color Blue. 
At time zero, the coloured partition process takes the value $\overline{[6]}_2$, since the unique partition $[6]$ follows an individual type 2. 
We depict this in the three lines above at time 0. 
Just after time $\tau_1$, the coloured partition process takes the value $\mathcal{T}_{1}=\{\{1,3,4\}\},\emptyset,\{2,5,6\}\}$, since the partition $\{1,3,4\}$ follows a type one individual, and the partition $\{2,5,6\}$ follows a type 3.
Finally, just after time $\tau_2$, the coloured partition process takes the value $\mathcal{T}_{2}=\{P_{2,1},P_{2,2},\emptyset\}$ where $P_{2,1}=\{\{1\} \}$ and $P_{2,2}=\{\{2,5,6\},\{3\},\{4\} \}$. 
Note that in this case, the ancestral coloured subsequence is $\mathcal{P}_0=\{\emptyset,[6],\emptyset\}$, ${\mathcal{P}}_1=(\{1,3,4\},\emptyset,\{2,5,6\})$, ${\mathcal{P}}_2=\{\{3,4\},\{1\},\emptyset\}$, and $\mathcal{P}_3=\{\emptyset,\{6\},\{2,5\}\}$. 
 }\label{figNotationColouredPartitionSplittingPartitionProcessV2}	
		\end{figure}

The tree topology of a splitting coloured sequence is obtained by discarding colour information. That is, from  $(\mathcal{T}_0,\cdots, \mathcal{T}_{M})$ we derive the corresponding sequence of uncoloured partitions $(\Xi_0, \ldots, \Xi_{M})$ which encodes the hierarchical splitting structure independently of the types.

We now turn to the main objective of this work, namely the description of the joint distribution of all spine splitting events occurring up to a fixed time horizon $T$. Our goal is to characterise, in a unified framework, the full collection of random objects generated by these events: the splitting times, the types of the individuals involved, the offspring configurations produced at each splitting, and the evolution of the induced ancestral coloured subsequence of the label set $[k]$. In particular, we keep track of how the labels of a uniform sample of size  $k$  are redistributed among descendants through successive spine splittings.

To make this precise, we fix $n\leq k-1$ and  assume that exactly $n$ spine splitting events occur before time $T$, that is,   $M=n$ and $0<\tau_1<\cdots <\tau_n<T$. At each splitting time $\tau_h$, $h\in [n]$, a single individual on the spine -- denoted by the label $v(h)$ -- gives birth to $\bo L_{v(h)}$ new offspring. We denote by  $C_{\tau_h}$  the type (or colour) of the spine individual $v(h)$ involved in the $h$-th splitting event. 
This reproduction event induces a redistribution of a subset of the $k$ sampled marks among the offspring, which we encode by $\mathcal P_{\tau_h}$. 
 
 The sequence of spine splitting events thus generates an ancestral coloured subsequence $\mathcal P=(\mathcal P_{\tau_h})_{h\in [n]}$ of $[k]$, which records the genealogical evolution of the sampled lineages along the spine. For notational convenience, we write $ C_h:=C_{\tau_h}$ and $\mathcal P_h:=\mathcal P_{\tau_h}$. Each element $\mathcal P_h$ takes values of the form 
 \[
 \bo P_h=(P_{h,1},\ldots, P_{h,d}),
 \]
  where   $P_{h,m}$ corresponds to offspring of type $m\in[d]$. More precisely $P_{h,m}=\{A_{h,m,q}\}_{q\in [G_{h,m}] }$, 
is a family of $G_{h,m}$ disjoint blocks, each block representing a group of marks that follow the same descendant of type $m$.  This coloured partition therefore simultaneously encodes the offspring structure at the splitting time and the induced redistribution of the sampled lineages.

 We now formalise the event of interest.  Fix $n\leq k-1$,  the times $0<t_1<t_2<\cdots <t_n$, and for each $h\in [n]$ and offspring configuration
 \[
 {\bo g}_{h}:=(g_{h,1},\ldots, g_{h,d})\leq \bm \ell_{h}:=(\ell_{h,1},\ldots, \ell_{h,d})\in \mathbb{Z}^d_+,
 \]
 together with the partition  $\bo P_h$ with block counts ${\bo g}_{h}$ and a type $i_h\in [d]$. We  define the event
\begin{equation}\label{deltatn}
\Delta_T(n):=\bigcap_{h\in [n]}\left\{\tau_h\in {\rm d}t_h,\mathcal{P}_{h}=\bo{P}_h,\bo{L}_{v(h)}=\bm{\ell}_h , C_{h}=i_h, M=n \right\}. 
\end{equation}

Moreover, if $A_{h,m,q}$ is a block of $\mathcal{P}_h$, we denote by  $k_{v(h, m,q)}:=\textrm{card}\{A_{h,m,q}\}$ the number of marks associated with the descendant of $v(h)$ corresponding to the block $A_{h,m,q}$ and having type $m$. We denote this descendant  by $v(h,m,q)$.
We refer to Figure \ref{figNotationJointSplittingTimes} for an illustration of this notation.

\begin{center}
	\begin{figure}
		\includegraphics[width=.4\textwidth]{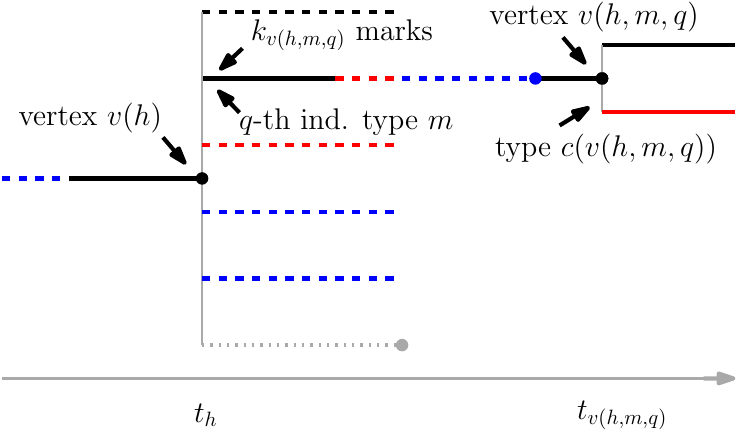}\caption{We show a small window of a multitype tree around a vertex $v(h)$ involved in a spine splitting event at time $t_h$. 
Such a vertex has $L_{v(h)}$ offspring. We represent in gray vertices being born that do not carry marks after time $t_h$. We follow the subtree generated by the $q$-th vertex type $m$ carrying $k_{v(h,m,q)}$ marks (in the picture $q=2$ and $m=1$). Such  vertex, at time $t_{v(h,m,q)}$ is denoted by $v(h,m,q)$, has type $c(v(h,m,q))$ and undergoes a spine splitting event.
}\label{figNotationJointSplittingTimes}
	\end{figure}
\end{center} 

Recall that  sampling $k$ individuals uniformly without replacement at time $T$, conditionally on the event $\{N_T\ge k\}$
, induces a probability measure that we denote by $ \mathbb{P}^{(k)}_{unif,T,r}$, as  noted for instance in \cite{MR4718398, MR4133376}. This measure acts on   measurable functionals of the genealogies of $k$-tuples of particles.

Let $f$ be a functional of the (Ulam-Harris labelling of) the ancestors of the $k$ particles, their birth times,
death times, types, and the number of offspring they have upon death.
Let $\bm{\varsigma}_T=(\varsigma^{(1)}_T, \ldots,\varsigma^{(k)}_T )$ denote a uniform sample without replacement at time $T$ taken from a MBGW process started from a single  individual of type $r$, conditioned on the event $\{N_T\ge k\}$. We define  the probability measure $ \mathbb{P}^{(k)}_{unif,T,r}$ on $\{N_T\ge k\}$ as follows
\begin{equation}\label{probuniIntro}
 \mathbb{E}^{(k)}_{unif,T,r}\left[f( \bm{\varsigma}_T)\right]=\mathbb{E}_r\left[\frac{1}{N_T(N_T-1)\cdots(N_T-k+1)}\sum_{\bo{v}\in \mathcal{N}^{(k)}_T} f(\bo{v})\Bigg| N_T\ge k\right]
\end{equation}
where $\mathcal{N}_{T}^{(k)}$ is the set of all possible $k$-tuples of particles which are alive at time $T$. The first term of the right hand side of  \eqref{probuniIntro} is the probability for any given choice of distinct $\bo{v}\in \mathcal{N}^{(k)}_T.$ 

In Section \ref{defipunif} we provide a formal definition of $ \mathbb{P}^{(k)}_{unif,T,r}$ and show that it is a uniform choice without replacement from all particles alive at time $T$. Under this law, the following theorem provides an explicit expression for the joint probability of observing a prescribed spine splitting history, describing the complete genealogy of the sample.

For later use, we introduce the  notation $\delta_{i, m}=1$ if $m=i$, and 0 otherwise.

\begin{teo}\label{teofsplitting1}
For any $k\geq 1$, $T\in \re_+$,  $r\in [d]$, we have
\begin{equation}\label{eqnfsplitting2}
	\begin{split}
\mathbb{P}^{(k)}_{unif,T,r}&\paren{\Delta_T(n)}\\
&= \frac{1}{(k-1)!}\frac{1}{\mathbb{P}_{r}\paren{\ N_T\geq k}}\int_0^\infty(e^\phi-1)^{k-1}\Q^{(k),\phi\vec{1}}_{T,r}(\Delta_T(n))\E_{r}\left[N^{\floor{k}}_Te^{-\phi\vec{1}\cdot \bo Z_T}\right]{\rm d}\phi,
\end{split}
\end{equation}
where
\[
\begin{split}
\Q^{(k),\phi\vec{1}}_{T,r}(\Delta_T(n))&= \prod_{h=1}^n\prod_{m=1}^d\E_{m}\left[e^{-\phi\vec{1} \cdot \bo Z_{T-t_h}}\right]^{\ell_{h,m} -g_{h,m}}p_{i_h}(\bm \ell_{h})\bm \ell_{h}^{\floor{\bo g_h}}\\
		&\hspace{-1.5cm}  \times 
		\prod_{h=0}^{n-1}\prod_{m=1}^d\prod_{\substack{ q\leq g_{h,m}:\\ k_{v(h,m,q)}\geq 2}}\E_{m}\left[Z^{(c(v(h,m,q)))}_{t_{v(h,m,q)}-t_h} \prod_{j=1}^d\E_{j}\left[e^{-\phi\vec{1}\cdot \bo Z_{T-t_{v(h,m,q)}}}\right]^{ Z^{(j)}_{t_{v(h,m,q)}-t_h}-\delta_{c(v(h,m,q)),j}}\right]\\
	& \hspace{1.5cm} \times
	\frac{\prod_{h=1}^{n}\prod_{m=1}^{d}\E_{m}\left[N_{T-t_{h}}e^{-\phi\vec{1} \cdot \bo Z_{T-t_{h}}}\right]
		^{\#\{ q\leq g_{h,m}:k_{v(h,m,q)}=1\}}
	}{\E_{r}\left[N_T(N_T-1)\cdots(N_T-k+1)e^{-\phi\vec{1}\cdot \bo Z_{T}}\right]} \prod_{h=1}^n\alpha_{i_{h}} {\rm d}t_h,
	\end{split}
\]
 $t_{v(h,m,q)}$ and $c(v(h,m,q))$ are the spine splitting times  and type associated to the vertex $v(h,m,q)$.  When  $h=0$, we observe that $v(0)=\emptyset$, $g_{0,m}=\delta_{r,m}$ and $v(0, r, 1)=v(1)$. 
\end{teo}

This theorem provides a key structural ingredient for deriving the limiting genealogies analysed in \cite{AHP-p2,AHP-p3,AHP-p4} across the different regimes of the process. In particular, the results established above allow us, in \cite{AHP-p2}, to relate the additional sampling mechanisms introduced there to the uniform sampling scheme described in Theorem~\ref{teofsplitting1}. More precisely, we establish an explicit connection showing how these type-dependent sampling schemes can be represented in terms of the uniform sampling scheme. This relationship plays a central role in identifying the limiting genealogical structures in the critical regime with finite variance.

\subsection{Outline of proof}

To conclude, we briefly outline the strategy used to establish our main result. 
The central idea is to introduce a collection of distinguished lineages, or \emph{spines}, evolving within a continuous-time MBGW tree started from a single individual of type~$r$. 
Building on, and extending, the spine techniques developed by Harris et al.~\cite{MR4133376}, we construct a change of measure $\Q^{(k),\bm{\theta}}_{T,r}$ under which the spines are biased so that, at time~$T$, they form a uniform sample of $k$ distinct individuals from the population.

At the same time, the population is reweighted by the size-biased functional
\[
\bo z \longmapsto n(n-1)\cdots(n-k+1)\, \mathrm e^{-\bm{\theta}\cdot \bo z},
\]
where $\bo z=(n_1,\ldots,n_d)$ denotes the vector of type counts with total size 
$n=\sum_{m=1}^d n_m$. 
In the setting of Harris et al.~\cite{MR4133376}, this corresponds to the special case in which the vector parameter $\bm{\theta}$ reduces to the scalar value $\theta=0$.

The vector $\bm{\theta}$ plays the role of an exponential discounting parameter, regulating the growth of the tree and allowing the process to be interpreted as sampling from a $k$-fold size-biased multitype population, even in the absence of higher-order moment assumptions. 
Related exponential discounting techniques were developed for single type trees in~\cite{MR4718398}.

Under the change of measure $\Q^{(k),\bm{\theta}}_{T,r}$, the model admits a significantly simplified and more tractable description. 
While the formal definition of this measure is technical and deferred to~\eqref{defMeasureQkTrGivenAllInfo}, an intuitive understanding of the resulting dynamics is essential for the proof of Theorem~\ref{teofsplitting1}.

Specifically, under $\Q^{(k),\bm{\theta}}_{T,r}$, the Ulam--Harris labelled population $\mathcal N$ is augmented with $k$ distinguished spines. 
The underlying branching dynamics are otherwise unchanged: particles evolve exactly as under $\mathbb P_r$, except that some individuals may carry one or more spines, while particles without spines behave as in the original MBGW process. 
The spines are constrained to be distinct at time~$T$ and, crucially, are distributed at that time as a uniform sample without replacement from the population alive at~$T$. 
Moreover, conditional on the genealogical structure of the tree up to time~$T$, the spines evolve independently. 
This conditional independence is the key structural feature that permits explicit computations under $\Q^{(k),\bm{\theta}}_{T,r}$.

Theorem~\ref{teofsplitting1} then provides a precise mechanism for transferring these computations back to the original measure $\mathbb P_r$, conditioned on the event $\{N_T\ge k\}$. 
In this way, the change of measure serves as a powerful analytical tool, yielding tractable expressions while preserving the genealogical information relevant to uniform sampling.

The remainder of the paper is devoted to the proof of the main result. 
We begin by introducing in Section \ref{subsectionChangeOfMEasure} the multiple spines framework  and the associated changes of measure that underpin our approach. 
Within this framework, in Section \ref{subsectionPropertiesOfQkUnderForwardConstruction} we establish the key structural properties of the spines under the measure $\Q^{(k),\bm{\theta}}_{T,r}$, including a forward construction (see Proposition \ref{forwconstr}) of the multitype branching tree. 
Building on these results, we derive in Sections \ref{subsectionJointLawSpineSplittingTimes} and \ref{subsectionJointLawOfSplittingTimesUnderPUnif} the joint law of the spine splitting times at fixed horizons under both the probability measure $\Q^{(k),\bm{\theta}}_{T,r}$ and the one induced by applying the sampling scheme to the original process. 
Together, these results provide the core probabilistic ingredients required for the analysis.

\subsection{Related works on genealogies of multitype branching processes}

Genealogical questions for multitype branching processes are considerably less developed than in the single-type setting. Nevertheless, a number of works address multitype genealogical structures under various assumptions, most often in large-time regimes and typically under conditioning on non-extinction.

An early contribution is due to Jagers and Nerman~\cite{MR1459475}, who analysed the genealogy of a single sampled individual in a multitype Markov branching process. 
By studying the backward evolution of ancestral types, they identified a governing Markov chain, known as the line of descent types, thereby laying the foundations for type-based genealogical analysis.

A more systematic treatment of genealogies in MBGW processes is due to J.-Y. Hong and coauthors; see, for example, \cite{MR2942128,MR3306444,MR3317481,MR3570095,MR3749363}. In Hong's thesis~\cite{MR2942128},  a general framework based on simple random sampling without replacement is introduced, in which ancestral lineages are traced backward in time until their coalescence. Within this setting, the distribution of the most recent common ancestor (MRCA), including its generation, type, and death time, is analysed for both discrete and continuous-time models and across all regimes. The subcritical and critical cases are further investigated in \cite{MR3570095}. In the critical case, under a   finite variance assumption,  a point process associated with the genealogical structure of the sample is studied, together with the distribution of the generation of the MRCA. In the subcritical case, assuming an $X\log X$-moment condition explicit expressions are obtained for the MRCA and for the joint law of its type and generation, as well as the types of the sampled individuals. Corresponding results for the supercritical case, also under an $X\log X$-moment condition, are obtained in~\cite{MR3317481}. Collectively, these works provide a detailed description of the MRCA for finite samples, resolving both its generation and its type.

Another related contribution is due to Popovic and Rivas~\cite{MR3264441}, who consider branching processes with infinitely many types. When the offspring distribution is linear fractional, they derive explicit laws  for coalescence times of pairs of individuals, including same-type coalescences, and propose an algorithmic construction of the ancestral tree of the standing population under quasi-stationarity via an associated Markov chain.

The work of Foutel-Rodier and Schertzer~\cite{MR4664584} is particularly close in spirit to the present paper. They study the asymptotic genealogy of finite samples taken at large times by means of a many-to-few formula and method of moments, within a general critical branching Markov framework that includes MBGW processes. Their results establish convergence toward coalescent-type limits described by the marked Brownian coalescing point process. A central contribution is the introduction of spinal probability measures under which $k$ individuals are sampled uniformly without replacement at large times, together with a size-biased change of measure that preserves the branching Markov property. 

We emphasise that their approach assumes  the existence of moments of order $k$ together with a suitable notion  of criticality, implemented through a harmonic Doob $h$-transform. Using the associated spinal decomposition, they characterise the pairwise distances between splitting times and the types of the sampled individuals. In contrast, we focus on MBGW processes and obtain explicit descriptions at fixed times $T$, without requiring  the process to be critical. This perspective allows us to capture the full structure of multiple mergers which is not visible under their approach. In particular, our results provide a complete description of the  genealogical structure of the spines, including splitting times, offspring configurations, and types.

\section{Spines and changes of measures}\label{Spines}

Throughout, we use $\mathbb{R}:= [0,\infty)$ and adopt the standard Ulam-Harris labelling system to encode the genealogical structure of particles.  We  recall that $\mathbf{Z}=(\mathbf{Z}_t)_{t\ge 0}$ is a continuous time $\mathbb{Z}^d_+$-valued BGW branching process with  probabilities $(\mathbb{P}_{\mathbf{x}})_{\mathbf{x}\in \mathbb{Z}_+^d}$ on  the filtered probability space $(\Omega, \mathcal{F}, (\mathcal{F}_t)_{t\ge 0})$.

\subsection{Change of measures}\label{subsectionChangeOfMEasure}

For any fixed $k\in \mathbb{N}$ and $r\in[d]$, we now define a measure $\p^{(k)}_{r}$ under which  the population process $\mathcal{N}=(\mathcal{N}_t)_{t\ge 0}$, with an initial ancestor of type $r$, has $k$ distinguished lines of descent, known as {\it spines}, similarly as in Harris et al. \cite{MR4133376}. The measure $\p^{(k)}_{r}$ will serve as a natural and convenient {\it reference measure} when looking at the behaviour of other population processes with $k$ distinguished particles.

We denote the $k$ spines   by  $\bm{\varsigma}=(\varsigma^{(1)},\ldots, \varsigma^{(k)})$, where $\varsigma^{(i)}$ corresponds to the distinguished line of descent of the $i$-th \emph{spine}. 
Each spine is represented by a sequence of Ulam-Harris labels $v_0 v_1 v_2 \dots$ which start at the initial ancestor, and where the next label in the sequence is always an offspring of the previous, i.e.\ $v_0=\emptyset$ and, for each $i\in\mathbb{Z}_+$, $v_{i+1}=v_i \ell$ for some $\ell\in\{1,\dots,\bo 1\cdot \bo L_{v_i}\}$,  where $\bo{L}_{u}=(L^{(1)}_u, \ldots, L^{(d)}_u)$ denotes the offspring of particle $u$ which has  distribution $\bo{p}$.  A spine may be an infinite line of descent, or a finite path which terminates at a leaf in the underlying genealogical tree of the population. 
If a particle $u$ has $j$ distinct spines passing though it (i.e. ${\rm card}\{i\in [k]: u\in\varsigma^{(i)}\}=j$), then we say that the particle $u$  is \emph{carrying $j$ spines}.

The process $\mathcal{N}$ with $k$ spines $\bm{\varsigma}$ under the measure $\mathbb{P}^{(k)}_r$ is constructed as an extension of $\mathbf{Z}$ under $\mathbb{P}_r$, in that, all particles behave exactly as in the original MBGW process but some particles are additionally identified as carrying the spines, as follows:

\begin{enumerate}
	\item Begin with one particle type $r$ carrying $k$ marks $\{1,\ldots, k\}$.
	\item Each mark follows a \emph{spine}.  
	\item A type $i$ particle carrying $q$ marks $m_1<\cdots <m_q$ at time $t$ branches at rate $\alpha_i>0$, dying and being replaced by a random number of particles $\mathbf{L}=(L^{(1)}, \ldots, L^{(d)})$ according to the probability 
	$p_i(\bm{\ell})=\mathbb{P}_i(\mathbf{L}=\bm{\ell})$, for $\bm{\ell}\in \mathbb{Z}^d_+$, $i\in [d]$,
	independently of the  rest of the system, just as under $\p_r$.
	\item Given that $\bm{\ell}=(\ell_1,\ldots,\ell_d)$ particles are born at a branching event, with $\ell_m$ particles of type $m$, say $w_1,\ldots, w_{\ell_m}$, each of the $q$ marks independently chooses to follow type $m$ with probability $\ell_m\xi_m /\bm{\ell}\cdot \bm{\xi}$, where $\bm{\xi}=(\xi_1,\ldots,\xi_d)$ is a probability distribution; conditional on this choice, it then  follows particle $w_{h}$ with probability $1/\ell_m$ for $h\in [\ell_m]$. 
\end{enumerate}

Under the assumption that the mean matrix has finite entries, the canonical choice for the probability distribution $\bm \xi$ is the normalised right eigenvector associated with the Perron root, that is, the real eigenvalue that is strictly larger than the real part of any other eigenvalue, as guaranteed by the Perron-Frobenius theorem. This will be the natural choice in \cite{AHP-p2}.

In Figure \ref{figTreeUnderP8}, we show an example of a MBGW tree under $\p^{(8)}_2$. 
\begin{center}
		\begin{figure}
			\includegraphics[width=.6\textwidth]{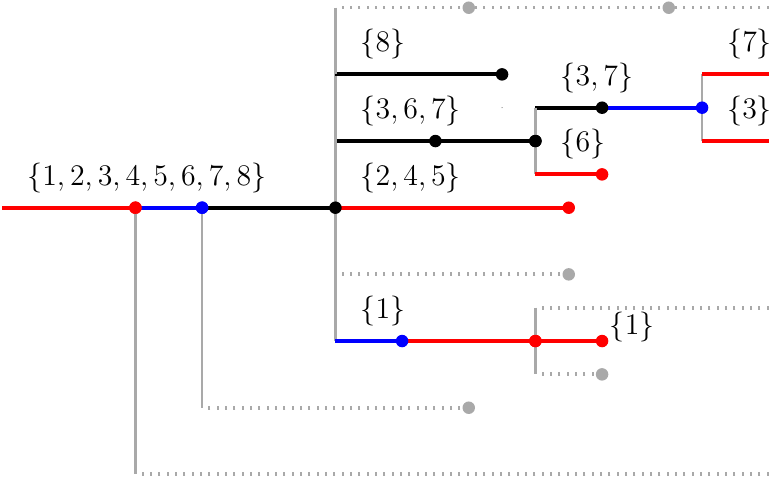}\caption{3-type MBGW tree under $\p^{(8)}_2$, where individuals type 1 are depicted with color Black, type 2 with color Red, and type 3 with color Blue. The time of death of a particle is represented by a dot of its color. The dotted lines represent those particles that carry no marks. }\label{figTreeUnderP8}	
		\end{figure}
\end{center}

Let $\F^{(k)}_t$ contain all the information about the system up to time $t$, including the information about the $k$ spines. We write $\bm{\varsigma}_t=(\varsigma_t^{(1)},\ldots, \varsigma_t^{(k)})$ to identify the  $k$ spines at time $t$, where $\varsigma_t^{(h)}$ is the label of the particle carrying spine $h$ at time $t$. For $h\in [k]$, we define 
\[
\textrm{spine}(\varsigma_t^{(h)})=\Big((u^{(h)}_1,c^{(h)}_{2}),(u^{(h)}_2,c^{(h)}_{3}),\ldots, (u^{(h)}_{m-1},c^{(h)}_{m}),(u^{(h)}_{m},c^{(h)}_{m+1})\Big)
\]
where $m$ is such that $\varsigma_t^{(h)}=u^{(h)}_m$,   be the spine generated by $\varsigma_t^{(h)}$ together with the color of the vertex that the $h$-th mark decided to follow.
That is, assume that the $g$-th branching event occurs by the individual $u^{(h)}_g$, which is an ancestor of $\varsigma_t^{(h)}$. 
After this branching event, the children $v$ of $u^{(h)}_g$ has type $c^{(h)}_{g+1}$ and is the ancestor of $\varsigma_t^{(h)}$ (we assume there are $m\geq g$ such branching events in the lifetime of the spine of $\varsigma_t^{(h)}$).  It is important to note that in the  pair $(u^{(h)}_g,c^{(h)}_{g+1})$, the particle and the color  are in different generations in the ancestral linage of $\varsigma_t^{(h)}$.

Assume that each vertex $u^{(h)}_{g}$ has children $\{\bo{L}_{h,g}=(\ell^{(1)}_{h,g},\ldots, \ell^{(d)}_{h,g})\}$. 
By definition, for each label $v$ we have
\begin{align*}
	\p^{(k)}_r\paren{\left.\varsigma_t^{(h)}=v\right|\ \F_t}&=\frac{\ell^{(c^{(h)}_{2})}_{h,1}\xi_{c^{(h)}_{2}}}{\bo{L}_{h,1}\cdot \bm{\xi}}\frac{1}{\ell^{(c^{(h)}_{2})}_{h,1}}\frac{\ell^{(c^{(h)}_{3})}_{h,2}\xi_{c^{(h)}_{3}}}{\bo{L}_{h,2}\cdot \bm \xi}\frac{1}{\ell^{(c^{(h)}_{3})}_{h,2}}\cdots \frac{\ell^{(c^{(h)}_{m})}_{h,m}\xi_{c^{(h)}_{m+1}}}{\bo{L}_{h,m}\cdot \bm \xi }\frac{1}{\ell^{(c^{(h)}_{m+1})}_{h,m}} = \prod_{g=1}^{m}\frac{\xi_{c^{(h)}_{g+1}}}{ \bo{L}_{h,g}\cdot\bm \xi},
	\end{align*}
where $\bm \xi $ is  the probability distribution introduced  in the definition of $\p^{(k)}_r$ above.

Observe that the first term in the right-hand side represents the way we  choose the color $c^{(h)}_{2}$ from the  $\bo{L}_{h,1}$ children of $u^{(h)}_{1}$ and the next term represents a uniform choice from the  $\ell^{(c^{(h)}_{2})}_{h,1}$ offsprings of type $c^{(h)}_{2}$ and so on. We also observe that in the previous identity the knowledge of the types of each vertex $u^{(h)}_g$ are not required. Indeed, the spine may change types between branching events along its path.

Another way to write the previous identity  is the following
\begin{equation}\label{identitypk}
\p^{(k)}_r\paren{\left.\varsigma_t^{(h)}=v\right|\ \F_t}=\prod_{(w,c_w)\in \textrm{spine}(v)}\frac{\xi_{c_{w}}}{\bo L_{w}\cdot \bm \xi },
\end{equation}
where $c_w$ represents the color of the offspring of $w$ that the mark follows. Later on, we will use the notation $c_w$ for  the color of the offspring of $w$ and $c(w)$ for the color of $w$.

For any integer $n,j\in \na$, define 
\begin{equation}\label{decfac}
n^{\floor{j}}:=n(n-1)\cdots (n-j+1),
\end{equation}
 the decreasing factorial with the convention that $n^{\floor{j}}=0$ if $n<j$ and $n^{\floor{0}}=1$.  Thus \[
 N_{t}^{\floor{k}}=N_{t}(N_{t}-1)\cdots (N_{t}-k+1)\qquad\textrm{ if }\quad N_{t}\geq k
 \]
 and $N_{t}^{\floor{k}}=0$, otherwise. Recall that  $\mathcal{N}_{t}^{(k)}$ denotes the set of all possible $k$-tuples of particles which are alive at time $t$ and  observe that $N_{t}^{\floor{k}}$ is precisely the cardinality of $\mathcal{N}_{t}^{(k)}$.

 \subsubsection{Uniform sampling}\label{subsubsectionUniformSampling}
For our purposes, we introduce
\begin{equation}\label{eqnDefinitionOfG_ktAndZeta_kt}
\begin{split}
g_{k,t}:
=\sum_{\bo{v}\in \mathcal{N}_{t}^{(k)}}\mathbf{1}_{\{\bo{\varsigma}_t=\bo{v}\}}\prod_{h\in [k]}\prod_{(w,c_w)\in \mathrm{spine}(v^{(h)})}\frac{ \bo{L}_{w}\cdot \bm{\xi}}{\xi_{c_{w}}},
\end{split}
\end{equation}
with the convention that $g_{0,t}=1$; and for $\bm{\theta}\in \mathbb{R}_+^{d}$,
\begin{equation}\label{cebolla}
 \zeta^{\bm{\theta}}_{k,t}:=\frac{g_{k,t}e^{-\bm \theta\cdot \bo Z_{t}}}{\E^{(k)}_r\left[N^{\floor{k}}_{t}e^{-\bm \theta\cdot \bo Z_{t}}\right]}.
\end{equation}
 Then, it is clear that, for $t>0$,
\[
\begin{split}
 \E^{(k)}_r\left[\left. g_{k,t} \right|\ \F_t\right]&= \sum_{\bo{v}\in \mathcal{N}_{t}^{(k)}}\prod_{h\in [k]}\prod_{(w,c_w)\in \mathrm{spine}(v^{(h)})}\frac{ \bo{L}_{w}\cdot \bm{\xi}}{\xi_{c_{w}}}\p^{(k)}_r\paren{\left.\bo{\varsigma}_t=\bo{v} \right|\ \F_t}\\
&={\rm card}\{\mathcal{N}_{t}^{(k)}\}=N^{\floor{k}}_{t},
\end{split}
\] 
and thus
\begin{equation}\label{tomate}
 \E^{(k)}_r\left[\left.  g_{k,t}e^{-\bm \theta\cdot \bo Z_{t}} \right|\ \F_t\right]
= N^{\floor{k}}_{t}e^{-\bm \theta\cdot \bo Z_{t}}.
\end{equation}
 
The latter suggest that we can construct, for $t>0$,  the  following probability measure 
\begin{equation}\label{defMeasureQkTrGivenAllInfo}
	\left.\frac{{\rm d}\Q^{(k),\bm{\theta}}_{t,r}}{{\rm d}\p^{(k)}_r}\right|_{\F^{(k)}_t}:=\zeta^{\bm{\theta}}_{k,t}.
\end{equation}
The measure $\Q^{(k),\bm{\theta}}_{t,r}$ will be very useful for our purposes. Since $\E^{(k)}_r\left[\left.g_{k,t} \right|\ \F_t\right]= N^{\floor{k}}_{t}$, it follows 
\begin{equation}\label{defMeasureQkTrGivenTopologicalInfo}
\left.\frac{{\rm d}\Q^{(k),\bm{\theta}}_{t,r}}{{\rm d}\p^{(k)}_r}\right|_{\F_t}=\frac{N^{\floor{k}}_{t}e^{-\bm \theta\cdot \bo Z_{t}}}{\E^{(k)}_r\left[N^{\floor{k}}_{t}e^{-\bm \theta\cdot \bo Z_{t}}\right]}.
\end{equation}

We also introduce a change of measure $\p^{\bm \theta}_{t,r}$ for $\bm \theta\in \re^d_+$ and $t\in \re_+$, corresponding to the discounted MBGW process. It is defined by
\begin{equation}\label{eqnChangeOfMeasureMtypeBGWDiscounted}
\left.\frac{{\rm d}\p^{\bm \theta}_{t,r}}{{\rm d}\p_r}\right|_{\F_t}=\frac{e^{-\bm \theta\cdot \bo Z_{t}}}{\E_r\left[e^{-\bm \theta\cdot \bo Z_{t}}\right]}.
\end{equation}
This change of measure will naturally arise in our analysis of individuals  without  marks.

We emphasise that the law of the underlying MBGW tree is the same under $\mathbb{P}^{(k)}_{\mathbf{z}}$ as under $\mathbb{P}_{\mathbf{z}}$, for $\mathbf{z}\in \mathbb{Z}_+^d$, that is, for all $t\ge 0$,
\[
\mathbb{P}^{(k)}_{\mathbf{z}}=\mathbb{P}_{\mathbf{z}}\qquad \textrm{on } \mathcal{F}_t.
\] 
This fact will be used without comment later on.

\subsubsubsection{\it The probability measure \texorpdfstring{ $\p^{(k)}_{unif,t,r}$}{TEXT}}\label{defipunif}

Let us recall briefly Lemma 13 in \cite{MR4133376} which will be useful in the sequel.

\begin{lemma}\label{lemHJR} Let $\mu$ and $\nu$ be two probability measures on a $\sigma$-algebra $\F$ and $\G\subseteq \F$ is also a $\sigma$-algebra with Radon-Nikodym derivatives
\[
\left.\frac{{\rm d}\mu}{{\rm d}\nu}\right|_{\F}=:Y\qquad \mbox{ and }\qquad \left.\frac{{\rm d}\mu}{{\rm d}\nu}\right|_{\G}=:Z.
\]
Then for any non-negative random variable $X$, $\F$-measurable, we have
\[
Z\mu[X |\, \G]=\nu[XY|\, \G]\qquad \nu-a.s.
\]
\end{lemma}
 
We will use the previous Lemma with $\F=\F^{(k)}_t$, $\G=\F_t$, $\mu=\Q^{(k),\bm \theta}_{t,r}$, $\nu=\p^{(k)}_r$ and $X$ a non-negative random variable $\F^{(k)}_t$-measurable.
Thus
\[
Y=\frac{g_{k,t}e^{-\bm \theta\cdot \bo Z_{t}}}{\E^{(k)}_r\left[N^{\floor{k}}_{t}e^{-\bm \theta\cdot \bo Z_{t}}\right]}\qquad \mbox{and}\qquad Z=\frac{N^{\floor{k}}_{t}e^{-\bm \theta\cdot \bo Z_{t}}}{\E^{(k)}_r\left[N^{\floor{k}}_{t}e^{-\bm \theta\cdot \bo Z_{t}}\right]},
\] 
implying that
\[
	\frac{N^{\floor{k}}_{t}e^{-\bm \theta\cdot \bo Z_{t}}}{\E^{(k)}_r\left[N^{\floor{k}}_{t}e^{-\bm \theta\cdot \bo Z_{t}}\right]}\Q^{(k),\bm \theta}_{t,r}\Big[ X \Big|\ \F_t\Big]=\E^{(k)}_r\left[\left. X \frac{g_{k,t}e^{-\bm \theta\cdot \bo Z_{t}}}{\E^{(k)}_r\left[N^{\floor{k}}_{t}e^{-\bm \theta\cdot \bo Z_{t}}\right]}\right|\ \F_t\right].
\]
Considering $X=\mathbf{1}_{\{{\bm \varsigma_t}=\bo v\}}$, for $\bo v \in \mathcal{N}^{(k)}_t$, we get
\[
\begin{split}
	\frac{N^{\floor{k}}_{t}e^{-\bm \theta\cdot \bo Z_{t}}}{\E^{(k)}_r\left[N^{\floor{k}}_{t}e^{-\bm \theta\cdot \bo Z_{t}}\right]}\Q^{(k),\bm \theta}_{t,r}\paren{\left.\bm \varsigma_t=\bo v \right|\ \F_t}&=\E^{(k)}_r\left[\left. \mathbf{1}_{\{\bm \varsigma_t=\bo v \}} \frac{g_{k,t}e^{-\bm \theta\cdot \bo Z_{t}}}{\E^{(k)}_r\left[N^{\floor{k}}_{t}e^{-\bm \theta\cdot \bo Z_{t}}\right]}\right|\ \F_t\right]\\
	&=e^{-\bm \theta\cdot \bo Z_{t}}\E^{(k)}_r\left[\left. \mathbf{1}_{\{\bm \varsigma_t=\bo v\} } \frac{g_{k,t}}{\E^{(k)}_r\left[N^{\floor{k}}_{t}e^{-\bm \theta\cdot \bo Z_{t}}\right]}\right|\ \F_t\right].
\end{split}
\]
Recalling the definition of $g_{k,t}$,
we deduce
\begin{equation}\label{eqnQIsUniformGivenF0}
\begin{split}
	\Q^{(k),\bm \theta}_{t,r}\paren{\left.\bm \varsigma_t=\bo v \right|\ \F_t}
	&=\frac{1}{N^{\floor{k}}_{t}}\E^{(k)}_r\left[\left. \mathbf{1}_{\{\bm \varsigma_t=\bo v \}} g_{k,t}\right|\ \F_t\right]\\
	&= \frac{1}{N^{\floor{k}}_{t}}\prod_{h\in [k]}\prod_{(w,c_w)\in {\rm spine}(v^{(h)})}\frac{\bo{L}_w\cdot \bm \xi}{\xi_{c_w}}\p^{(k)}_r\paren{\left.   \varsigma_t^{(h)}= v^{(h)}  \right|\ \F_t}= \frac{1}{N^{\floor{k}}_{t}}.
\end{split}
\end{equation}
In other words,  under the measure $\Q^{(k),\bm \theta}_{t,r}$, the $k$-spines are a uniform choice without replacement from all  particles alive at time $t$. 
The latter also implies a more complete description of $\Q^{(k),\bm \theta}_{t,r}$ from its definition in \eqref{defMeasureQkTrGivenAllInfo}. 
Namely,
\begin{equation}\label{defMeasureQkTrGivingTwoSteps}
	\left.\frac{{\rm d}\Q^{(k),\bm \theta}_{t,r}}{{\rm d}\p^{(k)}_r}\right|_{\F^{(k)}_t}=\frac{N^{\floor{k}}_{t}e^{-\bm \theta\cdot \bo Z_{t}}}{\E^{(k)}_r\left[N^{\floor{k}}_{t}e^{-\bm \theta\cdot \bo Z_{t}}\right]}\frac{1}{N^{\floor{k}}_{t}}g_{k,t},
\end{equation}
which says that first we $k$-size bias and $\bm \theta$-discount the process given $\F_t$, and then we choose $k$ spines uniformly without replacement under $\Q^{(k),\bm \theta}_{t,r}$ given $\F_t$.

We now formally introduce the probability measure  $\p^{(k)}_{unif,t,r}$, which was previously introduced in \eqref{probuniIntro} without detailed exposition. This measure characterises uniform sampling without replacement from a MBGW tree, and we highlight its connection with the probability measures defined earlier. The validity of this interpretation will become apparent once the properties of this measure are established. Let 
\begin{equation}\label{probuni1}
\left.\frac{{\rm d}\p^{(k)}_{unif,t,r}}{{\rm d}\p^{(k)}_r}\right|_{\F^{(k)}_t}:=\frac{1}{\p_r\paren{N_t\geq k}}\frac{g_{k,t}}{N^{\floor{k}}_{t}}. 
\end{equation}
Note that the above display is a probability measure since by \eqref{tomate}, we have
\[
\E^{(k)}_r\left[\frac{g_{k,t}}{N^{\floor{k}}_{t}}\right]=\E^{(k)}_r\left[\frac{\mathbf{1}_{\{N_t\geq k\}}}{N^{\floor{k}}_{t}}\E^{(k)}\left[g_{k,t}\Big| \F_t\right]\right]=\p_r\paren{N_t\geq k}.
\]Also, just as the proof of \eqref{eqnQIsUniformGivenF0}, one can show that if $\bm \varsigma_t$ is a uniform sample without replacement at time $t$ and $\bo v\in \mathcal{N}^{(k)}_t$, then 
\[
\mathbf{1}_{\{N_t\geq k\}}\p^{(k)}_{unif,t,r}\paren{\left.\bm \varsigma_t=\bo v \right|\ \F_t}
	= \mathbf{1}_{\{N_t\geq k\}}\frac{1}{N^{\floor{k}}_{t}}.
\]Similarly, for  a function of $k$ distinct vertices of the tree at time $t$, say $f$,  and $A\in \F_t$; we have
\begin{equation}\label{probuni}
\p^{(k)}_{unif,t,r}(A)=\p_r\paren{A\big| N_t\geq k}\qquad \mbox{and} \qquad \mathbb{E}^{(k)}_{unif,t,r}\left[f( \bm{\varsigma})\right]=\mathbb{E}_r\left[\frac{1}{N^{\floor{k}}_t}\sum_{\bo{v}\in \mathcal{N}^{(k)}_t} f(\bo{v})\Bigg| N_t\ge k\right].
\end{equation}

From \eqref{defMeasureQkTrGivingTwoSteps}, \eqref{probuni1} and \eqref{probuni}, we observe the following relationship under the event $\{N_t\geq k\}$
\begin{equation}\label{defMeasureQkTrandPUnif}
\begin{split}
	\left.\frac{{\rm d}\Q^{(k),\bm \theta}_{t,r}}{{\rm d}\p^{(k)}_{unif,t,r}}\right|_{\F^{(k)}_t}& =\left.\frac{{\rm d}\Q^{(k),\bm \theta}_{t,r}}{{\rm d}\p^{(k)}_r}\right|_{\F^{(k)}_t}\times\left.\frac{{\rm d}\p^{(k)}_r}{{\rm d}\p^{(k)}_{unif,t,r}}\right|_{\F^{(k)}_t}\\
& =\frac{g_{k,t}e^{-\bm \theta\cdot \bo Z_{t}}}{\E^{(k)}_r\left[N^{\floor{k}}_{t}e^{-\bm \theta\cdot \bo Z_{t}}\right]}\p_r\paren{N_t\geq k}\frac{N^{\floor{k}}_t}{g_{k,t}}=\frac{N^{\floor{k}}_{t}e^{-\bm \theta\cdot \bo Z_{t}}}{\E^{(k)}_{unif,t,r}\left[N^{\floor{k}}_{t}e^{-\bm \theta\cdot \bo Z_{t}}\right]}.
\end{split}
\end{equation}
In other words when going from $\p^{(k)}_{unif,t,r}$ to $\Q^{(k),\bm \theta}_{t,r}$, then events in $\F^{(k)}_t$ are only affected through the tree topology (size-biasing and discounting), but the marks (or the uniform sample) is not affected. 
This relationship is useful for the following result. 

\begin{lemma}
\label{lemmaFromTheMeasureQToTheMeasureEunif}
	Suppose $A\in \F^{(k)}_t$. Then
	\[
\p^{(k)}_{unif,t,r}\big( A\big)=\E^{(k)}_r\Big[\left.N^{\floor{k}}_{t}e^{-\bm \theta\cdot \bo Z_{t}}\right|\ N_t\geq k\Big]\Q^{(k),\bm \theta}_{t,r}\left[\frac{\mathbf{1}_{A}}{N^{\floor{k}}_{t}e^{-\bm \theta\cdot \bo Z_{t}}}\right].
		\]
\end{lemma}
\begin{proof}
The proof follows directly from \eqref{defMeasureQkTrandPUnif}, since either under $\p^{(k)}_{unif,t,r}$ or $\Q^{(k),\bm \theta}_{t,r}$, we have $N_t\geq k$. 
\end{proof}

\subsection{Properties under \texorpdfstring{ $\mathbb{Q}^{(k),\bm{\theta}}_{T,r}$}{TEXT} and forward construction}\label{subsectionPropertiesOfQkUnderForwardConstruction}

From now on, we  fix $T\in \re_+$. 
As discussed above,  we have constructed a corresponding auxiliary measure $\mathbb{Q}^{(k),\bm{\theta}}_{T,r}$. This construction is useful because it allows us to analyse certain functionals of the sampling process under the original measure in a more transparent manner. Indeed, under this auxiliary measure, the tree itself evolves as a branching process, which significantly simplifies the analysis. In particular, many functionals of interest admit a simpler and more tractable representation. We will illustrate this idea in detail below.

The next lemma is fundamental for our results and follows similar ideas to those  in Lemma 8 of \cite{MR4133376}. We include its proof for the sake of completeness. In particular, it implies that under $\Q^{(k),\bm \theta}_{T,r}$ the tree also behaves as a branching process, that is, individuals give birth to independent branching processes with marks. However, one must take into account the initial parameters, namely the number of marks, the remaining  and the type of the root, which determine the evolution of the object.

\begin{lemma}[Markov branching property]\label{lemmaMrkovPropertyUnderQ}
Assume $j\in \na\cup\{0\}$ and $k\in\na$ such that $j\leq k$ and $t\in (0,T)$.
Suppose that a vertex $v$ in $\mathcal{N}_{t}$, of type $c(v)$, carries $j$ marks at time $t$. 
Then, under $\Q^{(k),\bm \theta}_{T,r}$, the subtree generated by $v$ after time $t$ is independent of the rest of the system and behaves as if under $\Q^{(j),\bm \theta}_{T-t,c(v)}$. In particular, if $j=0$ then $\Q^{(0),\bm \theta}_{T-t,c(v)}=\p^{\bm \theta}_{T-t,c(v)}$.
\end{lemma}
\begin{proof}
Define $\mathcal{H}:=\mathcal{H}_{v,t}$ the $\sigma$-algebra generated by all the information except in the subtree generated by $v$ after time $t$. This includes the type of $v$. 
Define $\chi_v$ as the  life time of $v$, after which  instantaneously dies and gives birth to $\bo{L}_{v}$ children (we define similarly $\chi_\emptyset$ and $\bo L_\emptyset$ for the life time and number of offspring of the root).
Let $I_{v,t}$ be the set of marks carried by $v$ at time $t$. 
Let $\bo{Z}^{[v]}_{T}$ be the vector of number of individuals alive at time $T$ whose ancestor is $v$, from the original MBGW process  under $\p_r$.   

Our aim is to prove that for any $t'\in (t,T)$ and $\bm \ell\in \z^d_+\setminus\{\bo 0\}$, we have
\[
\Q^{(k),\bm \theta}_{T,r}\paren{\chi_v>t',\bo{L}_v=\bm{\ell}\left.\right|\mathcal{H}}=\Q^{(j),\bm \theta}_{T-t,c(v)}\paren{\chi_\emptyset>t'-t,\bo{L}_\emptyset=\bm{\ell}}.
\]First we decompose $g_{k,T}$ in two terms, one containing only the information of the subtree generated by $v$ and between the interval of time $[t,T]$, that is
\[
g_{v,[t,T]}:=\mathbf{1}_{\{\varsigma^{(m)}_T\neq \varsigma^{(n)}_T,\mbox{ for }\ m\neq n,\mbox{ with }m,n\in I_{v,t}\}}\prod_{h\in I_{v,t}}\prod_{\substack{(w,c_w)\in {\rm spine}(\varsigma^{(h)}_T),\\{\rm gen}(w)\in [t,T]}}\frac{\bo{L}_w\cdot \bm \xi}{\xi _{c_w}},
\]
where ${\rm gen}(v)\in [a,b]$ denotes that the branching event creating $v$ occurs in $[a,b]$.  The second terms contains the remainder of the information of the tree which is $\mathcal{H}$-measurable, i.e.
\begin{equation}\label{descomph}
	\begin{split}
		h_{T}:=&\mathbf{1}_{\{\varsigma^{(m)}_T\neq \varsigma^{(n)}_T,\mbox{ for }\ m\neq n,\mbox{ with }m,n\in I^c_{v,t}\}}\prod_{h\in I^c_{v,t}}\prod_{(w,c_w)\in {\rm spine}(\varsigma^{(h)}_T)}\frac{\bo{L}_w\cdot \bm \xi}{\xi _{c_w}}\\
		& \hspace{6cm}\times \prod_{h\in I_{v,t}}\prod_{\substack{(w,c_w)\in {\rm spine}(\varsigma^{(h)}_T),\\
		{\rm gen}(w)\in[0,t)}}\frac{\bo{L}_w\cdot \bm \xi}{\xi _{c_w}},
	\end{split}
\end{equation}
where $I^c_{v,t}$ is the complement of $I_{v,t}$. Observe that when $v$ has no marks $I_{v,t}$ is the empty set and; 
\[
g_{v,[t,T]}=1\qquad \textrm{ and }\qquad h_{T}=g_{k,T}.
\]

Using Lemma \ref{lemHJR} with  $\F=\F^{(k)}_T$, $\G=\mathcal{H}$, $\mu=\Q^{(k),\bm \theta}_{T,r}$, $\nu=\p^{(k)}_r$,  $X=\mathbf{1}_{\{\chi_v>t',\bo{L}_v=\bm{\ell}\}}$, 
\[
Y=\frac{g_{k,T}e^{-\bm \theta\cdot \bo Z_{T}}}{\E^{(k)}_r\Big[N^{\floor{k}}_Te^{-\bm \theta\cdot \bo Z_{T}}\Big]}\qquad \mbox{and}\qquad Z=\E^{(k)}_r\left[\left.\frac{g_{k,T}e^{-\bm \theta\cdot \bo Z_{T}}}{\E^{(k)}_r\Big[N^{\floor{k}}_Te^{-\bm \theta\cdot \bo Z_{T}}\Big]}\right|\, \mathcal{H}\right],
\]
we thus obtain
\begin{align*}
	&\E^{(k)}_r\left[\left.\frac{g_{k,T}e^{-\bm \theta\cdot \bo Z_{T}}}{\E^{(k)}_r\Big[N^{\floor{k}}_Te^{-\bm \theta\cdot \bo Z_{T}}\Big]}\right|\, \mathcal{H}\right]\Q^{(k),\bm \theta}_{T,r}\paren{\left.\chi_v>t',\bo{L}_v=\bm{\ell} \right|\ \mathcal{H}}\\
	&\hspace{5cm}=\E^{(k)}_r\left[\left. \mathbf{1}_{\{\chi_v>t',\bo{L}_v=\bm{\ell}\}} \frac{g_{k,T}e^{-\bm \theta\cdot \bo Z_{T}}}{\E^{(k)}_r\left[N^{\floor{k}}_Te^{-\bm \theta\cdot \bo Z_{T}}\right]}\right|\ \mathcal{H}\right].
\end{align*}
We observe that in the last expression, we may cancel the denominators on both sides.  Next, we use the decomposition
\[
g_{k,T}e^{-\bm \theta\cdot \bo Z_{T}}=g_{v,[t,T]}e^{-\bm \theta\cdot \bo Z^{[v]}_{T}}\times h_Te^{-\bm \theta\cdot \paren{\bo Z_{T}-\bo Z^{[v]}_{T}}},
\]
to cancel on both sides the $\mathcal{H}$-measurable factor, that is  $h_Te^{-\bm \theta\cdot \paren{\bo Z_{T}-\bo Z^{[v]}_{T}}}$, thus we obtain the following identity
\begin{align*}
	&\Q^{(k),\bm \theta}_{T,r}\paren{\left.\chi_v>t',\bo{L}_v=\bm{\ell} \right|\ \mathcal{H}}=\frac{\E^{(k)}_r\left[\left. \mathbf{1}_{\{\chi_v>t',\bo{L}_v=\bm{\ell}\}}g_{v,[t,T]}e^{-\bm \theta\cdot \bo Z^{[v]}_{T}}\right|\ \mathcal{H}\right]}{\E^{(k)}_r\Big[g_{v,[t,T]}e^{-\bm \theta\cdot \bo Z^{[v]}_{T}}\Big|\, \mathcal{H}\Big]}.
\end{align*}
Since the vertex $v$ is of type $c(v)$ and has $j$ marks at time $t$ by hypothesis, by the Markov branching property under $\p^{(k)}_r$ (which is naturally inherited from its construction since the dynamics of the marks are also Markovian) and \eqref{defMeasureQkTrGivenAllInfo}, we have
\[
\begin{split}
\Q^{(k),\bm \theta}_{T,r}\paren{\left.\chi_v>t',\bo{L}_v=\bm{\ell} \right|\ \mathcal{H}}&=\frac{\E^{(j)}_{c(v)}\left[\mathbf{1}_{\{\chi_\emptyset>t'-t,\bo{L}_\emptyset=\bm{\ell}\}} g_{j,T-t}e^{-\bm \theta\cdot \bo Z_{T-t}}\right]}{\E^{(j)}_{c(v)}\left[g_{j,T-t}e^{-\bm \theta\cdot \bo Z_{T-t}}\right]}\\
&=\frac{\E^{(j)}_{c(v)}\left[\mathbf{1}_{\{\chi_\emptyset>t'-t,\bo{L}_\emptyset=\bm{\ell}\}} g_{j,T-t}e^{-\bm \theta\cdot \bo Z_{T-t}}\right]}{\E^{(j)}_{c(v)}\left[N^{\floor{j}}_{T-t}e^{-\bm \theta\cdot \bo Z_{T-t}}\right]}\\
& = \E^{(j)}_{c(v)}\left[\mathbf{1}_{\{\chi_\emptyset>t'-t,\bo{L}_\emptyset=\bo{\ell}\}} \zeta^{\bm \theta}_{j,T-t}\right]\\
& = \Q^{(j),\bm \theta}_{T-t,c(v)}\paren{\chi_\emptyset>t'-t,\bo{L}_\emptyset=\bm{\ell}},
\end{split}
\]
where in the second equality we have used  \eqref{tomate}, the third equality uses the definition in \eqref{cebolla} and finally in the last equality we use the definition of $\Q^{(j),\bm \theta}_{T,r}$ in \eqref{defMeasureQkTrGivenAllInfo}. This completes  the proof.
\end{proof}
Our goal is to construct the multitype tree under $\Q^{(k),\bm \theta}_{T,r}$, forward in time. This construction is based on a recursive procedure. In the single-type case, a similar approach is detailed in Lemma 2.5 of \cite{MR4718398}. Here, we extend the construction from \cite{MR4718398} to the multitype setting.

Before we do so, we recall the notation used in the definition of coloured partitions above identity \eqref{prom} and introduce for $m\in [d]$, 
\[
d_{m,n}:=\textrm{card}\{q:a_{m,q}=n\}\qquad \textrm{for any }\quad n\geq 1.
\] 
\begin{propo}[Size-biased and discounted Galton-Watson process with $k$ spines under $\Q^{(k),\bm \theta}_{T,r}$]\label{forwconstr}  
Let $T\in \re_+$, $t\in (0,T)$, $k\in \na$, and $r\in [d]$.
Consider the process $\mathcal{N}=(\mathcal N_t)_{t\in [0,T]}$ with $k$ spines $\bm \varsigma_t=( \varsigma^{(1)}_t,\ldots,  \varsigma^{(k)}_t)$ at time $t$, under $\Q^{(k),\bm \theta}_{T,r}$. 
Then, $\mathcal N$ evolves as follows:
\begin{enumerate}
	\item The process starts at time 0 with one particle carrying all $k$ spines (i.e., $\mathcal{N}_0=\{\emptyset\} $, and $\bm\varsigma_0=(\emptyset,\ldots, \emptyset) $).
	\item A particle carrying $h\in [k]$ spines evolves a subtree forward in time independently of the rest of the process (branching Markov property).
	\item A particle type $i\in [d]$ carrying $h\in [k]$ spines at time $t\in (0,T)$, branches into $\bm \ell\in \z^d_+\setminus\{\bo 0\}$ offspring, and the $h$ spines split into $g_m\in \{0,1,\ldots,\ell_m\}$ different individuals (groups) of type $m\in [d]$, the $q$-th carrying $a_{m,q} \in [h]$ marks, with $\sum_{m\in[d]} \overline{a}_{m}=h$ and $q\in [g_m]$, at rate
	\begin{equation}
		\begin{split}
			& \alpha_i\E_i\left[\bo L^{\floor{\bo g}}\prod_{m=1}^d\E_m\left[e^{-\bm \theta \cdot \bo Z_{T-t}}\right]^{L^{(m)}-g_m}\right]
			\\
			& \hspace{1cm} \times\frac{\bm \ell^{\floor{\bo g}}			\prod_{m=1}^d\E_m\left[e^{-\bm \theta \cdot \bo Z_{T-t}}\right]^{\ell_m-g_m}p_i(\bm{\ell})}{\E_i\left[\bo L^{\floor{\bo g}}\prod_{m=1}^d\E_m\left[e^{-\bm \theta \cdot \bo Z_{T-t}}\right]^{L^{(m)}-g_m}\right]}
			\frac{\prod_{\substack{m\in [d]\\ g_m\neq 0}}\prod_{q=1}^{g_m}\E_m\left[N^{\floor{a_{m,q}}}_{T-t}e^{-\bm \theta \cdot \bo Z_{T-t}}\right]}{\E_i\left[N^{\floor{h}}_{T-t}e^{-\bm \theta\cdot \bo Z_{T-t}}\right]}.
		\end{split}
	\end{equation}In particular, a birth off the spine, that is, all marks follow the same vertex type $j\in [d]$, occurs at rate 
	\begin{equation}
		\begin{split}
			& \alpha_i\E_i\left[L^{(j)}\prod_{m=1}^d\E_m\left[e^{-\bm \theta \cdot \bo Z_{T-t}}\right]^{L^{(m)}-\delta_{m,j}}\right]\\
			&\hspace{4cm}\times\frac{\ell_{j}p_i(\bm{\ell})
			\prod_{m=1}^d\E_m\left[e^{-\bm \theta \cdot \bo Z_{T-t}}\right]^{\ell_m}}{\E_i\left[L^{(j)}\prod_{m=1}^d\E_m\left[e^{-\bm \theta \cdot \bo Z_{T-t}}\right]^{L^{(m)}}\right]}
			\frac{\E_{j}\left[N^{\floor{h}}_{T-t}e^{-\bm \theta \cdot \bo Z_{T-t}}\right]}{\E_i\left[N^{\floor{h}}_{T-t}e^{-\bm \theta\cdot \bo Z_{T-t}}\right]}.
		\end{split}
	\end{equation}
	\item 
Given a particle carrying $h$ spines branches into $\bm \ell$ offspring, where the $h$ spines split into $g_m$ groups of type $m$, each of size $a_{m,q}$, the spines are assigned between the offspring as follows:
	\begin{itemize}
	\item Choose $\bo g$ of the $\bm \ell$ offspring to carry the spine groups uniformly amongst the $ \prod_{m=1}^d{\ell_m \choose g_m}$ distinct ways.
\item For each of the $g_m$ offspring chosen to carry the spine groups, assign group sizes $a_{m,1},\ldots, a_{m,g_m}$ amongst the $g_m$ offspring uniformly amongst the $g_m!/\prod_{n\geq 1}d_{m,n}!$ distinct allocations.
\item Partition the $h$ (labelled) spines between the $d$ types with their given group sizes $\overline a_{1},\ldots, \overline a_{d}$ uniformly amongst the $h!/\prod_{m=1}^{d} \overline a_m!$ distinct ways.
\item Partition the $\overline a_m$ (labelled) spines between the $g_m$  individuals carrying marks, using the block sizes $(a_{m,q})_{q\in [g_m]}$, uniformly amongst the $\overline a_m!/\prod_{q=1}^{g_m} a_{m,q}!$ distinct ways.
	\end{itemize}
	\item Finally, any particle $v$ type $i$, which is alive at time $t$ and carries no spines, behaves independently of the remainder of the process and undergoes branching into offspring that carry no spines at rate
\[
\alpha_i\frac{\E_i\left[\prod_{m=1}^d\E_m\left[e^{-\bm \theta\cdot \bo Z_{T-t} }\right]^{ L^{(m)}} \right]}{\E_i\left[e^{-\bm \theta\cdot \bo Z_{T-t} }\right]}{\rm d} t,
\]and given there is a branching event, there are $\bm \ell$ offspring with probability
\[
p_i(\bm \ell)\frac{\prod_{m=1}^d\E_m\left[e^{-\bm \theta\cdot \bo Z_{T-t} }\right]^{\ell_m}}{\E_i\left[\prod_{m=1}^d\E_m\left[e^{-\bm \theta\cdot \bo Z_{T-t} }\right]^{ L^{(m)}} \right]}.
\]

\end{enumerate}
\end{propo}

In order to prove Proposition \ref{forwconstr}, we require some intermediate results. We  present such results in the forthcoming subsection, together with their proofs and at the end we present the proof of  Proposition \ref{forwconstr}. Actually, it will follow from Proposition \ref{propoFirstSplittingEventWithPartitionChildAndType} (which is step (3)), Corollary \ref{coroSplittingEventWithPartitionChildAndType} (which is step (4)) and Lemma \ref{lemmaLawOfIndividualsWithNoMarks} (step (5)).

 \subsection{Proof of the forward construction} Before we deduce Proposition \ref{forwconstr}, we first obtain a decomposition, along the  time interval $[0,T]$, of $g_{k,T}$ and $\bo Z_T$ which will be useful for the forthcoming analysis. These decompositions are essential for computing the probabilities of some branching  events under $\Q^{(k),\bm \theta}_{T,r}$, as both quantities appear in the Radon-Nikodym derivative presented in  \eqref{cebolla}. 
\subsubsection{Decomposing $g_{k,T}$ and $\bo Z_T$}

Recall that $\chi_\emptyset$ denotes the life time of the root, which is also the first \emph{branching} or \emph{birth event}. 
Let us now denote by 
\[
\tau_1:=\inf\Big\{t\geq 0:\ \exists\ i,j\in [k]\mbox{ such that }\varsigma^{(i)}_t\neq \varsigma^{(j)}_t \Big\}
\]the first \emph{spine splitting event}. We recall that  $c(\varsigma^{(1)}_t)$ denotes the type of the vertex $\varsigma^{(1)}_t$ which carries mark one at time $t$.

Let $0\le s\le t$ and $I\subset [0,T]$ be either $(s,t)$ or $[s,t)$. We introduce  the following event
\[
B_{I}:=\{ \textrm{only one branching event occurs in  $I$}\}.
\] Let $\overline{s}=s$ or $s-$ accordingly as $I$ is either $(s,t)$ or $[s,t)$. We also recall that ${\rm gen}(v)\in [a,b]$ denotes that the branching event of particle $v$  occurs in $[a, b]$, for any $v$ in the MBGW tree.

Under the event that $\{\tau_1\in I,\mathcal{P}_{\tau_1}=\bo{P},\bo{L}_{\tau_1}=\bm{\ell} , c(\varsigma^{(1)}_{\overline{s}})=i \}\cap B_{I}$, where $\bo{P}$ is a coloured partition and $\bo{L}_{\tau_1}$ denotes the offspring distribution of the vertex $\varsigma^{(1)}_{\tau_1-}$, we have 
\begin{align*}
	g_{k,T} &= \mathbf{1}_{\{\varsigma^{(h_1)}_{\overline{s}}=\varsigma^{(h_2)}_{\overline{s}}, h_1,h_2\in [k]\}}\prod_{(w,c_w)\in {\rm spine}(\varsigma^{(1)}_{\overline{s}})}\left(\frac{\bo{L}_w\cdot \bm{\xi}}{\xi_{c_{w}}}\right)^k\\
	&\hspace{1cm} \times 
\prod_{\substack{m\in [d]\\ g_m\neq 0}}\paren{\frac{ \bm{\ell}\cdot \bm \xi }{\xi_m}}^{\overline{a}_{m}}\prod_{q=1}^{g_m}\mathbf{1}_{\{\varsigma^{(h_1)}_{T}\neq \varsigma^{(h_2)}_{T},\, h_1\neq h_2,h_1,h_2\in A_{m,q}\}}\prod_{h\in A_{m,q}}\prod_{\text{$(w,c_w)\in {\rm spine}(\varsigma^{(h)}_{T})$}\atop\text{${\rm gen}(w)\in[t, T]$}}\frac{ \bo{L}_{w}\cdot \bm \xi }{\xi _{c_w}},
\end{align*}
where the first term, i.e. 
\[
\overline{g}_{k,[0,\overline{s}]}:=\mathbf{1}_{\{\varsigma^{(h_1)}_{\overline{s}}=\varsigma^{(h_2)}_{\overline{s}}, h_1,h_2\in [k]\}}\prod_{(w,c_w)\in {\rm spine}(\varsigma^{(1)}_{\overline{s}})}\left(\frac{\bo{L}_w\cdot \bm{\xi}}{\xi_{c_{w}}}\right)^k,
\]
arises from the fact that all spines are together before time $\overline{s}$. Observe that when $\overline{s}=s-$, then $[0,\overline{s}]=[0,s)$. The second term, i.e.
	\[
	\widetilde{g}_{k, I}:=\prod_{\substack{m\in [d]\\ g_m\neq 0}}\paren{\frac{ \bm{\ell}\cdot \bm \xi}{\xi_m}}^{\overline{a}_{m}}, 
	\]
	follows from the fact that in $I$ there is a branching or spine splitting  event with the partition $\bo{P}=(P_1,\ldots, P_d)$.  Finally the last term is decomposed on the time interval $[t, T]$ where the spines  are distributed, for each $m\in[d]$, according to the blocks  $(A_{m,q})_{q\in [g_m]}$,  i.e.
	\[
	\widehat{g}_{A_{m,q},[t,T]}:=\mathbf{1}_{\{\varsigma^{(h_1)}_{T}\neq \varsigma^{(h_2)}_{T},\ h_1\neq h_2,h_1,h_2\in A_{m,q}\}}\prod_{h\in A_{m,q}}\prod_{\text{$(w,c_w)\in {\rm spine}(\varsigma^{(h)}_{T})$}\atop\text{${\rm gen}(w)\in[t, T]$}}\frac{ \bo{L}_{w}\cdot \bm \xi }{\xi _{c_w}}.
	\]
	In other words, we have
 \begin{equation}\label{eqnDecompositionOfGLemmaOfPartitions}	g_{k,T}=\overline{g}_{k,[0,\overline{s}]} \widetilde{g}_{k,I}\prod_{\substack{m\in [d]\\g_m\neq 0}}\prod_{q=1}^{g_m}\widehat{g}_{A_{m,q},[t,T]}. 
	\end{equation}

In Figure \ref{figFirstSplittingTimeDecomposingG} we present an example illustrating such a decomposition of $g_{k,T}$. 

\begin{center}
	\begin{figure}
		\includegraphics[width=.4\textwidth]{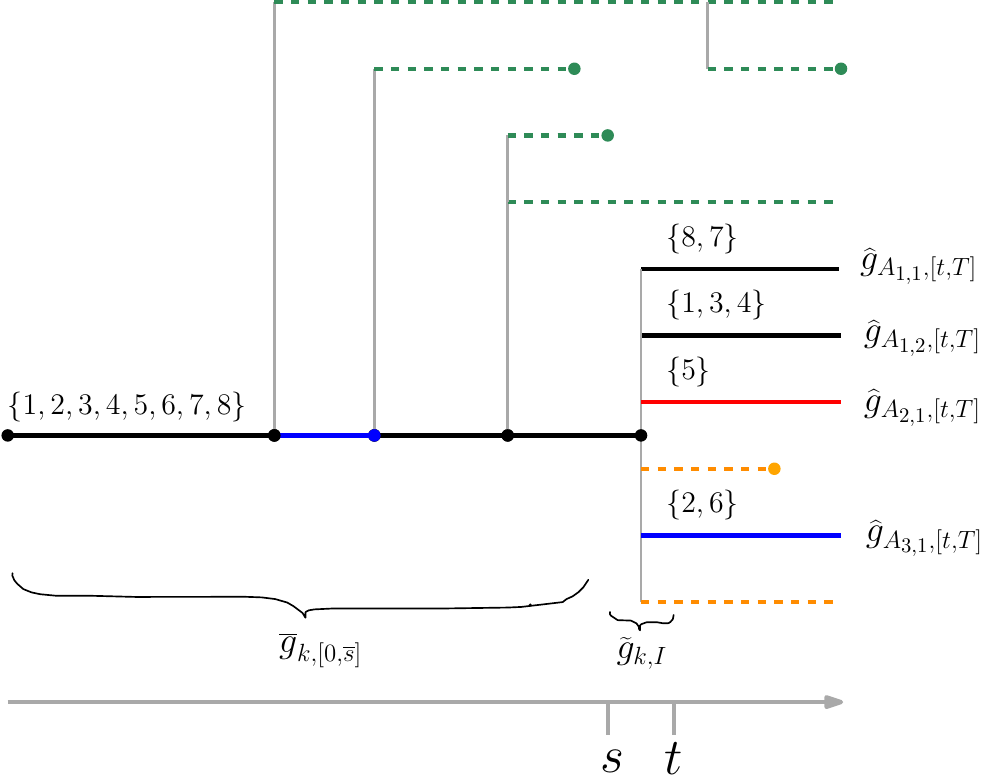}\caption{Decomposition of $g_{k,T}$ for a three type BGW tree with $8$ marks. Individuals type 1 are depicted in Black color, type 2 in Red and type 3 in Blue. 
Subtrees (with no marks) growing from births events occurring before the first spine splitting event, are depicted in green; subtrees growing after the first spine splitting  event, that do not carry a mark, are depicted in orange. Since those subtrees do not carry marks, they do not appear in the decomposition of $g_{k,T}$. 
}\label{figFirstSplittingTimeDecomposingG}
	\end{figure}
\end{center} 
	In the particular case when  $I=(s,T]$ and  $\{\tau_1\in I, c(\varsigma^{(1)}_{s})=i \}$, we have 
\begin{align*}
	g_{k,T} &= \overline{g}_{k,[0,s]}
	\mathbf{1}_{\{\varsigma^{(h_1)}_{T}\neq \varsigma^{(h_2)}_{T},\, h_1\neq h_2,h_1,h_2\in A_{i,1}\}}\prod_{h\in A_{i,1}}\prod_{\text{$(w,c_w)\in {\rm spine}(\varsigma^{(h)}_{T})$}\atop\text{${\rm gen}(w)\in I$}}\frac{ \bo{L}_{w}\cdot \bm \xi }{\xi _{c_w}},
\end{align*}
where the first term arises from the fact that all spines are together before time $s$, while in the second term, the spines remain together up to the first spine splitting  event but  now starting in a specific type $i$. We observe that at time $s$, the partition is such that  $\bo P=(\emptyset, \ldots, P_i, \ldots,\emptyset)$ with  $P_i=\{A_{i,1}\}$ where $A_{i,1}=[k]$. In other words, we have
 \begin{equation}\label{eqnDecompositionOfGLemmaOfPartitions1}	g_{k,T}=\overline{g}_{k,[0,s]} \widehat{g}_{A_{i,1},(s,T]}. 
	\end{equation}

Recall that $\delta_{i, m}=1$ when $m=i$ and 0 elsewhere. Similarly as above, under the event $\{\tau_1\in I,\mathcal{P}_{\tau_1}=\bo{P},\bo{L}_{\tau_1}=\bm{\ell} , c(\varsigma^{(1)}_{\overline{s}})=i \}\cap B_{I}$, we rewrite $\bo{Z}_T$ as follows
\begin{equation}\label{decompZ}
\begin{split}
	\bo{Z}_T=\sum_{m=1}^d\sum_{n=1}^{Z^{(m)}_{\overline{s}}-\delta_{i,m}}\overline{\bo Z}_{[t,T],m,n}+\sum_{m=1}^d\sum_{n=1}^{g_m}\bo Z_{[t,T],m,n}+\sum_{m=1}^d\sum_{n=1}^{\ell_m-g_m}\widetilde {\bo Z}_{[t,T],m,n},
\end{split}
\end{equation}where the first term represents the  $\bo Z_{\overline{s}}-\delta_{i,m}$ individuals alive at time $\overline{s}$ which generate the subpopulations $(\overline{\bo Z}_{[t,T],m,n})_{m\in[d], n\ge 1}$, where  $\overline{\bo Z}_{[t,T],m,n}$ denotes the  contribution at time $T$ from the $n$-th individual of type $m$ alive at time $t$. 
Notably, each subpopulation $\overline{\bo Z}_{[t,T],m,n}$ follows the same distribution as $\bo Z_{T-t}$ under $\mathbb{P}_{m}$. Note that  the spine, which has type $i$ at time $\overline{s}$,  does not contribute to the first term and must therefore be excluded from the total population size $\bo Z_{\overline{s}}$.

 The second term accounts for the population at time $T$ that descends  from  individuals alive at  time $t$ who carry  at least one mark. Each subpopulation $\bo Z_{[t,T],m,n}$ has the same law as $\bo Z_{T-t}$ under $\mathbb{P}^{(a_{m,n})}_m$. Finally, the last term corresponds   to individuals born at the splitting event who are alive at  $t$ and  carry  no marks.
Each subpopulation $\widetilde{\bo Z}_{[t,T],m,n}$ has the same law as $\bo Z_{T-t}$ under $\mathbb{P}_{m}$.
See Figure \ref{figFirstSplittingTimeDecomposingZ} for an illustrative example of such a decomposition of $\bo{Z}_T$. 

\begin{center}
	\begin{figure}
		\includegraphics[width=.4\textwidth]{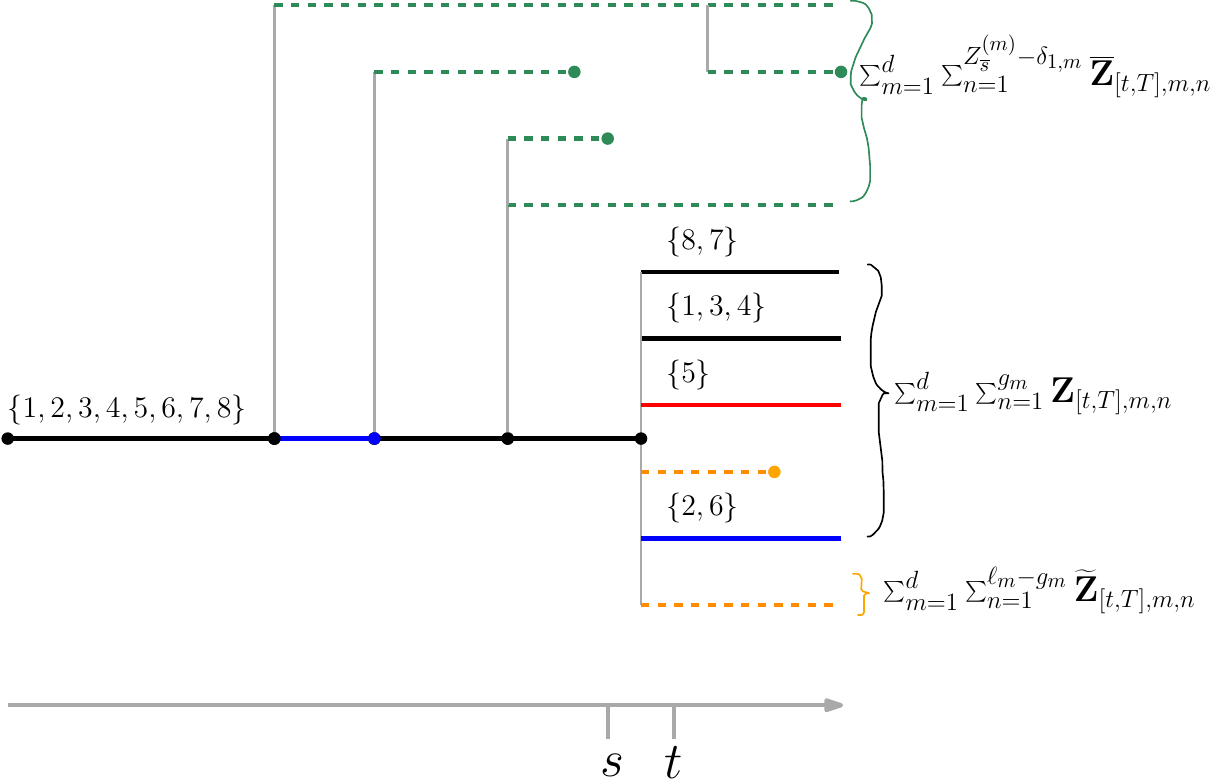}\caption{Decomposition of $\bo Z_T$ for a three type BGW tree with $8$ marks. Individuals type 1 are depicted in Black color, type 2 in Red and type 3 in Blue. 
Subtrees (with no marks) growing from birth events   occurring before the first spine splitting event, are depicted in green; subtrees growing after the first spine splitting  event, that do not carry a mark, are depicted in orange. 
}\label{figFirstSplittingTimeDecomposingZ}
	\end{figure}
\end{center}

When $I=(s,T]$ and  $\{\tau_1\in I,\, c(\varsigma^{(1)}_s)=i \}$,  the decomposition of $\bo Z_T$ is such that
\begin{equation}\label{impdecomp1}
\bo Z_{T}=\sum_{m=1}^d\sum_{n=1}^{Z^{(m)}_{s}-\delta_{i,m}}\overline{\bo Z}_{(s,T],m,n}+ \bo Z_{(s,T],i,1}.
\end{equation} 
  Indeed in this case, in the same spirit as in \eqref{decompZ}, all $g_m$ and $\ell_m$ equals 0 except $g_i=\ell_i=1$, for $m\in [d]$. Moreover since there is no splitting event by time $s$, all contributions $\widetilde {\bo Z}_{(s,T],m,n}$ equals 0 and $\bo Z_{(s,T],i,1}$ carries all marks.\\

\subsubsection{First birth time, spine splitting event and birth-off the spine}

The identities in \eqref{eqnDecompositionOfGLemmaOfPartitions1}	and \eqref{impdecomp1} allow us to deduce the following lemma (with $s=t$), which computes the probability of the following two events: that no births occur along the spine by time $t$, and that the spines remain together by time $t$. Our argument follows similar reasonings as  in Lemma 9 of \cite{MR4133376},  we include their proofs  for completeness.

\begin{lemma}\label{lemmaTailOfPsi_1}For any $t\in (0,T)$ and $i\in [d]$, we have
\[	
\Q^{(k),\bm \theta}_{T,r}\paren{\chi_\emptyset>t}=\frac{\E_r\left[N^{\floor{k}}_{T-t}e^{-\bm \theta\cdot \bo Z_{T-t}}\right]}{\E_r\left[N^{\floor{k}}_{T}e^{-\bm \theta\cdot \bo Z_{T}}\right]}e^{-\alpha_rt}
\]and 
\begin{align*}
& \Q^{(k),\bm \theta}_{T,r}\left(c(\varsigma_t^{(1)})=i,\tau_1>t\right)
 = \frac{\E_i\left[N^{\floor{k}}_{T-t}e^{-\bm \theta\cdot \bo Z_{T-t}}\right]}{\E_r\left[N^{\floor{k}}_{T}e^{-\bm \theta\cdot \bo Z_{T}}\right]}\frac{\E_r\left[Z^{(i)}_{t}\prod_{j\in [d]}\E_j\left[e^{-\bm \theta\cdot \bo Z_{T-t}}\right]^{Z^{(j)}_{t}}\right]}{\E_i\left[e^{-\bm \theta\cdot \bo Z_{T-t}}\right]}.
\end{align*}
\end{lemma}
\begin{proof} We first observe, 
\[
\begin{split}
 \Q^{(k),\bm \theta}_{T, r}\paren{\chi_\emptyset>t}
& =\frac{\E^{(k)}_r\left[g_{k,T}e^{-\bm \theta\cdot \bo Z_{T}}\mathbf{1}_{\{\chi_\emptyset>t\}} \right]}{\E^{(k)}_r\left[N^{\floor{k}}_{T}e^{-\bm \theta\cdot \bo Z_{T}}\right]}\\
&= \frac{\E^{(k)}_r\left[\mathbf{1}_{\{\chi_\emptyset>t\}} \E^{(k)}_r\left[g_{k,T}e^{-\bm \theta\cdot \bo Z_{T}} \Big |\F^{(k)}_t \right]\right]}{\E^{(k)}_r\left[N^{\floor{k}}_{T}e^{-\bm \theta\cdot \bo Z_{T}}\right]}\\
&= \frac{\E^{(k)}_r\left[g_{k,T-t}e^{-\bm \theta\cdot \bo Z_{T-t}}\right]}{\E^{(k)}_r\left[N^{\floor{k}}_{T}e^{-\bm \theta\cdot \bo Z_{T}}\right]} \p^{(k)}_r\paren{\chi_\emptyset>t}\\
& =\frac{\E_r\left[N^{\floor{k}}_{T-t}e^{-\bm \theta\cdot \bo Z_{T-t}}\right]}{\E_r\left[N^{\floor{k}}_{T}e^{-\bm \theta\cdot \bo Z_{T}}\right]}e^{-\alpha_rt},
\end{split}
\]
where in the fourth line, we have used that all terms inside $\p^{(k)}_r$ does not depend on the marks so we can replace it by $\p_r$.

For the second part of the statement, we use  decompositions  \eqref{eqnDecompositionOfGLemmaOfPartitions1}	and \eqref{impdecomp1} of  $g_{k,T}e^{-\bm \theta\cdot \bo Z_{T}}$, with $s=t$. 
Thus under the event $\{\tau_1>t,\, c(\varsigma^{(1)}_t)=i\}$ and conditioning on $\F^{(k)}_t$, we have
\[
\begin{split}
	 \E_r\left[N^{\floor{k}}_{T}e^{-\bm \theta\cdot \bo Z_{T}}\right]&\Q^{(k),\bm \theta}_{T, r}\paren{c(\varsigma^{(1)}_t)=i,\tau_1>t}=\E^{(k)}_r\left[g_{k,T}e^{-\bm \theta\cdot \bo Z_{T}}\mathbf{1}_{\{c(\varsigma^{(1)}_t)=i,\tau_1>t \}}\right]\\
	&= \E^{(k)}_r\left[\mathbf{1}_{\{c(\varsigma^{(1)}_t)=i,\tau_1>t\}} \E^{(k)}_r\left[\left.g_{k,T}e^{-\bm \theta\cdot \bo Z_{T}} \right|\F^{(k)}_t \right]\right]\\
	&= \E^{(k)}_r\left[\mathbf{1}_{\{c(\varsigma^{(1)}_t)=i,\tau_1>t\}}\prod_{(w,c_w)\in {\rm spine}(\varsigma^{(1)}_t) }\paren{\frac{\bo L_w\cdot \bm \xi}{\xi_{c_{w}}}}^k \right.\\
	& \qquad \hspace{4cm}\times  \E^{(k)}_r\left[\left.\prod_{m\in [d]}e^{-\sum_{n=1}^{Z^{(m)}_{t}-\delta_{i,m}}\bm \theta\cdot  \overline{\bo Z}_{(t,T],m,n}}\right|\F^{(k)}_t\right] \\ 
	& \quad \left.\times \E^{(k)}_r\left[\left.\mathbf{1}_{\{\varsigma^{(m)}_{T}\neq \varsigma^{(n)}_{T}, m, n\in [k]\}}\prod_{h\in [k]}\prod_{\substack{(w,c_w)\in {\rm spine}(\varsigma^{(h)}_T)\\{\rm gen}(w)\in (t,T]}}\frac{\bo L_w\cdot \bm \xi}{\xi_{c_{w}}}e^{-\bm \theta\cdot \bo Z_{(t,T],i,1}}\right|\F^{(k)}_t\right]\right]\\
	&= \E^{(k)}_r\left[\mathbf{1}_{\{c(\varsigma^{(1)}_t)=i,\tau_1>t\}}\prod_{(w,c_w)\in {\rm spine}(\varsigma^{(1)}_t) }\paren{\frac{\bo L_w\cdot \bm \xi}{\xi_{c_{w}}}}^k \right.\\
&\hspace{3cm} \left.\times\prod_{m\in [d]}\E_m\left[e^{-\bm \theta\cdot \bo Z_{T-t}}\right]^{Z^{(m)}_{t}-\delta_{i,m}} \right] \E^{(k)}_i\left[g_{k,T-t}e^{-\bm \theta\cdot \bo Z_{T-t}}\right],
\end{split}
\]
where we recall that $\delta_{i,m}$ equals one when $m=i$ and zero otherwise. 

We now compute the first expectation on the right-hand side. Denote by $\mathcal{N}^{(i)}_{t}$  the set of all  particles alive at time $t$ of type $i$.
Since $\bo Z_{t}$ is measurable w.r.t. $\F_t$, we condition on the latter $\sigma$-algebra to obtain that the first expectation on the right-hand side above equals
\begin{equation}
\begin{split}
\label{eqnChangeFromInfoUpTo_t}
	& \E^{(k)}_r\left[\prod_{m\in[d]}\E_m\left[e^{-\bm \theta\cdot \bo Z_{T-t}}\right]^{Z^{(m)}_{t}-\delta_{i,m}} \E^{(k)}_r\left[\left.\mathbf{1}_{\{c(\varsigma^{(1)}_t)=i,\tau_1>t\}}\prod_{(w,c_w)\in {\rm spine}(\varsigma^{(1)}_t) }\paren{\frac{\bo L_w\cdot \bm \xi}{\xi_{c_{w}}}}^k \right|\F_t\right]\right]\\ 
	& \hspace{7.5cm}= \E_r\left[Z^{(i)}_{t}\prod_{m\in[d]}\E_m\left[e^{-\bm \theta\cdot \bo Z_{T-t}}\right]^{Z^{(m)}_{t}-\delta_{i,m}}\right],
 \end{split}
\end{equation}
where the identity follows from similar arguments as those used above identity \eqref{tomate} where we computed the conditional expectation of $g_{k,t}$ with respect to $\mathcal{F}_t$, but in this case we sum over all  possible spines which are of type $i$ at time $t$. 
Thus putting all pieces together, we deduce
\begin{align*}
	& \Q^{(k),\bm \theta}_{T, r}\paren{c(\varsigma^{(1)}_t)=i,\tau_1>t}= \frac{\E_i\left[N^{\floor{k}}_{T-t}e^{-\bm \theta\cdot \bo Z_{T-t}}\right]}{\E_r\left[N^{\floor{k}}_{T}e^{-\bm \theta\cdot \bo Z_{T}}\right]}\frac{\E_r\left[Z^{(i)}_{t}\prod_{m\in [d]}\E_m\left[e^{-\bm \theta\cdot \bo Z_{T-t}}\right]^{Z^{(m)}_{t}}\right]}{\E_i\left[e^{-\bm \theta\cdot \bo Z_{T-t}}\right]},	
\end{align*}where the last equality follows by observing that  
\[
\E_m\left[e^{-\bm \theta\cdot \bo Z_{T-t}}\right]^{-\delta_{i,m}}\mathbf{1}_{\{m=i\}}=\E_i\left[e^{-\bm \theta\cdot \bo Z_{T-t}}\right]^{-1}\mathbf{1}_{\{m=i\}}.
\] This completes the proof.
\end{proof}

It is important to note that in the single-type case, the following relationship holds
\[
\mathbb{E}\left[Z_T F_{T-t}(e^{-\theta})^{Z_T-1}\right]=F'_t(F_{T-t}(e^{-\theta}))=\frac{F'_T(e^{-\theta})}{F'_{T-t}(e^{-\theta})},
\]
where $F_{T}(e^{-\theta})=\mathbb{E}[e^{-\theta Z_T}]$;
see for instance the proof of Lemma 3.4 in Harris et al. \cite{MR4718398}. In the multitype case a similar identity can be also obtained. Let us define  $F_{t,r}$ as the probability generating function of $\bo Z_{t}$ starting from an individual type $r\in [d]$, i.e. $F_{t,r}(\bo s)=\mathbb{E}_r[{\bo s}^{\bo Z_{t}}]$, for ${\bo s}\in [0,1]^d$. 
Thus, writing $\bm \theta:=\bm \theta (u)$, we have
\begin{equation}\label{eqmultytipeSam}
\begin{split}
\frac{{\rm d}}{{\rm d}u}&F_{T-t, i}(e^{-\bm \theta (u)})\E_r\left[Z^{(i)}_{t}\prod_{m\in [d]}\E_m\left[e^{-\bm \theta(u)\cdot \bo Z_{T-t}}\right]^{Z^{(m)}_{t}-\delta_{m,i}}\right]=\\
&\frac{{\rm d}}{{\rm d}u}F_{t, r}(\Vec{F}_{T-t}(e^{-\bm \theta(u)}))-\sum_{\ell\neq i}\E_r\left[Z^{(\ell)}_{t}\prod_m\E_m\left[e^{-\bm \theta(u)\cdot \bo Z_{T-t}}\right]^{Z^{(m)}_{t}-\delta_{m,\ell}}\right] \frac{{\rm d}}{{\rm d}u}F_{T-t,\ell}(e^{-\bm \theta (u)}),
\end{split}
\end{equation}
where
\begin{equation}\label{eqnDerivativeProbGenFn}
\Vec{F}_{T}(e^{-\bm \theta})=\paren{\E_1\left[e^{-\bm \theta\cdot \bo Z_{T}}\right],\ldots, \E_d\left[e^{-\bm \theta\cdot \bo Z_{T}}\right]}.
\end{equation}For the proof of the previous identity, we refer to the \refappendix{Appendix}.

Our next result computes the number of offspring at {\it births-off the spine}, that is a birth event that occurs along the spines but  do not involve a spine splitting event. This result correspond to  Lemma 10 of \cite{MR4133376}.

Recall that  $\bo{L}_\emptyset$ denotes the offspring distribution of the root and observe that  $c (\varsigma^{(1)}_{t+})$ refers to the type of the first mark immediately after time $t$.

\begin{lemma}\label{lemmatomate}
	Fix $i\in [d]$. For any $\bm{\ell}\in \z^d_+\setminus\{\bo 0\}$ with $\ell_i>0$ and $0< t< T$, we have 
\begin{align*}
	& 	\Q^{(k),\bm \theta}_{T,r}\paren{\bo{L}_\emptyset=\bm{\ell}\left|\chi_\emptyset\in {\rm d} t,\tau_1>t ,c (\varsigma^{(1)}_{t+})=i \right.} = \frac{\ell_ip_r\paren{\bm{\ell}}\prod_{m\in[d]}\E_m\left[e^{-\bm \theta\cdot \bo Z_{T-t}}\right]^{\ell_m}}{
	\mathbb{E}_r\left[L^{(i)}\prod_{m\in[d]}\E_m\left[e^{-\bm \theta\cdot \bo Z_{T-t}}\right]^{L^{(m)}}\right]},
\end{align*}
where $L^{(i)}$ denotes the $i$-th coordinate of the r.v. $\bo{L}$.
\end{lemma}
\begin{proof}
Let  $\epsilon>0$ be small enough. From  the definition of $\Q^{(k),\bm \theta}_{T,r}$, we first note that 
\[
\begin{split}
	 	\Q^{(k),\bm \theta}_{T,r}&\paren{\bo{L}_\emptyset=\bm{\ell}\left|\chi_\emptyset\in [t, t+\epsilon),\tau_1>t +\epsilon ,c (\varsigma^{(1)}_{t+\epsilon})=i \right.} \\
		&\hspace{3cm}= \frac{\E^{(k)}_r\left[g_{k,T}e^{-\bm \theta\cdot \bo Z_{T}}\mathbf{1}_{\{\bo{L}_\emptyset=\bm{\ell}, \chi_\emptyset\in [t, t+\epsilon),\tau_1>t +\epsilon,c (\varsigma^{(1)}_{t+\epsilon})=i\}}\right]}{\E^{(k)}_r\left[g_{k,T}e^{-\bm \theta\cdot \bo Z_{T}}\mathbf{1}_{ \{\chi_\emptyset\in [t, t+\epsilon),\tau_1>t +\epsilon ,c (\varsigma^{(1)}_{t+\epsilon})=i\}}\right]}.
\end{split}
\]
Let us first study  the numerator. We proceed similarly as in \eqref{eqnDecompositionOfGLemmaOfPartitions} but under the first branching event instead of the first spine splitting event.  Thus under
the event 
\[
\{\chi_\emptyset\in [t, t+\epsilon), \bo{L}_\emptyset=\bm{\ell}, \tau_1>t +\epsilon, c (\varsigma^{(1)}_{t+\epsilon})=i\}\cap B_{[t,t+\epsilon)},
\]  
the term $g_{k,T}$ can be rewritten as follows
\[
	g_{k,T}=\left(\frac{\bm{\ell}\cdot \bm \xi}{\xi _{i}}\right)^k\mathbf{1}_{\{\varsigma^{(m)}_T\neq \varsigma^{(n)}_T,\, m, n\in[k]\}}\prod_{h\in[k]}\prod_{\substack{(w,c_w)\in {\rm spine}(\varsigma^{(h)}_T),\\{\rm gen}(w)\in [t+\epsilon,T]}}\frac{\bo{L}_w\cdot \bm \xi}{\xi _{c_w}}.
\]
Moreover,   under the same event, we observe that  the r.v. $\bo Z_{T}$, given $\F_{t+\epsilon}^{(k)}$, can be decomposed as in \eqref{impdecomp1} but with $s=t+\epsilon$.
Thus, under  $\{\bo{L}_\emptyset=\bm{\ell}, \chi_\emptyset\in[t,t+\epsilon),\tau_1>t +\epsilon ,c (\varsigma^{(1)}_{t+\epsilon})=i\}$, we have
\begin{align*}
	& \E^{(k)}_r\left[\left.g_{k,T}e^{-\bm \theta\cdot \bo Z_{T}}\right|\F^{(k)}_{t+\epsilon}\right] = \mathbf{1}_{B_{[t,t+\epsilon)}}\left(\frac{\bm{\ell}\cdot \bm \xi}{\xi _{i}}\right)^k\E^{(k)}_i\left[g_{k,T-t-\epsilon}e^{-\bm \theta\cdot \bo Z_{T-t-\epsilon}}\right]\\
	&\hspace{7cm}\times\prod_{m\in[d]}\E_m\left[e^{-\bm \theta\cdot \bo Z_{T-t-\epsilon}}\right]^{\ell_m-\delta_{i,m}} +o(\epsilon),
\end{align*}
where the term $o(\epsilon)$ corresponds to the contribution of the expectation of the r.v. $g_{k,T} e^{-\bm \theta\cdot \bo Z_{T}}$ under the event that more than one branching event have occur.

In other words from Lemma \eqref{lemmaMrkovPropertyUnderQ}, we get
\[
\begin{split}
	 \E^{(k)}_r&\left[g_{k,T}e^{-\bm \theta\cdot \bo Z_{T}}\mathbf{1}_{\{\bo{L}_\emptyset=\bm{\ell}, \chi_\emptyset\in[t,t+\epsilon),\tau_1>t +\epsilon ,c (\varsigma^{(1)}_{t+\epsilon})=i\}}\right]\\
	 &\hspace{2cm} =  \E^{(k)}_r\left[\mathbf{1}_{\{\bo{L}_\emptyset=\bm{\ell}, \chi_\emptyset\in[t,t+\epsilon),\tau_1>t +\epsilon ,c (\varsigma^{(1)}_{t+\epsilon})=i\}}\E^{(k)}_r\left[\left.g_{k,T}e^{-\bm \theta\cdot \bo Z_{T}}\right|\F^{(k)}_{t+\epsilon}\right]\right]\\
	&\hspace{2cm} = \left(\frac{\bm{\ell}\cdot \bm \xi}{\xi _{i}}\right)^k\E^{(k)}_i\left[g_{k,T-t-\epsilon}e^{-\bm \theta\cdot \bo Z_{T-t-\epsilon}}\right]\prod_{m\in[d]}\E_m\left[e^{-\bm \theta\cdot \bo Z_{T-t-\epsilon}}\right]^{\ell_m-\delta_{i,m}}\\
	& \qquad\hspace{1.5cm} \times \p^{(k)}_r\paren{\bo{L}_\emptyset=\bm{\ell}, \chi_\emptyset\in[t,t+\epsilon),\tau_1>t +\epsilon ,c (\varsigma^{(1)}_{t+\epsilon})=i, B_{[t,t+\epsilon)}} +o(\epsilon).\\
\end{split}
\]
 Next, we observe that 
\begin{align*}
	 \p^{(k)}_r&\paren{ \bo{L}_\emptyset=\bm{\ell}, \chi_\emptyset\in[t,t+\epsilon),\tau_1>t +\epsilon ,c (\varsigma^{(1)}_{t+\epsilon})=i, B_{[t,t+\epsilon)}}\\
	&\hspace{1.5cm} =\E^{(k)}_r\left[\mathbf{1}_{\{\bo{L}_\emptyset=\bm{\ell}, B_{[t,t+\epsilon)}, \chi_\emptyset\in[t,t+\epsilon) \}}\p^{(k)}_r\paren{\left.\varsigma^{(1)}_{t+\epsilon}=\cdots =\varsigma^{(k)}_{t+\epsilon},c (\varsigma^{(1)}_{t+\epsilon})=i \right|\ \F_{t+\epsilon}}\right]\\
	&\hspace{1.5cm}=e^{-\alpha_rt}(1-e^{-\alpha_r\epsilon})p_r(\bm{\ell})\paren{\prod_{h=1}^k\frac{\ell_{i}\xi_i}{\bm{\ell}\cdot\bm{\xi}}}\paren{\sum_{j=1}^{\ell_i}\frac{1}{\ell_{i}}}\paren{\frac{1}{\ell_{i}}}^{k-1}\\
	&\hspace{1.5cm}=e^{-\alpha_rt}(1-e^{-\alpha_r\epsilon})p_r(\bm{\ell})\paren{\frac{\ell_{i}\xi_i}{\bm{\ell}\cdot\bm{\xi}}}^k\paren{\frac{1}{\ell_{i}}}^{k-1}\\
	&\hspace{1.5cm}=e^{-\alpha_rt}(1-e^{-\alpha_r\epsilon})p_r(\bm{\ell})\paren{\frac{\xi_i}{\bm{\ell}\cdot\bm{\xi}}}^k \ell_{i}.
\end{align*}Putting all pieces together, we deduce
\begin{equation}\label{eqnTomate1}
\begin{split}
	 \E^{(k)}_r&\left[g_{k,T}e^{-\bm \theta\cdot \bo Z_{T}}\mathbf{1}_{\{\bo{L}_\emptyset=\bm{\ell}, \chi_\emptyset\in[t,t+\epsilon),\tau_1>t +\epsilon ,c (\varsigma^{(1)}_{t+\epsilon})=i\}}\right]\\	& 
	= \left(\frac{\bm{\ell}\cdot \bm \xi}{\xi _{i}}\right)^k\E^{(k)}_i\left[g_{k,T-t-\epsilon}e^{-\bm \theta\cdot \bo Z_{T-t-\epsilon}}\right]\prod_{m\in[d]}\E_m\left[e^{-\bm \theta\cdot \bo Z_{T-t-\epsilon}}\right]^{\ell_m-\delta_{i,m}}\\
	& \qquad\hspace{5cm} \times e^{-\alpha_rt}(1-e^{-\alpha_r\epsilon})p_r(\bm{\ell})\paren{\frac{\xi_i}{\bm{\ell}\cdot\bm{\xi}}}^k \ell_{i}+o(\epsilon)\\
	&= \E^{(k)}_i\left[g_{k,T-t-\epsilon}e^{-\bm \theta\cdot \bo Z_{T-t-\epsilon}}\right]\prod_{m\in[d]}\E_m\left[e^{-\bm \theta\cdot \bo Z_{T-t-\epsilon}}\right]^{\ell_m-\delta_{i,m}}e^{-\alpha_rt}(1-e^{-\alpha_r\epsilon})p_r(\bm{\ell}) \ell_{i}+o(\epsilon)\\
	& 
	= e^{-\alpha_rt}(1-e^{-\alpha_r\epsilon})\E_i\left[N^{\floor{k}}_{T-t-\epsilon}e^{-\bm \theta\cdot \bo Z_{T-t-\epsilon}}\right]\prod_{m\in[d]}\E_m\left[e^{-\bm \theta\cdot \bo Z_{T-t-\epsilon}}\right]^{\ell_m-\delta_{i,m}}p_r(\bm{\ell}) \ell_{i}+o(\epsilon).
\end{split}
\end{equation}
Note that the first expectation on the right-hand side does not depend on $\bm{\ell}$. 
Thus, performing similar steps for the denominator,  we get
\[
\begin{split}
	\Q^{(k),\bm \theta}_{T,r}&\paren{\bo{L}_\emptyset=\bm{\ell}\left|\chi_\emptyset\in [t, t+\epsilon),\tau_1>t +\epsilon ,c (\varsigma^{(1)}_{t+\epsilon})=i \right.}\\
	&\hspace{1cm} =\frac{\prod_{m\in[d]}\E_m\left[e^{-\bm \theta\cdot \bo Z_{T-t-\epsilon}}\right]^{\ell_m-\delta_{i,m}}p_r(\bm{\ell}) \ell_{i}}{
	\sum_{\bo j:j_i>0}j_ip_r(\bo{j})\prod_{m\in[d]}\E_m\left[e^{-\bm \theta\cdot \bo Z_{T-t-\epsilon}}\right]^{j_m-\delta_{i,m}}+o(\epsilon)}+o(\epsilon).
\end{split}
\]
The proof is completed once we take the limit in both sides of the identity as $\epsilon\to 0$ and using the fact that $ \bo Z$ is a Feller process with  c\`adl\`ag paths.
\end{proof}

A similar computation gives us the law of the time that a birth-off the spine occurs, that is, a time where a particle carrying $k$ marks, gives birth to $\bm \ell$ individuals, and all the marks follow the same individual type $i$. 

\begin{lemma}\label{lemmaBirthsOffTheSpine}
	Fix $i\in [d]$. For any $\bm{\ell}\in \z^d_+\setminus\{\bo 0\}$ with $\ell_i>0$ and $0< t<T$, we have
	\begin{align*}
		\Q^{(k),\bm \theta}_{T,r}&\paren{\chi_\emptyset\in {\rm d}t,\bo{L}_{\emptyset}=\bm{\ell}\left| c (\varsigma^{(1)}_{t+})=i , \tau_1>t\right. }= \frac{\ell_ip_r(\bm{\ell})\paren{\prod_{m\in[d]}\E_m\left[e^{-\bm \theta\cdot \bo Z_{T-t}}\right]^{\ell_m}}	}{\E_r\left[Z^{(i)}_{t}\prod_{m\in[d]}\E_m\left[e^{-\bm \theta\cdot \bo Z_{T-t}}\right]^{Z^{(m)}_{t}}\right]}\alpha_re^{-\alpha_rt}{\rm d}t.	
	\end{align*}
\end{lemma}
\begin{proof}
Let $\epsilon$ be small enough. Using the definition of   $\Q^{(k),\bm \theta}_{T,r}$ and the second part of Lemma \ref{lemmaTailOfPsi_1}, we get
\[
\begin{split}
	 \Q^{(k),\bm \theta}_{T,r}&\paren{\chi_\emptyset\in[t,t+\epsilon),\bo{L}_{\emptyset}=\bm{\ell}\Big| c (\varsigma^{(1)}_{t+\epsilon})=i , \tau_1>t+\epsilon }\\
	&\hspace{2.5cm} = \frac{\E^{(k)}_r\left[g_{k,T}e^{-\bm \theta\cdot \bo Z_{T}}\mathbf{1}_{\{\chi_\emptyset\in[t,t+\epsilon),\bo{L}_{\emptyset}=\bm{\ell}, c (\varsigma^{(1)}_{t+\epsilon})=i , \tau_1>t +\epsilon \}}\right]}{\E_i\left[N^{\floor{k}}_{T-t-\epsilon}e^{-\bm \theta\cdot \bo Z_{T-t-\epsilon}}\right]\E_r\left[Z^{(i)}_{t+\epsilon}\prod_{m\in [d]}\E_m\left[e^{-\bm \theta\cdot \bo Z_{T-t-\epsilon}}\right]^{Z^{(m)}_{t+\epsilon}-\delta_{i,m}}\right]}.
\end{split}
\]
From the proof of the previous Lemma, the denominator satisfies
\[
\begin{split}
	 \E^{(k)}_r&\left[g_{k,T}e^{-\bm \theta\cdot \bo Z_{T}}\mathbf{1}_{\{\bo{L}_\emptyset=\bm{\ell}, \chi_\emptyset\in[t,t+\epsilon),\tau_1>t +\epsilon ,c (\varsigma^{(1)}_{t+\epsilon})=i\}}\right]\\	& 
	= e^{-\alpha_rt}(1-e^{-\alpha_r\epsilon})\E_i\left[N^{\floor{k}}_{T-t-\epsilon}e^{-\bm \theta\cdot \bo Z_{T-t-\epsilon}}\right]\prod_{m\in[d]}\E_m\left[e^{-\bm \theta\cdot \bo Z_{T-t-\epsilon}}\right]^{\ell_m-\delta_{i,m}}p_r(\bm{\ell}) \ell_{i}+o(\epsilon).
\end{split}
\]
Thus putting all pieces together we deduce
\[
\begin{split}
	 \Q^{(k),\bm \theta}_{T,r}&\left(\chi_\emptyset\in[t,t+\epsilon),\bo{L}_{\emptyset}=\bm{\ell}\Big| c (\varsigma^{(1)}_{t+\epsilon})=i , \tau_1>t+\epsilon \right)\\
	&\hspace{3cm} = \frac{e^{-\alpha_rt}(1-e^{-\alpha_r\epsilon})\prod_{m\in[d]}\E_m\left[e^{-\bm \theta\cdot \bo Z_{T-t-\epsilon}}\right]^{\ell_m-\delta_{i,m}}p_r(\bm{\ell}) \ell_{i}}{\E_r\left[Z^{(i)}_{t+\epsilon}\prod_{m\in [d]}\E_m\left[e^{-\bm \theta\cdot \bo Z_{T-t-\epsilon}}\right]^{Z^{(m)}_{t+\epsilon}-\delta_{i,m}}\right]}+o(\epsilon).
\end{split}
\]
The result then follows by taking the limit as $\epsilon\to 0$ and using the fact that $ \bo Z$ is a Feller process with  c\`adl\`ag paths.\end{proof}
Summing over all the possible values of $\bm \ell$, we obtain the following corollary.

\begin{coro}
The law of the first birth event when all particles follow the same type $i\in [d]$ individual, is given by
\begin{align*}
	\Q^{(k),\bm \theta}_{T,r}&\left(\chi_\emptyset\in {\rm d}t\left| c (\varsigma^{(1)}_{t+})=i , \tau_1>t\right. \right)= \frac{\mathbb{E}_r\left[L^{(i)}\paren{\prod_{m\in[d]}\E_m\left[e^{-\bm \theta\cdot \bo Z_{T-t}}\right]^{L^{(m)}}}\right]}{\E_r\left[Z^{(i)}_{t}\prod_{m\in[d]}\E_m\left[e^{-\bm \theta\cdot \bo Z_{T-t}}\right]^{Z^{(m)}_{t}}\right]}\alpha_re^{-\alpha_rt}{\rm d}t.	
\end{align*}
\end{coro}

\subsubsection{Joint law of first splitting event, the partition and number of  offspring being born}
For the next lemma, we consider  a stochastic process that takes values in the space of coloured partitions. More precisely, we  define $\mathcal{P}:=(\mathcal{P}_t;t\geq 0)$ as the coloured partition process associated with the spines which encodes the grouping  of marks  at any time $t\geq 0$. 
 Also, for any $\bm \ell,\bo g\in \mathbb{Z}^d_+$, we introduce the following notation
\[
\bm \ell^{\floor{\bo g}}:=	\prod_{\substack{m\in [d]\\ g_m\neq 0}}\ell_m^{\floor{g_m}},
\]
where we recall that $\bo g=(g_1, \ldots, g_d)$. We also recall that $\bo{L}_{\tau_1}$ denotes the offspring distribution of the vertex $\varsigma^{(1)}_{\tau_1-}$. The next lemma corresponds to Lemma 3.6 in \cite{MR4718398}.

\begin{lemma}\label{lemmaGeneralRateWhenABranchingEventOccursKnowingInfo}
	 Conditional on $\{\chi_\emptyset>t\}$, the $\Q^{(k),\bm \theta }_{T,r}$-conditional probability that during the time interval $[t,t+\epsilon)$ there are no births-off the spine, the particle  carrying all marks dies and gives birth to $\bm \ell\in \z^d_+\setminus\{\bo 0\} $ offspring, and  the marks are partitioned according to $\bo P$ with block sizes $(a_{m,q})_{m\in[d], q\in [g_m]}$, is given by
	\[
	\begin{split}
		\Q^{(k),\bm \theta }_{T,r}&\left(\left.\tau_1<t+\epsilon,\mathcal{P}_{\tau_1}=\bo{P},\bo{L}_{\tau_1}=\bm{\ell},  B_{[t,t+\epsilon)} \right| \chi_\emptyset>t\right)\\
		&\hspace{.5cm} =\E_r\left[\bo L^{\floor{\bo g}}\prod_{m=1}^d\E_m\left[e^{-\bm \theta \cdot \bo Z_{T-t-\epsilon}}\right]^{L^{(m)}-g_m}\right]\frac{\bm \ell^{\floor{\bo g}} p_r(\bm{\ell})
		\prod_{m=1}^d\E_m\left[e^{-\bm \theta \cdot \bo Z_{T-t-\epsilon}}\right]^{\ell_m-g_m}}{\E_r\left[\bo L^{\floor{\bo g}}\prod_{m=1}^d\E_m\left[e^{-\bm \theta \cdot \bo Z_{T-t-\epsilon}}\right]^{L^{(m)}-g_m}\right]}\\
		& \hspace{5cm} \times
		\frac{\prod_{\substack{m\in [d]\\ g_m\neq 0}}\prod_{q=1}^{g_m}\E_m\left[N^{\floor{a_{m,q}}}_{T-t-\epsilon}e^{-\bm \theta \cdot \bo Z_{T-t-\epsilon}}\right]}{\E_r\left[N^{\floor{k}}_{T-t}e^{-\bm \theta\cdot \bo Z_{T-t}}\right]}\alpha_r \epsilon+o(\epsilon).
	\end{split}
	\]
\end{lemma}
\begin{remark}\label{remark1}
In particular, for a birth-off the spine, that is, if the partition consist of only one type and only one group, say $\bo P=\{P_{j} \}=\{\{A_{j,1}\} \}$, then we have 
\[
\begin{split}
	\Q^{(k),\bm \theta }_{T,r}&\left(\left.\chi_\emptyset<t+\epsilon,\mathcal{P}_{\chi_\emptyset}=\bo{P},\bo{L}_{\emptyset}=\bm{\ell},   B_{[t,t+\epsilon)}\right| \chi_\emptyset>t\right)\\
	& = \ell_jp_r(\bm{\ell})
	\left(\prod_{m=1}^d\E_m\left[e^{-\bm \theta \cdot \bo Z_{T-t-\epsilon}}\right]^{\ell_m-\delta_{m,j}} \right)
	\frac{\E_j\left[N^{\floor{k}}_{T-t-\epsilon}e^{-\bm \theta \cdot \bo Z_{T-t-\epsilon}}\right]}{\E_r\left[N^{\floor{k}}_{T-t}e^{-\bm \theta\cdot \bo Z_{T-t}}\right]}\alpha_r \epsilon+o(\epsilon),
\end{split}
\]
which clearly depends on $k$. 

Moreover, dividing by $\epsilon$ in the identity in Lemma \ref{lemmaGeneralRateWhenABranchingEventOccursKnowingInfo} and then taking $\epsilon$ going to $0$, we may deduce
\[
	\begin{split}
		\Q^{(k),\bm \theta }_{T,r}&\left(\left.\tau_1\in {\rm d} t,\mathcal{P}_{\tau_1}=\bo{P},\bo{L}_{\tau_1}=\bm{\ell} \right| \chi_\emptyset>t\right)\\
		&\hspace{1cm} =\E_r\left[\bo L^{\floor{\bo g}}\prod_{m=1}^d\E_m\left[e^{-\bm \theta \cdot \bo Z_{T-t}}\right]^{L^{(m)}-g_m}\right]\frac{\bm \ell^{\floor{\bo g}} p_r(\bm{\ell})
		\prod_{m=1}^d\E_m\left[e^{-\bm \theta \cdot \bo Z_{T-t}}\right]^{\ell_m-g_m}}{\E_r\left[\bo L^{\floor{\bo g}}\prod_{m=1}^d\E_m\left[e^{-\bm \theta \cdot \bo Z_{T-t}}\right]^{L^{(m)}-g_m}\right]}\\
		& \hspace{6.5cm} \times
		\frac{\prod_{\substack{m\in [d]\\ g_m\neq 0}}\prod_{q=1}^{g_m}\E_m\left[N^{\floor{a_{m,q}}}_{T-t}e^{-\bm \theta \cdot \bo Z_{T-t}}\right]}{\E_r\left[N^{\floor{k}}_{T-t}e^{-\bm \theta\cdot \bo Z_{T-t}}\right]}\alpha_r {\rm d} t.
	\end{split}
	\]
	The latter follows since $ \bo Z$ is a Feller process with  c\`adl\`ag paths and by the dominated convergence theorem.
\end{remark}
\begin{proof}
	The proof follows from the  Markov branching property  at time $t$ and starting new MBGW processes  with laws $\p^{(a_{m,q})}_m$,  for $m\in [d]$, $q\in[g_m]$, with $a_{m,q}={\rm card}\{ A_{m,q}\}$. 
Indeed by Lemma \ref{lemmaMrkovPropertyUnderQ} applied at time $t$, we have
	\[
	\begin{split}
		 \Q^{(k),\bm \theta }_{T,r}&\left(\left.\tau_1<t+\epsilon,\mathcal{P}_{\tau_1}=\bo{P},\bo{L}_{\tau_1}=\bm{\ell}, B_{[t,t+\epsilon)} \right| \chi_\emptyset>t\right)\\
		 &\hspace{5cm}= \Q^{(k),\bm \theta }_{T-t,r}\left(\tau_1<\epsilon,\mathcal{P}_{\tau_1}=\bo{P},\bo{L}_{\tau_1}=\bm{\ell}, B_{[0,\epsilon)}\right).
	\end{split}
	\]
	Using the definition of $\Q^{(k),\bm \theta }_{T-t,r}$ we have
	\begin{equation}\label{problemma7}
	\begin{split}
		 \Q^{(k),\bm \theta }_{T-t,r}&\left(\tau_1<\epsilon,\mathcal{P}_{\tau_1}=\bo{P},\bo{L}_{\tau_1}=\bm{\ell}, B_{[0,\epsilon)}\right)\\
		& \hspace{2cm}= \frac{1}{\E^{(k)}_r\left[N^{\floor{k}}_{T-t}e^{-\bm \theta\cdot \bo{Z}_{T-t}}\right]}\E^{(k)}_r\left[g_{k,T-t}e^{-\bm \theta\cdot \bo{Z}_{T-t}}\mathbf{1}_{\{\tau_1<\epsilon,\mathcal{P}_{\tau_1}=\bo{P},\bo{L}_{\tau_1}=\bm{\ell}, B_{[0,\epsilon)}\}}\right].
	\end{split}
	\end{equation}
	Now, we deal with the  numerator in the above identity and observe
	\begin{equation}\label{eqnLemmaLawAfterSplittingEventUsingPartitions1}
		\begin{split}
			 \E^{(k)}_r&\left[g_{k,T-t}e^{-\bm \theta\cdot \bo{Z}_{T-t}}\mathbf{1}_{\{\tau_1<\epsilon,\mathcal{P}_{\tau_1}=\bo{P},\bo{L}_{\tau_1}=\bm{\ell}, B_{[0,\epsilon)}\}}\right]\\
			&= \p^{(k)}_r\left(\tau_1<\epsilon,\mathcal{P}_{\tau_1}=\bo{P},\bo{L}_{\tau_1}=\bm{\ell}, B_{[0,\epsilon)}\right)\\
			&\hspace{2cm}\times\E^{(k)}_r\left[g_{k,T-t}e^{-\bm \theta\cdot \bo{Z}_{T-t}}\left|\tau_1<\epsilon,\mathcal{P}_{\tau_1}=\bo{P},\bo{L}_{\tau_1}=\bm{\ell}, B_{[0,\epsilon)}\right.\right]\\
			& = \p^{(k)}_r\left(\chi_{\emptyset}<\epsilon, \bo{L}_{\emptyset}=\bm{\ell}, B_{[0,\epsilon)}\right)\p^{(k)}_r\left(\left.\mathcal{P}_{\tau_1}=\bo{P}\right|\tau_1<\epsilon,\bo{L}_{\tau_1}=\bm{\ell}, B_{[0,\epsilon)}\right)\\
			&\hspace{4cm}\times\E^{(k)}_r\left[g_{k,T-t}e^{-\bm \theta\cdot \bo{Z}_{T-t}}\Big|\tau_1<\epsilon,\mathcal{P}_{\tau_1}=\bo{P},\bo{L}_{\tau_1}=\bm{\ell}, B_{[0,\epsilon)}\right],
		\end{split}
	\end{equation}
where the last identity follows from the fact that  
\[
\{\tau_1<\epsilon, \bo{L}_{\tau_1}=\bm{\ell}\}\cap B_{[0,\epsilon)}=\{\chi_{\emptyset}<\epsilon, \bo{L}_{\emptyset}=\bm{\ell}\}\cap B_{[0,\epsilon)}.
\]
	Given that there are $\bm \ell$ children born at $\tau_1$, which is the first branching event in $[0,\epsilon)$, all the $k$ marks spread through the children according to 
	\begin{equation}\label{eqnProbabilityOfAGivenPartition}
		 \begin{split}
		 	\p^{(k)}_r\left(\left.\mathcal{P}_{\tau_1}=\bo{P}\right|\tau_1<\epsilon,\bo{L}_{\tau_1}=\bm{\ell}, B_{[0,\epsilon)}\right)
		 &=\prod_{\substack{m\in [d]\\ g_m\neq 0}}\paren{\prod_{j=1}^{\overline{a}_{m}}\frac{\ell_m\xi _m}{\bm \ell \cdot \bm \xi}}\frac{\ell_m^{\floor{g_m}}}{\ell_m^{\overline{a}_{m}}}\\
		 &=\prod_{\substack{m\in [d]\\ g_m\neq 0}}\paren{\frac{\xi _m}{\bm \ell \cdot \bm \xi}}^{\overline{a}_{m}}\ell_m^{\floor{g_m}},
		 \end{split}
	\end{equation}
	since the number of marks that will follow an individual type $m$ is $\overline{a}_{m}$, and given that $\overline{a}_{m}$ balls are put uniformly and independently into $\ell_m$ urns, they are splitted into exactly $g_m$ groups with probability $\ell_m^{\floor{g_m}}/\ell_m^{\overline{a}_{m}}$. 
	
	For the last term in \eqref{eqnLemmaLawAfterSplittingEventUsingPartitions1} we have
	\begin{equation}\label{eqnLemmaLawAfterSplittingEventUsingPartitions2}
		\begin{split}
			\E^{(k)}_r&\left[g_{k,T-t}e^{-\bm \theta\cdot \bo{Z}_{T-t}}\Big|\tau_1<\epsilon,\mathcal{P}_{\tau_1}=\bo{P},\bo{L}_{\tau_1}=\bm{\ell}, B_{[0,\epsilon)}\right]\\
			& = \E^{(k)}_r\left[\mathbf{1}_{\{\tau_1<\epsilon,\mathcal{P}_{\tau_1}=\bo{P},\bo{L}_{\tau_1}=\bm{\ell}, B_{[0,\epsilon)}\}}\E^{(k)}_r\left[\left. g_{k,T-t}e^{-\bm \theta\cdot \bo Z_{T-t}} \right|\F^{(k)}_{\epsilon} \right]\right]\\
			& \hspace{5cm} \times \frac{1}{\p^{(k)}_r\left(\tau_1<\epsilon,\mathcal{P}_{\tau_1}=\bo{P},\bo{L}_{\tau_1}=\bm{\ell}, B_{[0,\epsilon)}\right)}.
		\end{split}	
\end{equation}
Let us now handle the term $\E^{(k)}_r\left[\left. g_{k,T-t}e^{-\bm \theta\cdot \bo Z_{T-t}} \right|\F^{(k)}_{\epsilon} \right]$ under the event $B_{[0,\epsilon)}$ and prove that it is equal to
\[
\begin{split}
\prod_{\substack{m\in [d]\\ g_m\neq 0}}\paren{\frac{ \bm{\ell}\cdot \bm \xi}{\xi_m}}^{\overline{a}_{m}} \left( \prod_{m=1}^d\E_m\left[e^{-\bm \theta \cdot \bo Z_{T-t-\epsilon}}\right]^{\ell_m-g_m}\right)\prod_{q=1}^{g_m}\E_m\left[N^{\floor{a_{m,q}}}_{T-t-\epsilon}e^{-\bm \theta \cdot \bo Z_{T-t-\epsilon}}\right],
		\end{split}
\]
which implies that the computation in  \eqref{eqnLemmaLawAfterSplittingEventUsingPartitions2} is identical to  the latter. Indeed, from \eqref{eqnDecompositionOfGLemmaOfPartitions} with $I=[0,\epsilon)$ and under the event $\{\tau_1<\epsilon,\mathcal{P}_{\tau_1}=\bo{P},\bo{L}_{\tau_1}=\bm{\ell},B_{[0,\epsilon)}\}$, we have the following decomposition of $g_{k,T-t}$,
	\begin{align*}
		g_{k,T-t}= \widetilde{g}_{k,I}\prod_{\substack{m\in [d]\\g_m\neq 0}}\prod_{q=1}^{g_m}\widehat{g}_{A_{m,q},[\epsilon,T-t]}.	\end{align*}
	Next from \eqref{decompZ}, we rewrite $\bo{Z}_{T-t}$ as follows
	\[
\begin{split}
	\bo{Z}_{T-t}=\sum_{m=1}^d\sum_{q=1}^{g_m}\bo Z_{[\epsilon,T-t],m,q}+\sum_{m=1}^d\sum_{q=1}^{\ell_m-g_m}\widetilde {\bo Z}_{[\epsilon,T-t],m,q}.
\end{split}
\]
Hence, under $\{\tau_1<\epsilon,\mathcal{P}_{\tau_1}=\bo{P},\bo{L}_{\tau_1}=\bm{\ell}, B_{[0,\epsilon)}\}$, we  may decompose $g_{k,T-t}e^{-\bm \theta \cdot \bo{Z}_{T-t}}$ as follows
	\[
	g_{k,T-t}e^{-\bm \theta \cdot \bo{Z}_{T-t}}=\widetilde{g}_{k,I} \left(\prod_{m=1}^d\prod_{q=1}^{\ell_m-g_m}e^{-\bm \theta \cdot \widetilde {\bo Z}_{[\epsilon,T-t],m,q}}\right) \prod_{\substack{m\in [d]\\ g_m\neq 0}}\prod_{q=1}^{g_m}\widehat{g}_{A_{m,q},[\epsilon,T-t]}e^{-\bm \theta\cdot\bo Z_{[\epsilon,T-t],m,q}}.
	\]Thus, fixing any $m\in [d]$ with $g_m\neq 0$ and $A_{m,q}$, for $q\in[g_m]$, and then using the Markov property under $\p^{(k)}_r$, under the event $\{\tau_1<\epsilon,\mathcal{P}_{\tau_1}=\bo{P},\bo{L}_{\tau_1}=\bm{\ell}, B_{[0,\epsilon)}\}$, we have
	\begin{align*}
		 \E^{(k)}_r\left[\widehat{g}_{A_{m,q},[\epsilon,T-t]}e^{-\bm \theta\cdot{\bo Z}_{[\epsilon,T-t],m,q}}\Big|\F^{(k)}_{\epsilon}\right] & = \E^{(a_{m,q})}_m\left[g_{a_{m,q},T-t-\epsilon}\,e^{-\bm \theta\cdot\bo{Z}_{T-t-\epsilon}}\right]\\
		 &=\E^{(a_{m,q})}_m\left[N^{\floor{a_{m,q}}}_{T-t-\epsilon}e^{-\bm \theta\cdot\bo{Z}_{T-t-\epsilon}}\right].
	\end{align*}
	Moreover, using again the Markov property under the same event, we deduce 
	\[
	\begin{split}
	\E^{(k)}_r&\left[\widetilde{g}_{k,I} \left(\prod_{m=1}^d\prod_{q=1}^{\ell_m-g_m}e^{-\bm \theta \cdot \widetilde{\bo Z}_{[\epsilon,T-t],m,q}}\right) \Bigg|\F^{(k)}_{\epsilon}\right]\\
	&\hspace{4cm}=\prod_{\substack{m\in [d]\\ g_m\neq 0}}\paren{\frac{ \bm{\ell}\cdot \bm \xi}{\xi_m}}^{\overline{a}_{m}} \left(\prod_{m=1}^d\E_m\left[e^{-\bm \theta \cdot \bo {Z}_{T-t-\epsilon}}\right]^{\ell_m-g_m}\right).
	\end{split}
	\]
	Hence, under the same event, the conditional expectation in \eqref{eqnLemmaLawAfterSplittingEventUsingPartitions2} satisfies
	\begin{align*}
		\E^{(k)}_r\left[\left. g_{k,T-t}e^{-\bm \theta\cdot \bo Z_{T-t}} \right|\F^{(k)}_{\epsilon} \right]
		& = \prod_{\substack{m\in [d]\\ g_m\neq 0}}\paren{\frac{ \bm{\ell}\cdot \bm \xi}{\xi_m}}^{\overline{a}_{m}} \left(\prod_{m=1}^d\E_m\left[e^{-\bm \theta \cdot \bo {Z}_{T-t-\epsilon}}\right]^{\ell_m-g_m}\right)\\
		& \hspace{3cm} \times \prod_{\substack{m\in [d]\\ g_m\neq 0}}\prod_{q=1}^{g_m}\E^{(a_{m,q})}_m\left[N^{\floor{a_{m,q}}}_{T-t-\epsilon}e^{-\bm \theta\cdot\bo{Z}_{T-t-\epsilon}}\right],
	\end{align*}
	as expected. The latter proves our claim which is below \eqref{eqnLemmaLawAfterSplittingEventUsingPartitions2}.

	Finally, putting all pieces together (i.e. in \eqref{problemma7} we use  \eqref{eqnLemmaLawAfterSplittingEventUsingPartitions1} \eqref{eqnProbabilityOfAGivenPartition}, \eqref{eqnLemmaLawAfterSplittingEventUsingPartitions2}),  we obtain
	
	\[
	\begin{split}
		\Q^{(k),\bm \theta }_{T,r}&\left(\left.\tau_1<t+\epsilon,\mathcal{P}_{\tau_1}=\bo{P},\bo{L}_{\tau_1}=\bm{\ell}, B_{[t,t+\epsilon)} \right| \chi_\emptyset>t\right) \\
		& = \frac{1}{\E^{(k)}_r\left[N^{\floor{k}}_{T-t}e^{-\bm \theta\cdot \bo Z_{T-t}}\right]}\p^{(k)}_r\left(\chi_\emptyset<\epsilon,\bo{L}_{\emptyset}=\bm{\ell}, B_{[0,\epsilon)}\right)\paren{\prod_{\substack{m\in [d]\\ g_m\neq 0}}\paren{\frac{\xi _m}{\bm \ell \cdot \bm \xi}}^{\overline{a}_{m}}\ell_m^{\floor{g_m}}}\\
		& \hspace{3cm} \times \prod_{\substack{m\in [d]\\ g_m\neq 0}}\paren{\frac{ \bm{\ell}\cdot \bm \xi}{\xi_m}}^{\overline{a}_{m}} \left(\prod_{m=1}^d\E_m\left[e^{-\bm \theta \cdot \bo {Z}_{T-t-\epsilon}}\right]^{\ell_m-g_m}\right)\\
		& \hspace{7cm} \times \prod_{\substack{m\in [d]\\ g_m\neq 0}}\prod_{q=1}^{g_m}\E^{(a_{m,q})}_m\left[N^{\floor{a_{m,q}}}_{T-t-\epsilon}e^{-\bm \theta\cdot\bo{Z}_{T-t-\epsilon}}\right]\\
		& = p_r(\bm{\ell})\bm  \ell^{\floor{\bo g}}
		\left(\prod_{m=1}^d\E_m\left[e^{-\bm \theta \cdot \bo {Z}_{T-t-\epsilon}}\right]^{\ell_m-g_m}\right)\\
		& \qquad \times
		\frac{\prod_{\substack{m\in [d]\\ g_m\neq 0}}\prod_{q=1}^{g_m}\E^{(a_{m,q})}_m\left[N^{\floor{a_{m,q}}}_{T-t-\epsilon}e^{-\bm \theta\cdot\bo{Z}_{T-t-\epsilon}}\right]}{\E^{(k)}_r\left[N^{\floor{k}}_{T-t}e^{-\bm \theta\cdot \bo Z_{T-t}}\right]}\paren{\alpha_r\epsilon+o(\epsilon)},
	\end{split}
	\]
where the latter follows from
	\[
	\p^{(k)}_r\left(\chi_\emptyset<\epsilon,\bo{L}_{\emptyset}=\bm{\ell}, B_{[0,\epsilon)}\right)=p_r(\bm{\ell})\big(\alpha_r\epsilon+o(\epsilon)\big).
	\] 
	This  completes the proof.
\end{proof}

Note that the formula obtained in the previous lemma only depends on the partition $\bo P$ through the sequences $\bo g$ and $(a_{m,q})_{m\in [d], q\in [g_m]}$.
Hence, as a corollary, we can obtain a similar result but now summing over all possible partitions having block sizes $(a_{m,q})_{m\in[d],q\in [g_m]}$.
Indeed, recall that $d_{m,n}:=\textrm{card}\{q:a_{m,q}=n\}$. Then, the number of colored partitions $(P_m)_{m\in [d]}$ such that each  $P_m$ is made of $g_m$ blocks of sizes $(a_{m,q})_{q\in [g_m]} $  is
\begin{equation}\label{eqnNumberOfPartitionsHavingGivenblockSizes}
\frac{k!}{\prod_{m=1}^d\left(\prod_{q=1}^{g_m} a_{m,q}!\prod_{n\geq 1}d_{m,n}!\right)}.
\end{equation}Note that this is equivalent to ordering the urns  as $A_{1,1},A_{1,2},\ldots, A_{1,g_1},A_{2,1},\ldots, A_{2,g_2},\ldots, A_{d,g_d}$ and counting the number of ways to distribute the $k$ marks among them so that each  urn $(m,q)$ contains exactly $a_{m,q}$ marks, for all $m\in [d],q\in[g_m]$, while disregarding the order of the marks within each urn.

Under the event $\{\bo L_\emptyset=\bm \ell,\chi_\emptyset=\tau_1\}$, let $\# \mathcal{P}_{\tau_1}:=(a_{m,q})_{m\in[d],q\in [g_m]}$ denote the \emph{block sizes of the colored partition} $\mathcal{P}_{\tau_1}$.
Then, summing over all possible partitions having the size $(a_{m,q})_{m\in[d],q\in [g_m]}$, we have the following corollary.
\begin{coro}\label{coroGeneralRateWhenABranchingEventOccursKnowingInfo}
	Conditional on $\{\chi_\emptyset>t\}$, the $\Q^{(k),\bm \theta }_{T,r}$-conditional probability that at time $t$ there are not births-off the spine, the  particle  carrying all marks dies and gives birth to $\bm \ell\in \z^d_+\setminus\{\bo 0\} $ offspring, and the marks follow a coloured partition with block sizes $(a_{m,q})_{m\in[d], q\in [g_m]}$ with $\sum_m\sum_qg_{m,q}=k$, is given by
	\[
	\begin{split}
		\Q^{(k),\bm \theta }_{T,r}&\left(\left.\tau_1\in {\rm d} t,\#\mathcal{P}_{\tau_1}=(a_{m,q})_{m\in[d], q\in [g_m]},\bo{L}_{\tau_1}=\bm{\ell},  \right| \chi_\emptyset>t \right)\\
\\
& =\frac{k!}{\prod_{m=1}^d\left(\prod_{q=1}^{g_m} a_{m,q}!\prod_{n\geq 1}d_{m,n}!\right)}\E_r\left[\bo L^{\floor{\bo g}}\prod_{m=1}^d\E_m\left[e^{-\bm \theta \cdot \bo Z_{T-t}}\right]^{L^{(m)}-g_m}\right]
\\
		&\hspace{1cm}\times\frac{\bm \ell^{\floor{\bo g}} p_r(\bm{\ell})
		\prod_{m=1}^d\E_m\left[e^{-\bm \theta \cdot \bo Z_{T-t}}\right]^{\ell_m-g_m}}{\E_r\left[\bo L^{\floor{\bo g}}\prod_{m=1}^d\E_m\left[e^{-\bm \theta \cdot \bo Z_{T-t}}\right]^{L^{(m)}-g_m}\right]}
		\frac{\prod_{\substack{m\in [d]\\ g_m\neq 0}}\prod_{q=1}^{g_m}\E_m\left[N^{\floor{a_{m,q}}}_{T-t}e^{-\bm \theta \cdot \bo Z_{T-t}}\right]}{\E_r\left[N^{\floor{k}}_{T-t}e^{-\bm \theta\cdot \bo Z_{T-t}}\right]} {\rm d} t .
	\end{split}
	\]
\end{coro}

Now, let us describe the rate of those vertices carrying no marks and the law of their offspring.
\begin{lemma}\label{lemmaLawOfIndividualsWithNoMarks}
Under $\Q^{(k),\bm \theta }_{T,r}$,   the rate of a branching event for a vertex of type $i$ with no marks at  time $t$ is given by 
\[
\alpha_i\frac{\E_i\left[\prod_{m=1}^d\E_m\left[ e^{-\bm \theta\cdot \bo Z_{T-t} }\right]^{ L^{(m)}} \right]}{\E_i\left[e^{-\bm \theta\cdot \bo Z_{T-t} }\right]}{\rm d} t,
\]
and the probability that there are $\bm \ell$ offspring satisfies
\[
p_i(\bm \ell)\frac{\prod_{m=1}^d\E_m\left[ e^{-\bm \theta\cdot \bo Z_{T-t} }\right]^{\ell_m}}{\E_i\left[\prod_{m=1}^d\E_m\left[ e^{-\bm \theta\cdot \bo Z_{T-t} }\right]^{ L^{(m)}} \right]}.
\]
\end{lemma}
\begin{proof} Let us consider a vertex $v$ of type $i$  at time $t$.  From Lemma \ref{lemmaMrkovPropertyUnderQ}, we have that the subtree generated by $v$ after time $t$ is independent of the system and behaves as if under $\p^{\bm \theta}_{T-t, i}$.

Recall that $\chi_v$ and $\bo L_v$ denote the lifetime and the offspring distribution of the vertex $v$. Thus in order to deduce the result, we are interested in the  probability of the event $\{\chi_v<\epsilon, \bo L_v= \bm \ell\}$, for $\epsilon>0$ and $\bm \ell \in \mathbb{Z}^{d}_+$. 

From the definition of $\p^{\bm \theta}_{T-t, i}$ in \eqref{eqnChangeOfMeasureMtypeBGWDiscounted}, we see 
\[
\p^{\bm \theta}_{T-t, i}\paren{\chi_v<\epsilon, \bo L_v= \bm \ell}=\frac{\E_i\left[\mathbf{1}_{\{\chi_v<\epsilon, \bo L_v= \bm \ell\}}e^{-\bm \theta\cdot \bo Z_{T-t} }\right]}{\E_i\left[e^{-\bm \theta\cdot \bo Z_{T-t} }\right]}.
\]
We first deal with the numerator, 
\begin{align*}
\E_i\left[\mathbf{1}_{\{\chi_v<\epsilon, \bo L_v= \bm \ell\}}e^{-\bm \theta\cdot \bo Z_{T-t} }\right]& =\p_i\paren{\chi_v<\epsilon, \bo L_v= \bm \ell}\E_i\left[\left. e^{-\bm \theta\cdot \bo Z_{T-t} }\right| \chi_v<\epsilon, \bo L_v= \bm \ell\right]\\
&=(\alpha_i \epsilon+o(\epsilon))p_i\paren{\bm \ell}\E_i\left[\left. e^{-\bm \theta\cdot \bo Z_{T-t} }\right| \chi_v<\epsilon, \bo L_v= \bm \ell\right]\\
&=(\alpha_i \epsilon+o(\epsilon))p_i\paren{\bm \ell}\prod_{m=1}^d\E_m\left[ e^{-\bm \theta\cdot \bo Z_{T-t-h} }\right]^{\ell_m}\\
&=\alpha_i\epsilon p_i\paren{\bm \ell}\prod_{m=1}^d\E_m\left[ e^{-\bm \theta\cdot \bo Z_{T-t-h} }\right]^{\ell_m}+o(\epsilon),
\end{align*}
where in the penultimate line we have used the branching property. Thus
\[
\p^{\bm \theta}_{T-t, i}\paren{\chi_v<\epsilon, \bo L_v= \bm \ell}=\alpha_i\epsilon p_i\paren{\bm \ell}\frac{\prod_{m=1}^d\E_m\left[ e^{-\bm \theta\cdot \bo Z_{T-t-h} }\right]^{\ell_m}}{\E_i\left[e^{-\bm \theta\cdot \bo Z_{T-t} }\right]}+o(\epsilon).
\]
The previous asymptotic clearly  implies the result. 
\end{proof}
Next, we  compute the joint law of the first splitting event together with its   offspring,  the partition generated by the spines, and the color of the individual previous to the splitting event. This result is quite useful for the description of  the whole genealogy of the tree of the $k$-sample under $\Q^{(k),\bm \theta}_{T,r}$.

\begin{propo}\label{propoFirstSplittingEventWithPartitionChildAndType}
Fix $k\geq 2$. 
Then we have 
\begin{equation}\label{eqnpropoFirstSplittingEvent}
    \begin{split}
    \Q^{(k),\bm \theta}_{T,r}&\left(\tau_1\in {\rm d}t,\mathcal{P}_{\tau_1}=\bo{P},\bo{L}_{\tau_1}=\bm{\ell} , c(\varsigma^{(1)}_{t-})=i\right)\\
        & =\alpha_{i}\E_i\left[\bo L^{\floor{\bo g}}\prod_{m=1}^d\E_m\left[e^{-\bm \theta \cdot \bo Z_{T-t}}\right]^{L^{(m)}-g_m}\right] \frac{p_{i}(\bm \ell)\bm \ell^{\floor{\bo g}}\prod_{m=1}^d \E_m\left[e^{-\bm \theta \cdot \bo Z_{T-t}}\right]^{\ell_m-g_m}}{\E_i\left[\bo L^{\floor{\bo g}}\prod_{m=1}^d\E_m\left[e^{-\bm \theta \cdot \bo Z_{T-t}}\right]^{L^{(m)}-g_m}\right]}\\
&\hspace{1.2cm} \times \E_r\left[ Z^{(i)}_{t}\prod_{m=1}^d\E_m\left[e^{-\bm \theta \cdot \bo Z_{T-t}}\right]^{Z^{(m)}_{t}-\delta_{i,m}}\right]\frac{\prod_{m=1}^d\prod_{q=1}^{g_m} \E_m\left[N^{\floor{a_{m,q}}}_{T-t}e^{-\bm \theta \cdot \bo Z_{T-t}}\right]}{\E_r\left[N^{\floor{k}}_{T}e^{-\bm \theta\cdot \bo Z_{T}}\right]}{\rm d} t.
    \end{split}
\end{equation}
\end{propo}
\begin{proof}

Let $\epsilon>0$ and recall that $B_{[t,t+\epsilon)}$ denotes the event that only one branching event occurs in  $[t,t+\epsilon)$.   From \eqref{eqnDecompositionOfGLemmaOfPartitions} with $I=[t, t+\epsilon)$, we observe that under the event 
\[
\{\tau_1\in I,\mathcal{P}_{\tau_1}=\bo{P},\bo{L}_{\tau_1}=\bm{\ell},c(\varsigma^{(1)}_{t-})=i \}\cap B_{[t,t+\epsilon)},
\]  
the following decomposition for $g_{k,T}$ holds
\begin{align*}
	g_{k,T} &= \overline{g}_{k,[0,t)} \widetilde{g}_{k,I}\prod_{\substack{m\in [d]\\g_m\neq 0}}\prod_{q=1}^{g_m}\widehat{g}_{A_{m,q},[t+\epsilon,T]}.
\end{align*}
Similarly, $\bo Z_T$ can be decomposed as in \eqref{decompZ}, that is
\[
\begin{split}
	\bo{Z}_T=\sum_{m=1}^d\sum_{q=1}^{Z^{(m)}_{t-}-\delta_{i,m}}\overline{\bo Z}_{[t+\epsilon,T],m,q}+\sum_{m=1}^d\sum_{q=1}^{g_m}\bo Z_{[t+\epsilon,T],m,q}+\sum_{m=1}^d\sum_{q=1}^{\ell_m-g_m}\widetilde {\bo Z}_{[t+\epsilon,T],m,q}.
\end{split}
\]
Next, we use the definition of the $\Q^{(k),\bm \theta}_{T, r}$ measure and then we condition with respect to $\F^{(k)}_{t+\epsilon}$ and obtain
\begin{equation}\label{eqnpropoFirstSplittingEventWithPartitionChildAndType}
\begin{split}
	 \E_r\left[N^{\floor{k}}_{T}e^{-\bm \theta\cdot \bo Z_{T}}\right]&\Q^{(k),\bm \theta}_{T, r}\left[\tau_1\in I,\mathcal{P}_{\tau_1}=\bo{P},\bo{L}_{\tau_1}=\bm{\ell} , c(\varsigma^{(1)}_{t-})=i, B_{[t,t+\epsilon)}\right]\\
  &=\E^{(k)}_r\left[g_{k,T}e^{-\bm \theta\cdot \bo Z_{T}}\mathbf{1}_{\{\tau_1\in I,\mathcal{P}_{\tau_1}=\bo{P},\bo{L}_{\tau_1}=\bm{\ell} , c(\varsigma^{(1)}_{t-})=i, B_{[t,t+\epsilon)}\}}\right]\\
	&= \E^{(k)}_r\left[\mathbf{1}_{\{\tau_1\in I,\mathcal{P}_{\tau_1}=\bo{P},\bo{L}_{\tau_1}=\bm{\ell} , c(\varsigma^{(1)}_{t-})=i, B_{[t,t+\epsilon)}\}} \E^{(k)}_r\left[\left.g_{k,T}e^{-\bm \theta\cdot \bo Z_{T}} \right|\F^{(k)}_{t+\epsilon} \right]\right].
\end{split}
\end{equation}
By the decompositions of $g_{k,T}$ and $\bo Z_T$ provided above, we have that the conditional expectation in the above display,  under the event $\{\tau_1\in I,\mathcal{P}_{\tau_1}=\bo{P},\bo{L}_{1}=\bm{\ell} , c(\varsigma^{(1)}_{t-})=i\}\cap B_{[t,t+\epsilon)}$,  satisfies
\begin{equation}\label{eqnpropoFirstSplittingEventWithPartitionChildAndType2}
\begin{split}
	 \E^{(k)}_{r}&\left[\left.g_{k,T}e^{-\bm \theta \cdot \bo Z_T}\right|\F^{(k)}_{t+\epsilon}\right] = \overline{g}_{k,[0,t)}
	\E_r\left[e^{-\sum_{m=1}^d\sum_{q=1}^{Z^{(m)}_{t-}-\delta_{i,m}}\bm \theta\cdot\overline{\bo Z}_{[t+\epsilon,T],m,q}}\bigg|\F^{(k)}_{t+\epsilon}\right]\\
	&\hspace{1.5cm}\times\widetilde{g}_{k,I}\prod_{\substack{m\in [d]\\ g_m\neq 0}}\E^{(k)}_r\left[e^{-\sum_{q=1}^{g_m}\bm \theta\cdot\bo Z_{[t+\epsilon,T],m,q}-\sum_{q=1}^{\ell_m-g_m}\bm \theta\cdot\widetilde {\bo Z}_{[t+\epsilon,T],m,q}}\prod_{q=1}^{g_m}\widehat{g}_{A_{m,q},[t+\epsilon,T]}\bigg|\F^{(k)}_{t+\epsilon}\right]\\
	&=\overline{g}_{k,[0,t)}\prod_{m=1}^d\E_m\left[e^{-\bm \theta \cdot \bo Z_{T-t-\epsilon}}\right]^{Z^{(m)}_{t-}-\delta_{i,m}} \\
	&\hspace{2cm}\times \widetilde{g}_{k,I}\prod_{\substack{m\in [d]\\ g_m\neq 0}}\prod_{q=1}^{g_m} \E_m\left[N^{\floor{a_{m,q}}}_{T-t-\epsilon}e^{-\bm \theta \cdot \bo Z_{T-t-\epsilon}}\right] \E_m\left[e^{-\bm \theta \cdot \bo Z_{T-t-\epsilon}}\right]^{\ell_m-g_m}.
\end{split}
\end{equation} 
Now, under the event $\{\tau_1\in I,\mathcal{P}_{\tau_1}=\bo{P},\bo{L}_{\tau_1}=\bm{\ell} , c(\varsigma^{(1)}_{t-})=i\}\cap B_{[t,t+\epsilon)}$, we consider the expectation of the random terms in \eqref{eqnpropoFirstSplittingEventWithPartitionChildAndType2}, i.e. the first two terms on the right-hand side in the last identity. Thus, conditioning with respect to $\F_{t+\epsilon}$ and summing over $\mathcal{N}^{(i)}_{t-}$,  the set of all individuals of type $i$ in $\mathcal{N}_{t-}$, we get
\begin{equation}
    \begin{split}
   \E^{(k)}_r&\left[\mathbf{1}_{\{\tau_1\in I,\mathcal{P}_{\tau_1}=\bo{P},\bo{L}_{\tau_1}=\bm{\ell} , c(\varsigma^{(1)}_{t-})=i\}\cap B_{[t,t+\epsilon)}}\overline{g}_{k,[0,t)}\prod_{m=1}^d\E_m\left[e^{-\bm \theta \cdot \bo Z_{T-t-\epsilon}}\right]^{Z^{(m)}_{t-}-\delta_{i,m}} \right]\\
&\hspace{1cm}=\E^{(k)}_r\left[\mathbf{1}_{\{\tau_1\in I,\mathcal{P}_{\tau_1}=\bo{P},\bo{L}_{\tau_1}=\bm{\ell} , c(\varsigma^{(1)}_{t-})=i\}\cap B_{[t,t+\epsilon)}}  \mathbf{1}_{\{ \varsigma^{(h_1)}_{t-}= \varsigma^{(h_2)}_{t-},h_1,h_2\in [k]\}} \right.\\
& \hspace{4cm} \times \left. \left(\prod_{h\in [k]}\prod_{\text{$(w,c_w)\in {\rm spine}(\varsigma^{(h)}_{t})$}\atop\text{${\rm gen}(w)\in [0,t)$}}\frac{ \bo{L}_{w}\cdot \bm \xi}{\xi_{c_w}}\right) \prod_{m=1}^d\E_m\left[e^{-\bm \theta \cdot \bo Z_{T-t-\epsilon}}\right]^{Z^{(m)}_{t-}-\delta_{i,m}}\right]\\
&=\E^{(k)}_r\left[\sum_{ v\in \mathcal{N}^{(i)}_{t-}}\mathbf{1}_{ \{\mathcal{P}_{\tau_1}=\bo{P},\bo{L}_{\tau_1}=\bm{\ell}, c(v)=i\}\cap B_{[t,t+\epsilon)}}
\prod_{\text{$(w,c_w)\in {\rm spine}(v)$}\atop\text{${\rm gen}(w)\in [0,t)$}}\left(\frac{ \bo{L}_{w}\cdot \bm \xi}{\xi_{c_w}}\right)^k \right.\\
& \hspace{3cm} \times \prod_{m=1}^d\E_m\left[e^{-\bm \theta \cdot \bo Z_{T-t-\epsilon}}\right]^{Z^{(m)}_{t-}-\delta_{i,m}}  \p^{(k)}_r\left(\left. \varsigma^{(h)}_{t-}=v,\ h\in [k]\right| \F_{t+\epsilon}\right)  \Bigg].
    \end{split}
\end{equation}
On the other hand, the  conditional expectation of the above equation is such that
\[
\p^{(k)}_r\left(\left. \varsigma^{(h)}_{t-}=v,\ h\in [k]\right| \F_{t+\epsilon}\right)=\prod_{\text{$(w,c_w)\in {\rm spine}(v)$}\atop\text{${\rm gen}(w)\in [0,t)$}}\left(\frac{\xi_{c_w}}{ \bo{L}_{w}\cdot \bm \xi}\right)^k,
\]
 see \eqref{identitypk}. Thus, putting all pieces together and  applying again the Markov property, we deduce
\begin{equation}
    \begin{split}
   \E^{(k)}_r&\left[\mathbf{1}_{\{\tau_1\in I,\mathcal{P}_{\tau_1}=\bo{P},\bo{L}_{\tau_1}=\bm{\ell} , c(\varsigma^{(1)}_{t-})=i\}\cap B_{[t,t+\epsilon)}}\overline{g}_{k,[0,t)}\prod_{m=1}^d\E_m\left[e^{-\bm \theta \cdot \bo Z_{T-t-\epsilon}}\right]^{Z^{(m)}_{t-}-\delta_{i,m}} \right]\\
&\hspace{2cm}=\E^{(k)}_r\left[
  \prod_{m=1}^d\E_m\left[e^{-\bm \theta \cdot \bo Z_{T-t-\epsilon}}\right]^{Z^{(m)}_{t-}-\delta_{i,m}}\right.\\
  &\hspace{3.5cm}\left. \times  \sum_{ v\in \mathcal{N}^{(i)}_{t-}}\mathbb{P}^{(k)}_r\left[\mathcal{P}_{\tau_1}=\bo{P},\bo{L}_{\tau_1}=\bm{\ell}, c(v)=i, B_{[t,t+\epsilon)}\Big| \mathcal{F}_t\right]  \right].
\end{split}
\end{equation}
Finally, under the event that $\{v\in \mathcal{N}^{(i)}_{t-}\}$ we have
\[
\begin{split}
\mathbb{P}^{(k)}_r\left(\mathcal{P}_{\tau_1}=\bo{P},\bo{L}_{\tau_1}=\bm{\ell}, c(v)=i, B_{[t,t+\epsilon)}\Big| \mathcal{F}_t\right)&=\mathbb{P}^{(k)}_i\left(\mathcal{P}_{\epsilon}=\bo{P},\bo{L}_{\emptyset}=\bm{\ell}, B_{[0,\epsilon)}\right)\\
&=p_{i}(\bm \ell)\prod_{\substack{m\in [d]\\ g_m\neq 0}}\left[\frac{\xi _m}{\bm \ell \cdot \bm \xi}\right]^{\overline{a}_{m}}\ell_m^{\floor{g_m}}(1-e^{-\alpha_i\epsilon}),
\end{split}
\]
where the first identity follows from the branching Markov property at time $t$ together with the fact that there is only one branching event during the time interval $[t, t+\epsilon)$.

Observing that  the cardinality of $\mathcal{N}^{(i)}_{t-}$ is $Z^{(i)}_{t-}$, we use a similar argument as in  
\eqref{eqnChangeFromInfoUpTo_t} to deduce
\[
\begin{split}
	\E^{(k)}_r&\left[\mathbf{1}_{\{\tau_1\in I,\mathcal{P}_{\tau_1}=\bo{P},\bo{L}_{\tau_1}=\bm{\ell} , c(\varsigma^{(1)}_{t-})=i\}\cap B_{[t,t+\epsilon)}}\overline{g}_{k,[0,t)}\prod_{m=1}^d\E_m\left[e^{-\bm \theta \cdot \bo Z_{T-t-\epsilon}}\right]^{Z^{(m)}_{t-}-\delta_{i,m}} \right]\\
& \hspace{1cm}=\E^{(k)}_r\left[ Z^{(i)}_{t}\prod_{m=1}^d\E_m\left[e^{-\bm \theta \cdot \bo Z_{T-t-\epsilon}}\right]^{Z^{(m)}_{t}-\delta_{i,m}}\right]p_{i}(\bm \ell)\prod_{\substack{m\in [d]\\ g_m\neq 0}}\paren{\frac{\xi _m}{\bm \ell \cdot \bm \xi}}^{\overline{a}_{m}}\ell_m^{\floor{g_m}}(1-e^{-\alpha_i\epsilon}),
\end{split}
\]
where in the previous identity we have replaced $t-$ by $t$ since the process does not make jumps at fixed times. Thus, putting all pieces together, that is, we use the above formula with \eqref{eqnpropoFirstSplittingEventWithPartitionChildAndType} and \eqref{eqnpropoFirstSplittingEventWithPartitionChildAndType2}, and obtain 
\[
    \begin{split}
\Q^{(k),\bm \theta}_{T, r}&\left(\tau_1\in I,\mathcal{P}_{\tau_1}=\bo{P},\bo{L}_{\tau_1}=\bm{\ell} , c(\varsigma^{(1)}_{t-})=i, B_{[t,t+\epsilon)}\right)\\
& =\frac{\E^{(k)}_r\left[ Z^{(i)}_{t}\prod_{m=1}^d\E_m\left[e^{-\bm \theta \cdot \bo Z_{T-t-\epsilon}}\right]^{Z^{(m)}_{t}-\delta_{i,m}}\right]}{\E_r\left[N^{\floor{k}}_{T}e^{-\bm \theta\cdot \bo Z_{T}}\right]}(1-e^{-\alpha_i\epsilon})p_{i}(\bm \ell)\prod_{\substack{m\in [d]\\ g_m\neq 0}}\paren{\frac{\xi _m}{\bm \ell \cdot \bm \xi}}^{\overline{a}_{m}}\ell_m^{\floor{g_m}}\\
& \hspace{3cm} \times \widetilde{g}_{k,I}\prod_{\substack{m\in [d]\\ g_m\neq 0}}\prod_{q=1}^{g_m} \E_m\left[N^{\floor{a_{m,q}}}_{T-t-\epsilon}e^{-\bm \theta \cdot \bo Z_{T-t-\epsilon}}\right] \E_m\left[e^{-\bm \theta \cdot \bo Z_{T-t-\epsilon}}\right]^{\ell_m-g_m}\\
&=\frac{\E_r\left[Z^{(i)}_{t}\prod_{m=1}^d\E_m\left[e^{-\bm \theta \cdot \bo Z_{T-t-\epsilon}}\right]^{Z^{(m)}_{t}-\delta_{i,m}}\right]}{\E_r\left[N^{\floor{k}}_{T}e^{-\bm \theta\cdot \bo Z_{T}}\right]}(1-e^{-\alpha_i\epsilon})p_{i}(\bm \ell)\bm\ell^{\floor{\bo g}}\\\
& \hspace{3cm} \times \prod_{\substack{m\in [d]\\ g_m\neq 0}}\prod_{q=1}^{g_m} \E_m\left[N^{\floor{a_{m,q}}}_{T-t-\epsilon}e^{-\bm \theta \cdot \bo Z_{T-t-\epsilon}}\right] \E_m\left[e^{-\bm \theta \cdot \bo Z_{T-t-\epsilon}}\right]^{\ell_m-g_m}.\\
    \end{split}
\]Note that we changed from $\E^{(k)}_r$ to $\E_r$ in the last equality since the event under consideration does not involve the marks. 
The latter clearly implies our results by dividing by $\epsilon$ and taking limits as $\epsilon$ goes to 0; and by multiplying and dividing the quantity 
\[
\E_i\left[\bo L^{\floor{\bo g}}\prod_{m=1}^d\E_m\left[e^{-\bm \theta \cdot \bo Z_{T-t}}\right]^{L^{(m)}-g_m}\right]. 
\]
Recall that the limit can be taken inside the expectations since $ \bo Z$ is a Feller process with  c\`adl\`ag paths and by the dominated convergence theorem. This completes the proof.
\end{proof}



The previous result implies the following corollary that computes the  $\Q^{(k),\bm \theta }_{T,r}$-probability that  the  particle  carrying all marks dies at time $t$ and gives birth to $\bm \ell\in \z^d_+\setminus\{\bo 0\} $ offspring, and the marks follow a coloured partition with block sizes $(a_{m,q})_{m\in[d], q\in [g_m]}$.

\begin{coro}\label{coroSplittingEventWithPartitionChildAndType}Fix $k\geq 2$. 
Then we have 
	\begin{align*} 
	\Q^{(k),\bm \theta}_{T,r}&\left(\tau_1\in {\rm d}t, \#\mathcal{P}_{\tau_1}=(a_{m,q})_{m\in[d], q\in [g_m]}, {\bo L}_{\tau_1}=\bm{\ell} , c(\varsigma^{(1)}_{t-})=i\right)\\
	&=\frac{k!}{\prod_{m=1}^d\overline a_{m}!}\prod_{m=1}^d{\ell_m \choose g_m}\frac{\overline a_{m}!}{\prod_{q=1}^{g_m} a_{m,q}!}\frac{g_m!}{\prod_{n\geq 1}d_{m,n}!}\E_i\left[\bo L^{\floor{\bo g}}\prod_{m=1}^d\E_m\left[e^{-\bm \theta \cdot \bo Z_{T-t}}\right]^{L^{(m)}-g_m}\right] \\
	&\hspace{1cm}\times\frac{p_{i}(\bm \ell)\prod_{m=1}^d \E_m\left[e^{-\bm \theta \cdot \bo Z_{T-t}}\right]^{\ell_m-g_m}}{\E_i\left[\bo L^{\floor{\bo g}}\prod_{m=1}^d\E_m\left[e^{-\bm \theta \cdot \bo Z_{T-t}}\right]^{L^{(m)}-g_m}\right]}\E_r\left[ Z^{(i)}_{t}\prod_{m=1}^d\E_m\left[e^{-\bm \theta \cdot \bo Z_{T-t}}\right]^{Z^{(m)}_{t}-\delta_{i,m}}\right]\\
&\hspace{6.5cm} \times \frac{\prod_{m=1}^d\prod_{q=1}^{g_m} \E_m\left[N^{\floor{a_{m,q}}}_{T-t}e^{-\bm \theta \cdot \bo Z_{T-t}}\right]}{\E_r\left[N^{\floor{k}}_{T}e^{-\bm \theta\cdot \bo Z_{T}}\right]}\alpha_{i}{\rm d} t.\\
\end{align*}
\end{coro}
With all these results in hand, we are now ready to provide a complete description of the evolution of the process  under $\Q^{(k),\bm \theta}_{T,r}$. In other words, we can complete the proof of the forward construction.
\begin{proof}[Proof of Proposition \ref{forwconstr}] As we mentioned before, the proof follows from previous intermediate results. More precisely, we start with a particle  with $k$ marks.  Step (2) follows from the Markov branching property, see Lemma \ref{lemmaMrkovPropertyUnderQ}. Proposition \ref{propoFirstSplittingEventWithPartitionChildAndType} establishes the joint distribution of the first spine splitting event, given that the spine had type 
$i$ along with a specific partition choice and $\bm \ell$ offspring. Meanwhile, Lemma \ref{lemmaGeneralRateWhenABranchingEventOccursKnowingInfo} together with Remark \ref{remark1} provide the rates of births off the spine. Both results correspond to Step (3). Corollary \ref{coroSplittingEventWithPartitionChildAndType} computes the probability that particle  carrying all marks dies at time $t$ and gives birth to $\bm \ell\in \z^d_+\setminus\{\bo 0\} $ offspring, and the marks follow a coloured partition with block sizes $(a_{m,q})_{m\in[d], q\in [g_m]}$ explaining  Step (4). Finally  Step (5)
follows from Lemma \ref{lemmaLawOfIndividualsWithNoMarks}.
\end{proof}

\subsection{Joint law of spine splitting times}\label{subsectionJointLawSpineSplittingTimes}

This subsection  generalises Proposition \ref{propoFirstSplittingEventWithPartitionChildAndType} by deriving  the joint law of the spine splitting times,  the associated partition process of $[k]$, the types of  individuals giving birth immediately before each splitting times  and  their offspring configurations.

Recall from \eqref{deltatn}  that  $M$ denotes the number of splitting events required for the 
$k$ marks, initially  carried by a single individual, to separate into 
$k$ distinct individuals, each carrying   one mark.  Fix $n\leq k-1$, and consider the event $\{M=n\}$. Let $0<\tau_1<\cdots <\tau_n<T$, denote the successive splitting times  before time $T$.  At each splitting time $\tau_h$, a single individual on the spine, whose vertex we denote by $v(h)$, gives birth, producing $\bo L_{v(h)}$ new offspring. For each $h\in[n]$, let   $ C_h$ denote the type of the spine individual $v(h)$, and let $\mathcal{P}_h$ be the corresponding partition of $[k]$. We write
 $\bo P_h=(P_{h,1},\ldots, P_{h,d})$, where for each $m\in [d]$, $P_{h,m}=\{A_{h,m,q}\}_{q\in [g_{h,m}] }$, with $g_{h,m}$  the number of blocks of type $m$ created at the $h$-th splitting event. Set $\bo g_h=(g_{h,1},\ldots,g_{h,d})$.

For convenience, recall the event,  
 \[
\Delta_T(n)=\bigcap_{h\in [n]}\left\{\tau_h\in {\rm d}t_h,\mathcal{P}_{h}=\bo{P}_h,\bo{L}_{v(h)}=\bm{\ell}_h , C_{h}=i_h, M=n \right\},
\]
where $0<t_1<t_2<\cdots <t_n<T$. Here $\bm{\ell}_h=(\ell_{h,1},\ldots,\ell_{h,d})\in\mathbb{Z}_+^d$ denotes the offspring of  $v(h)$ and satisfies $\bo g_h\le \bm{\ell}_h$ componentwise, since each block corresponds to at least one child. The parent type is $i_h\in[d]$.

Finally, for each block $A_{h,m,q}$, we recall  that $k_{v(h, m,q)}=\textrm{card}\{A_{h,m,q}\}$, the number of marks carried by  the descendant of $v(h)$ of type $m$ immediately after time $\tau_h$. This descendant is denoted by $v(h,m,q)$. 

Our next result  characterises, under $\Q^{(k),\bm \theta }_{T,r}$,  the joint distribution of the splitting times, partitions, parental types, and offspring configurations under the conditioning $M=n$.
\begin{propo}\label{propoJointSplittingEventWithPartitionChildAndType}
For $k\geq 2$, $r\in [d]$,  $n\in\{1,\ldots, k-1\}$, and $t_0:=0<t_1<\cdots < t_n<T$, we have
\begin{equation}\label{eqnJointSplittingEventWithPartitionChildAndTypeV2}
	\begin{split}
		\Q^{(k),\bm \theta }_{T,r}&\Big(\Delta_T(n)\Big) = \prod_{h=1}^n\prod_{m=1}^d\E_{m}\left[e^{-\bm \theta \cdot \bo Z_{T-t_h}}\right]^{\ell_{h,m} -g_{h,m}}p_{i_h}(\bm \ell_{h})\bm \ell_{h}^{\floor{\bo g_h}}\\
		&\hspace{.5cm}  \times 
		\prod_{h=0}^{n-1}\prod_{m=1}^d\prod_{\substack{ q\leq g_{h,m}:\\ k_{v(h,m,q)}\geq 2}}\E_{m}\left[Z^{(c(v(h,m,q)))}_{t_{v(h,m,q)}-t_h} \prod_{j=1}^d\E_{j}\left[e^{-\bm \theta\cdot \bo Z_{T-t_{v(h,m,q)}}}\right]^{ Z^{(j)}_{t_{v(h,m,q)}-t_h}-\delta_{c(v(h,m,q)),j}}\right]\\
	& \hspace{3cm} \times
	\frac{\prod_{h=1}^{n}\prod_{m=1}^{d}\E_{m}\left[N_{T-t_{h}}e^{-\bm \theta \cdot \bo Z_{T-t_{h}}}\right]
		^{\#\{ q\leq g_{h,m}:k_{v(h,m,q)}=1\}}
	}{\E_{r}\left[N_{T}^{\floor{k}}e^{-\bm \theta\cdot \bo Z_{T}}\right]} \prod_{h=1}^n\alpha_{i_{h}} {\rm d}t_h,
\end{split}
\end{equation}where  $t_{v(h,m,q)}$ and $c(v(h,m,q))$ are the spine splitting times  and type associated to the vertex $v(h,m,q)$.  When  $h=0$, we observe that $v(0)=\emptyset$, $g_{0,m}=\delta_{r,m}$ and $v(0, r, 1)=v(1)$. 

\end{propo}
\begin{remark}
	We note that the previous formula aligns with the single-type case, see identity (41) in Lemma 3.8 of \cite{MR4718398}. 
	Some terms clearly correspond to the single-type case, for example, the first line in \eqref{eqnJointSplittingEventWithPartitionChildAndTypeV2}, the rates simplifies to $\alpha^n$, as well as $\E_{r}\left[N_{T}^{\floor{k}}e^{-\bm \theta\cdot \bo Z_{T}}\right]$  simplifies to $\E\left[Z_{T}^{\floor{k}}e^{- \theta Z_{T}}\right]$ (note that in the single-type case the total size $N_{T}$ at generation $T$ is simply   $Z_{T}$). 
	Thus, it is enough to observe that (using the obvious notation which does not contains the types)
	\begin{equation}\label{eqnCheckingJointLawConcidesWithUnitype}
	\begin{split}
	\prod_{h=0}^{n-1}\prod_{\substack{\ell\leq g_h\\ k_{v(h,\ell)}\geq 2}}
		&\E\left[Z_{t_{v(h,\ell)}-t_{h}} \E\left[e^{-\theta Z_{T-t_{v(h,\ell)}}}\right]^{Z_{t_{v(h,\ell)}-t_{h}}-1}\right]\\
		&\times\left(\prod_{h=1}^n \E\left[Z_{T-t_{h}}e^{-\theta Z_{T-t_{h}}}\right]
		^{\#\{ \ell\leq g_{h}:k_{v(h,\ell)}=1\}}\right) = F'_T(e^{-\theta}) \prod_{h=1}^n F'_{T-t_h}(e^{-\theta})^{g_h-1}.
\end{split} 
	\end{equation}
	For the proof of the previous identity, we refer to the \refappendix{Appendix}.
	\end{remark}

\begin{proof}
The proof is by induction on $n$. The case $n=1$ follows from Proposition \ref{propoFirstSplittingEventWithPartitionChildAndType} with $\bo{P}_1$ being only blocks with singletons. That is in the first spine splitting events all spines are distinct.  Thus from \eqref{eqnpropoFirstSplittingEvent}, we have
\begin{align*}
	\Q^{(k),\bm \theta }_{T,r}&\Big(\tau_1\in {\rm d}t_1,\mathcal{P}_{1}=\bo{P}_1,\bo{L}_{v(1)}=\bm{\ell}_1, C_{1}=i_1,M=1\Big) = \alpha_{i_1} p_{i_1}(\bm \ell_1)\bm \ell_1^{\floor{\bo g_1}}\\
	& \qquad \times\left(\prod_{m=1}^{d}\E_{m}\left[e^{-\bm \theta \cdot \bo Z_{T-t_1}}\right]^{\ell_{1,m}- g_{1,m}} \right)
	\E_{r}\left[Z^{(i_1)}_{t_1}\prod_{m=1}^{d}\E_{m}\left[e^{-\bm \theta\cdot \bo Z_{T-t_1}}\right]^{ Z^{(m)}_{t_1}-\delta_{i_1,m}}\right]\\
	& \qquad \hspace{6cm}\times
	\frac{\prod_{m=1}^d\E_m\left[N_{T-t_1}e^{-\bm \theta \cdot \bo Z_{T-t_1}}\right]^{g_{1,m}}}{\E_{r}\left[N_{T}^{\floor{k}}e^{-\bm \theta\cdot \bo Z_{T}}\right]}{\rm d}t_1,
\end{align*}
which coincides with identity \eqref{eqnJointSplittingEventWithPartitionChildAndTypeV2}  by observing 
\begin{align*}
 \prod_{m=1}^d\E_m\left[N_{T-t_1}e^{-\bm \theta \cdot \bo Z_{T-t_1}}\right]^{g_{1,m}}=\prod_{m=1}^d\E_m\left[N_{T-t_{1}}e^{-\bm \theta \cdot \bo Z_{T-t_{1}}}\right]^{\#\{ q\in [g_{1,m}]:k_{v(1,m,q)}=1\}},
\end{align*}
since in our case  $k_{v(1,m,q)}=1$ for all $m\in[d]$ and $q\in[g_{1,m}]$.

Next, we consider the case $n\geq 2$. We assume that \eqref{eqnJointSplittingEventWithPartitionChildAndTypeV2} holds for $n-1$, and use Lemma \ref{lemmaMrkovPropertyUnderQ} together with the Markov property, under $\Q^{(k),\bm \theta }_{T,r}$, at time $\tau_1$. 
Indeed, after the first spine splitting event, the spines in each of the $\sum_m g_{1,m}$ blocks created, namely $\{A_{1,m,q}\}_{q\in [g_{1,m}], m\in [d] }$, each containing $k_{v(1,m,q)}=|A_{1,m,q}|$ marks, behave independently from one another and as if under $\Q^{(k_{v(1,m,q)}),\bm \theta }_{T-t_1,m}$. 
It is important to note that some blocks $\{A_{1,m,q}\}_{q\in [g_{1,m}], m\in [d] }$ may carry only one mark.

Consider each of the subsequent splitting events at times $\tau_2,\ldots, \tau_n$. 
They correspond to the spine splitting events of some sub-block of the ones just created $\{A_{1,m,q}\}_{q\in [g_{1,m}], m\in [d] }$, breaking into smaller blocks.
Since we want to keep track of which block has splitted at time $\tau_h$, we reindex $(\tau_h)_{2\leq h\leq n}$ as follows
\[
\Big\{(\tau^{(m,q)}_{h})_{h\in [n^{(m,q)}]}, q\in [g_{1,m}],m\in [d]\Big\},
\]
where $n^{(m,q)}\leq k_{v(1,m,q)}-1$, denotes the number of splitting events involving a sub-block of $A_{1,m,q}$.

Observe that $\sum_{m\in [d]} \sum_{q\in[g_{1,m}]} n^{(m,q)}=n-1$, since there are only  $n-1$ remaining spine splitting events.
That is, if we follow the spines in block $A_{1,m,q}$ (and the subsequent quantities will depend on $(m,q)$), the time $\tau^{(m,q)}_{h}$ is the $h$-th time,  after $\tau_1$,  that $A_{1,m,q}$ breaks into smaller sub-blocks, for $1\leq h\leq  n^{(m,q)}$. Similarly, we 
denote its corresponding partition process and the process that codes the types (or colors) by $(\mathcal{P}^{(m,q)}_h)_{h\in  [n^{(m,q)}]}$ and $(C^{(m,q)}_h)_{h\in  [n^{(m,q)}]}$, respectively.
Following the same reasoning, we reindex the values $(t^{(m,q)})_{h\in [n^{(m,q)}]}$, $(\bm \ell^{(m,q)}_h)_{h\in [n^{(m,q)}]}$, $(\bo g^{(m,q)}_h)_{h\in [n^{(m,q)}]}$, and $(i^{(m,q)}_{h})_{h\in [n^{(m,q)}]}$.
We also consider also the partitions $(\bo P^{(m,q)}_h)_{h\in [n^{(m,q)}]\cup\{0\}}$ where $\bo P^{(m,q)}_h=( P^{(m,q)}_{h,m'})_{m'\in[d]}$.
Hence $P^{(m,q)}_{0, m'}=\{A_{1,m,q} \}$ since no splitting has occurred, whilst $P^{(m,q)}_{n^{(m,q)}}$ consists only on the singletons of $A_{1,m,q}$. 

\begin{center}
	\begin{figure}
		\includegraphics[width=.6\textwidth]{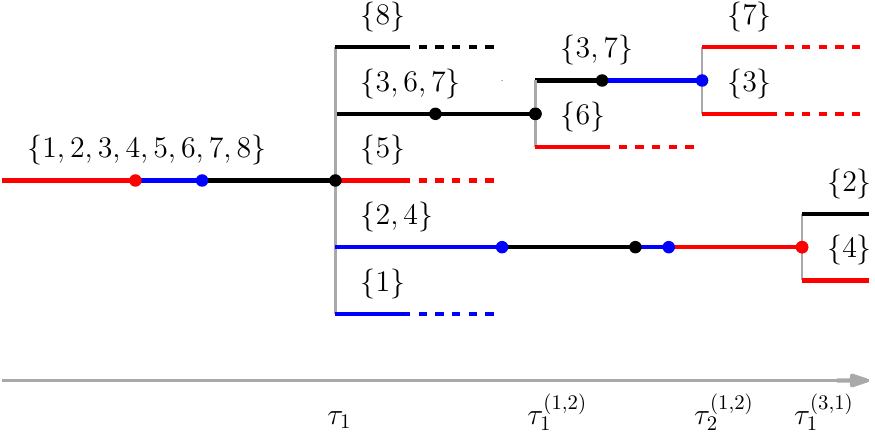}\caption{Illustration of the vertices carrying marks, in a 3 type tree, starting from 8 marks. The type 1 individuals are depicted in Black, type 2 in Red and type 3 in Blue. At the first spine splitting event, occurring at time $\tau_1-$, the vertex $v(1)$ type $i_1=1$, has $\bo L_{v(1)}$ children, and from them $\bo g_1=(2,1,2)$ carry at least one mark. The partition $\bo P_1$ consists of $P_{1,1}=\{A_{1,1,1},A_{1,1,2}\}$, $P_{1,2}=\{A_{1,2,1}\}$ and $P_{1,3}=\{A_{1,3,1},A_{1,3,2}\}$.  Only two blocks will split at future times, namely $A_{1,1,2}$ and $A_{1,3,1}$. At time $\tau^{(1,2)}_{1}-$ the individual $v(1,1,2)$, which has type $i^{(1,2)}_1=1$, undergoes a spine splitting event, having $\bm \ell^{(1,2)}_1$ offspring, and $\bo g^{(1,2)}_1=(1,1,0)$ carry at least one mark. The partition that it generates is $\bo P^{(1,2)}_1$ consisting of blocks $P^{(1,2)}_{1,1}=\{A^{(1,2)}_{1,1,1}\}$ and $P^{(1,2)}_{1,2}=\{A^{(1,2)}_{1,2,1}\}$. For block $A_{1,1,2}$ there are $n^{(1,2)}=2$ spine splitting events after $\tau_1$. The second one, occurs at time $\tau^{(1,2)}_2$, from an individual type $i^{(1,2)}_2=3$, giving rise to a partition $\bo P^{(1,2)}_2$ with $P^{(1,2)}_{2,2}=\{A^{(1,2)}_{2,2,1},A^{(1,2)}_{2,2,2}\}$.  }\label{figNotationInductionJointSplittingTimes}
	\end{figure}
\end{center}

Recall the definition  of $\Delta_T(n)$ in \eqref{deltatn} and set, for $t\in[0,T)$,
\[
\begin{split}
\Delta_{T-t}^{m,q}:=&\bigcap_{h\in [n^{(m,q)}]}\left\{\tau^{(m,q)}_h\in {\rm d}t^{(m,q)}_h-t,\mathcal{P}^{(m,q)}_h=\bo{P}^{(m,q)}_h,\bo{L}^{(m,q)}_h=\bm \ell^{(m,q)}_h,\right.
\\
&\hspace{6cm} \left .C^{(m,q)}_{h}=i^{(m,q)}_h,M^{(m,q)}=n^{(m,q)} \right\},
 \end{split}
 \]
 where $\bo{L}^{(m,q)}_h$ and $M^{(m,q)}$ denote the offspring of the individual involved in the $h$-th splitting event and  the number of splits associated to the block $A_{1,m,q}$. Since we are separating the first spine splitting  and the subsequent spine splitting events, it follows 
\[
\Delta_T(n)=\left\{\tau_1\in {\rm d}t_1,\mathcal{P}_{1}=\bo{P}_1,\bo{L}_{v(1)}=\bm{\ell}_1 , C_{1}=i_1 \right\}\bigcap_{m\in[d]; q\in [g_{1,m}]} \Delta_{T}^{m,q}.
\]
Thus from the Markov property (see Lemma \ref{lemmaMrkovPropertyUnderQ}), conditionally on the event 
\[
\left\{\tau_1\in {\rm d}t_1,\mathcal{P}_{1}=\bo{P}_1,\bo{L}_{v(1)}=\bm{\ell}_1 , C_{1}=i_1 \right\},
\]
the spines in each block $A_{1,m,q}$, behave as if under $\Q^{(k_{v(1,m,q)}),\bm \theta }_{T-t_1,m}$. 
Since $\tau_1$ is a splitting event, then $[k]$ is broken into smaller sub-blocks and hence $k_{v(1,m,q)}<k$ and so $n^{(m,q)}<n$ for any $m\in[d]$ and $q\in[g_{1,m}]$. When blocks are singletons, we observe that there is no subsequent splitting events implying that in this case the spine behaves as  if under $\Q^{(1),\bm \theta }_{T-t_1,m}$, for $m\in[d]$. Thus we may apply the induction hypotheses \emph{only} for the blocks that carry at least two marks, i.e. when $k_{v(1,m,q)}\geq 2$. Hence,
\[
\begin{split}
	&\Q^{(k),\bm \theta }_{T,r}\Big(\left.\Delta_T^{m,q}\right| \tau_1\in {\rm d}t_1,\mathcal{P}_{1}=\bo{P}_1,\bo{L}_{v(1)}=\bm \ell_1 , C_{1}=i_1\Big) =\Q^{(k_{v(1,m,q)}),\bm \theta }_{T-t_1,m} \left(\Delta_{T-t_1}^{m,q}\right)\\
	& = \prod_{h=1}^{n^{(m,q)}}\prod_{m'=1}^{d}\E_{m'}\left[e^{-\bm \theta \cdot \bo Z(T-t^{(m,q)}_h)}\right]^{\ell^{(m,q)}_{h,m'}-g^{(m,q)}_{h, m'}}p_{i^{(m,q)}_h}(\bm \ell^{(m,q)}_h)\left(\bm \ell^{(m,q)}_h\right)^{\floor{\bo g^{(m,q)}_h}}\\
	& \quad \times \prod_{h=0}^{n^{(m,q)}-1}
	\prod_{m'=1}^{d}\prod_{\substack{q'\leq g^{(m,q)}_{h, m'} \\ k^{(m, q)}_{v(h,m',q')}\geq 2}}\E_{m'}\left[Z^{(i^{(m,q)}_{v(h,m',q')})}\left(t^{(m,q)}_{v(h,m',q')}-t^{(m,q)}_{h}\right)\right. \\
	& \qquad \hspace{3cm}  \times \left.\prod_{j=1}^{d}\E_{j}\left[e^{-\bm \theta\cdot \bo Z(T-t^{(m,q)}_{v(h,m',q')})}\right]^{ Z^{(j)}\big(t^{(m,q)}_{v(h,m',q')}-t^{(m,q)}_{h}\big)-\delta\left(i^{(m,q)}_{v(h,m',q')}, j\right)}\right]\\
	&  \quad   \times
	\frac{\prod_{h=1}^{n^{(m,q)}} \prod_{m'=1}^d\E_{m'}\left[N(T-t^{(m,q)}_{h})e^{-\bm \theta \cdot \bo Z(T-t^{(m,q)}_{h})}\right]
		^{\#\{ q'\leq g^{(m,q)}_{h, m'}:k^{(m,q)}_{v(h,m',q')}=1\}}
	}{\E_{m}\left[N_{T-t_1}^{\floor{k_{v(1,m,q)}}}e^{-\bm \theta\cdot \bo Z_{T-t_1}}\right]} \prod_{h=1}^{n^{(m,q)}}\alpha_{i^{(m,q)}_{h}} {\rm d}t^{(m,q)}_h,
\end{split}
\]
where $k^{(m,q)}_{v(h,m',q')}$ denotes the number of marks following the $q'$-th individual of type $m'$, occurring in the $h$-th splitting event in block $A_{1,m,q}$. Such individual splits at  time $t^{(m,q)}_{v(h,m',q')}$ and has type $i^{(m,q)}_{v(h,m',q')}$.  We also observe, for the sake of simplicity, that we have used the notation $\bo Z(\cdot), Z(\cdot), N(\cdot)$ and $\delta(\cdot, \cdot)$ instead of $\bo Z_{\cdot}, Z_{\cdot}, N_{\cdot}$ and $\delta_{\cdot, \cdot}$, respectively.

By the Markov property, which implies the conditional independence of the splitting events $(\Delta_T^{m,q})_{q\in[g_{1,m}],m\in[d]}$, we have
\begin{equation}\label{eqnJointLawOfSplittingEventsAfterTheFirstOne}
	\begin{split}
	\Q^{(k),\bm \theta }_{T,r}&\paren{\left.\bigcap_{m\in [d];q\in [g_{1, m}]} \Delta_T^{m,q}\mathbf{1}_{\{k_{v(1,m,q)}\geq 2\}}\right| \tau_1\in {\rm d}t_1,\mathcal{P}_{1}=\bo{P}_1,\bo{L}_{v(1)}=\bm{\ell}_1 , C_{1}=i_1}\\
	& =\prod_{m=1}^d\prod_{q=1 }^{g_{1,m}}\Q^{(k_{v(1,m,q)}),\bm \theta }_{T-t_1,m} \left(
	\Delta_{T-t_1}^{m,q}\right)
\\
& = \prod_{m=1}^d\prod_{q=1 }^{g_{1,m}}
\prod_{h=1}^{n^{(m,q)}}\prod_{m'=1}^{d}\E_{m'}\left[e^{-\bm \theta \cdot \bo Z(T-t^{(m,q)}_h)}\right]^{\ell^{(m,q)}_{h,m'}-g^{(m,q)}_{h, m'}}p_{i^{(m,q)}_h}(\bm \ell^{(m,q)}_h)\left(\bm \ell^{(m,q)}_h\right)^{\floor{\bo g^{(m,q)}_h}}\\
& \quad \times \prod_{m=1}^d\prod_{q=1 }^{g_{1, m}}
\prod_{h=0}^{n^{(m,q)}-1}
	\prod_{m'=1}^{d}\prod_{\substack{q'\leq g^{(m,q)}_{h, m'} \\ k^{(m, q)}_{v(h,m',q')}\geq 2}}\E_{m'}\left[Z^{(i^{(m,q)}_{v(h,m',q')})}\left(t^{(m,q)}_{v(h,m',q')}-t^{(m,q)}_{h}\right)\right. \\
& \qquad    
\hspace{2cm}\times \left.\prod_{j=1}^{d}\E_{j}\left[e^{-\bm \theta\cdot \bo Z(T-t^{(m,q)}_{v(h,m',q')})}\right]^{ Z^{(j)}\big(t^{(m,q)}_{v(h,m',q')}-t^{(m,q)}_{h}\big)-\delta\left(i^{(m,q)}_{v(h,m',q')}, j\right)}\right]\\
&  \quad   \times\prod_{m=1}^d\prod_{\substack{q\leq g_{1,m}\\ k_{v(1,m,q)}\geq 2}}
\frac{\prod_{h=1}^{n^{(m,q)}} \prod_{m'=1}^d\E_{m'}\left[N(T-t^{(m,q)}_{h})e^{-\bm \theta \cdot \bo Z(T-t^{(m,q)}_{h})}\right]
		^{\#\{ q'\leq g^{(m,q)}_{h, m'}:k^{(m,q)}_{v(h,m',q')}=1\}}
	}{\E_{m}\left[N_{T-t_1}^{\floor{k_{v(1,m,q)}}}e^{-\bm \theta\cdot \bo Z_{T-t_1}}\right]}\\
	& \qquad\hspace{8cm}\times \prod_{m=1}^d\prod_{q=1 }^{g_{1,m}}
	\prod_{h=1}^{n^{(m,q)}}\alpha_{i^{(m,q)}_{h}} {\rm d}t^{(m,q)}_h.
\end{split}
\end{equation}Note that the previous term contains all the splitting events at times $(\tau_h)_{2\leq h\leq n}$, thus
the first line on the right-hand side, the branching rates and the differentials can be rewritten as follows
\begin{align*}	
&\prod_{m=1}^d\prod_{q=1 }^{g_{1,m}}
\prod_{h=1}^{n^{(m,q)}}\prod_{m'=1}^{d}\E_{m'}\left[e^{-\bm \theta \cdot \bo Z(T-t^{(m,q)}_h)}\right]^{\ell^{(m,q)}_{h,m'}-g^{(m,q)}_{h, m'}}p_{i^{(m,q)}_h}(\bm \ell^{(m,q)}_h)\left(\bm \ell^{(m,q)}_h\right)^{\floor{\bo g^{(m,q)}_h}}\alpha_{i^{(m,q)}_{h}} {\rm d}t^{(m,q)}_h\\
&\hspace{5cm}=\prod_{h=2}^n\prod_{m'=1}^d\E_{m'}\left[e^{-\bm \theta \cdot \bo Z_{T-t_h}}\right]^{ \ell_{h, m'} - g_{h, m'}}p_{i_h}(\bm \ell_h)\bm \ell_h^{\lfloor\bo g_h\rfloor}\alpha_{i_{h}} {\rm d}t_h. 
\end{align*}
Similarly, for the ratio on the last line of \eqref{eqnJointLawOfSplittingEventsAfterTheFirstOne} we have 
\begin{align*}
& \prod_{m=1}^d\prod_{\substack{q\leq g_{1,m}\\ k_{v(1,m,q)}\geq 2}}
\frac{\prod_{h=1}^{n^{(m,q)}} \prod_{m'=1}^d\E_{m'}\left[N(T-t^{(m,q)}_{h})e^{-\bm \theta \cdot \bo Z(T-t^{(m,q)}_{h})}\right]
		^{\#\{ q'\leq g^{(m,q)}_{h, m'}:k^{(m,q)}_{v(h,m',q')}=1\}}
	}{\E_{m}\left[N_{T-t_1}^{\floor{k_{v(1,m,q)}}}e^{-\bm \theta\cdot \bo Z_{T-t_1}}\right]}\\
&\hspace{5cm} = \frac{\prod_{h=2}^n \prod_{m=1}^d\E_m\left[N_{T-t_{h}}e^{-\bm \theta \cdot \bo Z_{T-t_{h}}}\right]
	^{\#\{ q\leq g_{h, m}\ :\ k_{v(h,m,q)}=1\}}
}{\prod_{m=1}^d\prod_{\substack{q\leq g_{1,m}\\ k_{v(1,m,q)}\geq 2}} \E_{m}\left[N_{T-t_1}^{\floor{k_{v(1,m,q)}}}e^{-\bm \theta\cdot \bo Z_{T-t_1}}\right]}.
\end{align*}Again, the numerator of the above equation appears since all of the spine splitting times $(\tau_h)_{2\leq h\leq n}$ are being considered. Again,  for the remaining term in \eqref{eqnJointLawOfSplittingEventsAfterTheFirstOne}, one can show that
\begin{align*}
	&  \prod_{m=1}^d\prod_{q=1 }^{g_{1, m}}
\prod_{h=0}^{n^{(m,q)}-1}
	\prod_{m'=1}^{d}\prod_{\substack{q'\leq g^{(m,q)}_{h, m'}: \\ k^{(m, q)}_{v(h,m',q')}\geq 2}}\E_{m'}\left[Z^{(i^{(m,q)}_{v(h,m',q')})}\left(t^{(m,q)}_{v(h,m',q')}-t^{(m,q)}_{h}\right)\right. \\
& \qquad    
\hspace{3cm}\times \left.\prod_{j=1}^{d}\E_{j}\big[e^{-\bm \theta\cdot \bo Z(T-t^{(m,q)}_{v(h,m',q')})}\big]^{ Z^{(j)}\big(t^{(m,q)}_{v(h,m',q')}-t^{(m,q)}_{h}\big)-\delta\left(i^{(m,q)}_{v(h,m',q')}, j\right)}\right]\\
	& \hspace{1cm}= \prod_{h=1}^{n-1}\prod_{m=1}^d\prod_{\substack{ q\leq g_{h,m}:\\ k_{v(h,m,q)}\geq 2}}\E_{m}\left[Z^{(c(v(h,m,q)))}_{t_{v(h,m,q)}-t_h} \prod_{m'=1}^d\E_{m'}\left[e^{-\bm\theta\cdot \bo Z_{T-t_{v(h,m,q)}}}\right]^{ Z^{(m')}_{t_{v(h,m,q)}-t_h}- \delta_{c(v(h,m,q)), m'}}\right].
\end{align*}Joining the previous computations we obtain
\begin{equation}\label{eqnJointLawOfSplittingEventsAfterTheFirstOneV2}
	\begin{split}
	\Q^{(k),\bm \theta }_{T,r}&\paren{\left.\bigcap_{m\in [d];q\in [g_{1, m}]} \Delta_T^{m,q}\mathbf{1}_{\{k_{v(1,m,q)}\geq 2\}}\right| \tau_1\in {\rm d}t_1,\mathcal{P}_{1}=\bo{P}_1,\bo{L}_{v(1)}=\bm{\ell}_1 , C_{1}=i_1}\\
	& =\prod_{h=2}^n\prod_{m'=1}^d\E_{m'}\left[e^{-\bm \theta \cdot \bo Z_{T-t_h}}\right]^{ \ell_{h, m'} - g_{h, m'}}p_{i_h}(\bm \ell_h)\bm \ell_h^{\lfloor\bo g_h\rfloor}\alpha_{i_{h}} {\rm d}t_h\\
	& \quad \times  \prod_{h=1}^{n-1}\prod_{m=1}^d\prod_{\substack{ q\leq g_{h,m}:\\ k_{v(h,m,q)}\geq 2}}\E_{m}\left[Z^{(c(v(h,m,q)))}_{t_{v(h,m,q)}-t_h} \prod_{m'=1}^d\E_{m'}\left[e^{-\bm\theta\cdot \bo Z_{T-t_{v(h,m,q)}}}\right]^{ Z^{(m')}_{t_{v(h,m,q)}-t_h}- \delta_{c(v(h,m,q)), m'}}\right]\\
	& \hspace{4cm}\times\frac{\prod_{h=2}^n \prod_{m=1}^d\E_m\left[N_{T-t_{h}}e^{-\bm \theta \cdot \bo Z_{T-t_{h}}}\right]
	^{\#\{ q\leq g_{h, m}\ :\ k_{v(h,m,q)}=1\}}
}{\prod_{m=1}^d\prod_{\substack{q\leq g_{1,m}\\ k_{v(1,m,q)}\geq 2}} \E_{m}\left[N_{T-t_1}^{\floor{k_{v(1,m,q)}}}e^{-\bm \theta\cdot \bo Z_{T-t_1}}\right]} .
\end{split}
\end{equation}Finally, using Proposition \ref{propoFirstSplittingEventWithPartitionChildAndType}, we get \begin{equation}\label{eqnFirstSplittingTimeUsedForTheJointSplittingTimes}
	\begin{split}
	\Q^{(k),\bm \theta }_{T,r}&\paren{\tau_1\in {\rm d}t_1,\mathcal{P}_{1}=\bo{P}_1,\bo{L}_{v(1)}=\bm{\ell}_1 , C_{1}=i_1}
	 = \prod_{m=1}^d\E_{m}\left[e^{-\bm \theta \cdot \bo Z_{T-t_1}}\right]^{\ell_{1, m}- g_{1, m}}p_{i_1}(\bm \ell_1)\bm \ell_1^{\floor{\bo g_1}}\\
	& \times 
	\E_{r}\left[Z^{(i_1)}_{t_1}\prod_{m=1}^d\E_{m}\paren{e^{-\bm \theta\cdot \bo Z_{T-t_1}}}^{ Z^{(m)}_{t_1}-\delta_{i_1,m}}\right]
	\frac{\prod_{m=1}^d\prod_{q=1}^{g_{1, m}}\E_{m}\left[N_{T-t_1}^{\floor{ k_{v(1,m,q)}}}e^{-\bm \theta \cdot \bo Z_{T-t_1}}\right]}{\E_{r}\left[N_{T}^{\floor{k}}e^{-\bm \theta\cdot \bo Z_{T}}\right]}\alpha_{i_1} {\rm d}t_1.
	\end{split}
\end{equation} 
Consider the numerator of the quotient on the right-hand side of the previous identity. We distinguish between vertices with a single mark and those with at least two marks, as follows
\begin{align*}
	\prod_{m=1}^d\prod_{q=1}^{g_{1,m}}\E_{m}\left[N_{T-t_1}^{\floor{ k_{v(1,m,q)}}}e^{-\bm \theta \cdot \bo Z_{T-t_1}}\right] &
	 = \prod_{m=1}^d\E_{m}\left[N_{T-t_1}e^{-\bm \theta \cdot \bo Z_{T-t_1}}\right]	^{\#\{ q\leq g_{1,m}:k_{v(1,m,q)}=1\}}\\
	& \hspace{3cm} \times \prod_{m=1}^d\prod_{\substack{ q\leq g_{1,m}:\\ k_{v(1,m,q)}\geq 2}}\E_{m}\left[N_{T-t_1}^{\floor{ k_{v(1,m,q)}}}e^{-\bm \theta \cdot \bo Z_{T-t_1}}\right] . 
\end{align*}Note that the last term in the identity on the right-hand side  above cancels with the denominator of the quotient on the right-hand side of the identity \eqref{eqnJointLawOfSplittingEventsAfterTheFirstOneV2}. 
Moreover, the first term in the identity on the right-hand side  above completes the product for $h=1$ in the numerator of the quotient on the right-hand side of the identity \eqref{eqnJointLawOfSplittingEventsAfterTheFirstOneV2}. 
Therefore, multiplying \eqref{eqnJointLawOfSplittingEventsAfterTheFirstOneV2} and \eqref{eqnFirstSplittingTimeUsedForTheJointSplittingTimes}, we deduce the desired result.
\end{proof}
\subsection{Subpopulations sizes} We now analyze the sizes of subpopulations coming off the spines under $\Q^{(k),\bm \theta }_{T,r}$. Here, we see a significant difference from the single-type case.

Suppose a non-spine particle $v$ is alive at time $t$ and recall that  $\bo{Z}^{[v]}_{T}$ denotes the vector of number of individuals alive at time $T$ whose ancestor is $v$. Then by the Markov branching property (see Lemma \ref{lemmaMrkovPropertyUnderQ}) and identity \eqref{eqnChangeOfMeasureMtypeBGWDiscounted}, we obtain
\[
	\begin{split}
		& \Q^{(k),\bm \theta }_{T,r}\left[e^{-{\bm \mu}\cdot \bo{Z}^{[v]}_{T}}\bigg| \mathcal{F}_t\right]  =\Q^{(0),\bm \theta }_{T-t,c(v)}\left[e^{-{\bm \mu}\cdot \bo{Z}_{T-t}}\right]=  \E^{\bm \theta}_{T-t,c(v)}\left[e^{-{\bm \mu}\cdot \bo Z_{T-t}}\right] =\frac{\E_{c(v)}\left[e^{-(\bm \theta+\bm \mu)\cdot \bo Z_{T-t}}\right]}{\E_{c(v)}\left[e^{-\bm \theta\cdot \bo Z_{T-t}}\right]}.
	\end{split}
\]
Similarly, if a particle $v$ alive at time $t$ is carrying $j$ spines then, again  by the Markov branching property  and identity \eqref{defMeasureQkTrGivenTopologicalInfo}, we have
\[
	\begin{split}
		& \Q^{(k),\bm \theta }_{T,r}\left[e^{-{\bm \mu}\cdot \bo{Z}^{[v]}_{T}}\bigg| \mathcal{F}^{(k)}_t\right]  =\Q^{(j),\bm \theta }_{T-t,c(v)}\left[e^{-{\bm \mu}\cdot \bo{Z}_{T-t}}\right] =\frac{\E_{c(v)}\left[N_{T-t}^{\floor{j}}e^{-(\bm \theta+\bm \mu)\cdot \bo Z_{T-t}}\right]}{\E_{c(v)}\left[N_{T-t}^{\floor{j}}e^{-\bm \theta\cdot \bo Z_{T-t}}\right]}.
	\end{split}
\]
In particular, the Laplace transform for the number of descendants at time $T$ of births coming off a single spine branch $(k=1)$ started at time $t$ (plus the spine itself) is given by 
\[
\frac{\E_{c(v)}\left[N_{T-t}e^{-(\bm \theta+\bm \mu)\cdot \bo Z_{T-t}}\right]}{\E_{c(v)}\left[N_{T-t}e^{-\bm \theta\cdot \bo Z_{T-t}}\right]}.
\]

Similarly to the single-type case, the rate of births off any spine particle and the corresponding offspring distribution are independent (see Lemma 3.5 in \cite{MR4718398}). However, in the single-type case, this rate does not depend on the number of spines following it. In our setting,  the rate of births off any spine particle strongly depends on both the number of spines following it and its type (see Lemma \ref{lemmaGeneralRateWhenABranchingEventOccursKnowingInfo} and its remark below). More precisely, in the single-type case the law of the number of descendant coming off any spine branch is determined by 
\[
\frac{\E\left[Z_{T-t}e^{-( \theta+ \mu)  Z_{T-t}}\right]}{\E\left[Z_{T-t}e^{- \theta  Z_{T-t}}\right]}.
\]
In our setting, if the spine branch has $h$-marks and the spine changes from type $r$ to type $j$, such law is determined by 
\[
\frac{\E_{j}\left[N^{\floor{h}}_{T-t}e^{-(\bm \theta+\bm \mu)\cdot \bo Z_{T-t}}\right]}{\E_{j}\left[N^{\floor{h}}_{T-t}e^{-\bm \theta\cdot \bo Z_{T-t}}\right]}.
\]
This implies that the lineage decomposition of the ancestral tree cannot be carried out as easily as in the single-type case (see Subsection 3.3 in  \cite{MR4718398})), since it requires knowing both the marks in any given lineage and the type following that lineage.

\subsection{Proof of Theorem \ref{teofsplitting1}}\label{subsectionJointLawOfSplittingTimesUnderPUnif}
Now, let us compute the functionals of our interest under  $\p^{(k)}_{unif,T,r}$. In particular we are interested in the law of  the splitting times under such probability. In order to do so, we recall that in Lemma \ref{lemmaFromTheMeasureQToTheMeasureEunif} we obtained the term 
\[
\Q^{(k),\bm \theta}_{T,r}\left[\frac{\mathbf{1}_{A}}{N^{\floor{k}}_{T}e^{-\bm \theta\cdot \bo Z_{T}}}\right],\qquad \qquad A\in \mathcal{F}^{(k)}_T.
\]
To simplify this term, we apply  the Beta integral formula
\begin{equation}\label{gamma+beta}
 \frac{\Gamma(k)}{x^{\floor{k}}}=\int_0^\infty (e^y-1)^{k-1}e^{-yx}{\rm d}y,
\end{equation}
together with Fubini's theorem and the identity $N_T=\vec{1}\cdot \bo Z_{T}$ to obtain
\begin{equation}
\begin{split}\label{eqnfsplitting}
\Q^{(k),\bm \theta}_{T,r}&\left[\frac{1}{N_{T}^{\floor{k}}e^{-\bm \theta\cdot \bo Z_{T}}}\mathbf{1}_{A}\right]= \frac{1}{(k-1)!}\int_0^\infty(e^\phi-1)^{k-1}\Q^{(k),\bm \theta }_{T,r}\left[\mathbf{1}_{A}e^{(\bm \theta-\phi\vec{1})\cdot \bo Z_{T}}\right]{\rm d}\phi.
\end{split}
\end{equation} 
Finally, we use the change of measure \eqref{defMeasureQkTrGivenAllInfo} and observe, for any $\bm \mu \in \mathbb{R}^d_+$, that 
\[
		 \Q^{(k),\bm \theta}_{T,r}\left[\mathbf{1}_{A}e^{(\bm \theta-\bm \mu)\cdot \bo Z_{T}}\right]  = \Q^{(k),\bm \mu}_{T,r}(A)\frac{\E_{r}\left[N^{\floor{k}}_Te^{-\bm \mu\cdot \bo Z_T}\right]}{\E_{r}\left[N^{\floor{k}}_Te^{-\bm \theta\cdot \bo Z_T}\right]}.
		 \]
Since each term on the right-hand side of the above identity are well-defined, it follows that \eqref{eqnfsplitting}  is well-defined as well.
Thus putting all pieces together, we obtain the following key result.

\begin{propo}\label{propofsplitting1}
For any $k\geq 1$, $T\in \re_+$,  $r\in [d]$ and $A\in \F^{(k)}_T$, we have
\begin{equation}\label{eqnfsplitting2}
	\begin{split}
\mathbb{P}^{(k)}_{unif,T,r}\paren{A}&= \frac{1}{(k-1)!}\frac{1}{\mathbb{P}_{r}\paren{\ N_T\geq k}}\int_0^\infty(e^\phi-1)^{k-1}\Q^{(k),\phi\vec{1}}_{T,r}(A)\E_{r}\left[N^{\floor{k}}_Te^{-\phi\vec{1}\cdot \bo Z_T}\right]{\rm d}\phi.
\end{split}
\end{equation}
\end{propo}
We note that for the event $A=\Delta_T(n)$ defined in \eqref{deltatn}, the formula derived above, combined with Proposition \ref{propoJointSplittingEventWithPartitionChildAndType} yields the proof of Theorem \ref{teofsplitting1}. In particular, it provides   an explicit representation of the distribution of the genealogy of a uniform sample of size $k$ taken at time $T$.


\section*{Appendix}
\label{Appendix}

{\bf Proof of Equation \eqref{eqmultytipeSam}.} Let $\bo x=(x_1, \cdots, x_d)\in \mathbb{Z}^d_+$ and  for $u\in \mathbb{R}$, we introduce the vector valued function $\theta(u)=(\theta_1(u), \ldots, \theta_d(u))$. From the branching property,  we see
\[
\E_{\bo x}\left[e^{-\bm \theta(u)\cdot \bo Z_{T}}\right]=\prod_{m=1}^{d}\E_{m}\left[e^{-\bm \theta (u)\cdot \bo Z_{T}}\right]^{x_m}. 
\]Thus, the Markov property implies
\[
F_{T,r}(e^{-\bm \theta})=\E_r\left[\prod_m\E_m\left[e^{-\bm \theta\cdot \bo Z_{T-t}}\right]^{Z^{(m)}_{t}}\right]=F_{t,r}(\Vec{F}_{T-t}(e^{-\bm \theta})),
\]
where $\Vec{F}_{T-t}$ is defined in \eqref{eqnDerivativeProbGenFn}. On the other hand, it is clear that for $\bo v=(v_1, \ldots,v_d )$, we have
\begin{align*}
&\E_r\left[Z^{(i)}_{t}\prod_{m=1}^d v_{m}^{Z^{(m)}_{t}-\delta_{m,i}}\right]=\frac{\partial}{\partial v_i}\E_r\left[\prod_{m=1}^d v_m^{Z^{(m)}_{t}}\right].
\end{align*}
Therefore
\begin{align*}
	 \frac{\rm d}{\rm du}F_{t,r}(\Vec{F}_{T-t,\cdot}(e^{-\bm \theta(u)}))&=\frac{\rm d}{\rm d u}\E_r\left[\prod_{m=1}^d\E_m\left[e^{-\bm \theta(u)\cdot \bo Z_{T-t}}\right]^{Z^{(m)}_{t}}\right]\\
	&=\sum_{j=1}^d\frac{\partial}{\partial v_j}\E_r\left[\prod_{m=1}^d v_m^{Z^{(m)}_{t}}\right]\Bigg|_{v_\ell=\E_\ell\left[e^{-\bm \theta(u)\cdot \bo Z_{T-t}}\right], \ell\in [d]} \frac{\rm d}{\rm d u}\E_j\left[e^{-\bm \theta(u)\cdot \bo Z_{T-t}}\right]\\
	&=\sum_{j=1}^d\E_r\left[Z^{(j)}_{t}\prod_{m=1}^d\E_m\left[e^{-\bm \theta(u)\cdot \bo Z_{T-t}}\right]^{Z^{(m)}_{t}-\delta_{m,j}}\right]\frac{\rm d}{\rm du}F_{T-t,j}(e^{-\bm \theta (u)}).
\end{align*}Hence
\begin{align*}
&\E_r\left[Z^{(i)}_{t}\prod_m\E_m\left[e^{-\bm \theta\cdot \bo Z_{T-t}}\right]^{Z^{(m)}_{t}-\delta_{m,i}}\right]\frac{\rm d}{\rm d u}F_{T-t,i}(e^{-\bm \theta (u)})\\
&=\frac{\rm d}{\rm du}F_{t,r}(\Vec{F}_{T-t,\cdot}(e^{-\bm \theta(u)}))-\sum_{j\neq i}\E_r\left[Z^{(j)}_{t}\prod_{m=1}^d\E_m\left[e^{-\bm \theta(u)\cdot \bo Z_{T-t}}\right]^{Z^{(m)}_{t}-\delta_{m,j}}\right] \frac{\rm d}{\rm du}F_{T-t,j}(e^{-\bm \theta (u)}).\\
\end{align*}

{\bf Proof of Equation \eqref{eqnCheckingJointLawConcidesWithUnitype}. } Recall that we are in the single-type case and that $\theta\ge 0$. From the semigroup identity $F_T(s)=F_t(F_{T-t}(s))$, we can write
	\begin{align*}
		& \E\left[Z_{t_{v(h,\ell)}-t_{h}} \E\left[e^{-\theta Z_{T-t_{v(h,\ell)}}}\right]^{Z_{t_{v(h,\ell)}-t_{h}}-1}\right]=F'_{t_{v(h,\ell)}-t_{h}}\paren{F_{T-t_{v(h,\ell)}}(e^{-\theta})}=\frac{F'_{T-t_{h}}(e^{-\theta})}{F'_{T-t_{v(h,\ell)}}(e^{-\theta})},
	\end{align*}which implies
	\begin{align*}
		& \prod_{h=0}^{n-1}\prod_{\substack{\ell\leq g_h\\ k_{v(h,\ell)}\geq 2}}
		\E\left[Z_{t_{v(h,\ell)}-t_{h}} \E\left[e^{-\theta\cdot Z_{T-t_{v(h,\ell)}}}\right]^{Z_{t_{v(h,\ell)}-t_{h}}-1}\right]= \prod_{h=0}^{n-1}\frac{\big(F'_{T-t_{h}}(e^{-\theta})\big)^{\#\{\ell\leq g_h:k_{v(h,\ell)}\geq 2 \}}}{\prod_{\substack{\ell\leq g_h\\ k_{v(h,\ell)}\geq 2}}F'_{T-t_{v(h,\ell)}}(e^{-\theta})}.
	\end{align*}Let us analyse the denominator of the right-hand side of the previous identity. Since $t_{v(h,\ell)}$ is a spine splitting time, this means that $v(h,\ell)$ carries at least two marks. Also, note that $\{ t_{v(h,\ell)}, h\in [n-1]\cup\{0\},\ell\in [g_h]\}=\{t_h,h\in [n] \}$ since each spine splitting time appears once. 
	The latter implies  
	\[
	\prod_{h=0}^{n-1}\prod_{\substack{\ell\leq g_h\\ k_{v(h,\ell)}\geq 2}}F'_{T-t_{v(h,\ell)}}(e^{-\theta})=\prod_{h=1}^nF'_{T-t_{h}}(e^{-\theta}). 
	\]
	Putting all pieces together and observing that 
$	\E[Z_{T-t_{h}}e^{-\theta  Z_{T-t_{h}}}]=F'_{T-t_h}(e^{-\theta})$, we get
	\begin{align*}
		 \prod_{h=0}^{n-1}\prod_{\substack{\ell\leq g_h\\ k_{v(h,\ell)}\geq 2}}
		&\E\left[Z_{t_{v(h,\ell)}-t_{h}} \E\left[e^{-\theta Z_{T-t_{v(h,\ell)}}}\right]^{Z_{t_{v(h,\ell)}-t_{h}}-1}\right]\left(  \prod_{h=1}^n \E\left[Z_{T-t_{h}}e^{-\theta  Z_{T-t_{h}}}\right]
		^{\#\{ \ell\leq g_{h}:k_{v(h,\ell)}=1\}}\right)\\
		& =F'_{T}(e^{-\theta}) \prod_{h=1}^{n}\big(F'_{T-t_{h}}(e^{-\theta})\big)^{\#\{\ell\leq g_h:k_{v(h,\ell)}\geq 2 \}-1} \left(\prod_{h=1}^n \big(F'_{T-t_{h}}(e^{-\theta})\big)
		^{\#\{ \ell\leq g_{h}:k_{v(h,\ell)}=1\}}\right)\\
		& =F'_{T}(e^{-\theta})\prod_{h=1}^n \big(F'_{T-t_{h}}(e^{-\theta})\big)
		^{g_h-1}, 
	\end{align*}as expected.

%
%

\begin{funding}
 The first author was partially supported by the Deutsche Forschungsgemeinschaft (through grant DFG-SPP-2265). 
The second author acknowledges the support of the New Zealand Aotearoa Royal Society Te Ap\={a}rangi
Marsden Fund (22-UOA-052).

 The third author was supported by the grant CF-2023-I-2566 from SECIHTI, Mexico.
\end{funding}

\providecommand{\bysame}{\leavevmode\hbox to3em{\hrulefill}\thinspace}
\providecommand{\MR}{\relax\ifhmode\unskip\space\fi MR }
\providecommand{\MRhref}[2]{%
  \href{http://www.ams.org/mathscinet-getitem?mr=#1}{#2}
}
\providecommand{\href}[2]{#2}

\end{document}